\documentclass[11pt]{article}
    
\usepackage{amsmath,amsfonts,amsthm,amscd,amssymb,graphicx}

\numberwithin{equation}{section}

\usepackage{hyperref}

\newtheorem{theorem}{Theorem}[section]

\newtheorem{lemma}[theorem]{Lemma}

\newtheorem{proposition}[theorem]{Proposition}
\newtheorem{prop}[theorem]{Proposition}

\newtheorem{assumption}[theorem]{Assumption}

\newtheorem{remark}[theorem]{Remark}

\def\eps{\varepsilon }
\def\D{\partial }

\newcommand{\RR}{\mathbb{R}}
\newcommand{\cO}{\mathcal{O}}

\newcommand{\CC}{\mathbb{C}}

\newcommand{\cA}{\mathcal{A}}

\newcommand{\e}{{\varepsilon}}
\newcommand{\Range}{{\rm range}  \,}

\def\beq{\begin{equation}}
\def\eeq{\end{equation}}
\def\bb1{{1\!\!1}}
\def\R{\mbox{Re }}
\def\I{\mbox{Im }}

\def\dz{\partial_z}
\def\dx{\partial_x}

\def\cA{\mathcal{A}}

\def\eps{\varepsilon}

\def\OS{\mathrm{OS}}
\def\Ray{\mathrm{Ray}}
\def\Airy{\mathrm{Airy}}
\def\Iter{\mathrm{Iter}}

\def\bl{\mathrm{bl}}

\begin{document}

\title{Green function for linearized Navier-Stokes around a boundary layer profile: near critical layers} 

\author{Emmanuel Grenier\footnotemark[1]
  \and Toan T. Nguyen\footnotemark[2]
}

\maketitle

\renewcommand{\thefootnote}{\fnsymbol{footnote}}

\footnotetext[1]{Equipe Projet Inria NUMED,
 INRIA Rh\^one Alpes, Unit\'e de Math\'ematiques Pures et Appliqu\'ees., 
 UMR 5669, CNRS et \'Ecole Normale Sup\'erieure de Lyon,
               46, all\'ee d'Italie, 69364 Lyon Cedex 07, France. Email: Emmanuel.Grenier@ens-lyon.fr}

\footnotetext[2]{Department of Mathematics, Penn State University, State College, PA 16803. Email: nguyen@math.psu.edu. }



\begin{abstract}
This is a continuation and completion of the program (initiated in \cite{GrN1,GrN2}) to derive pointwise estimates on the Green function and sharp bounds on the semigroup of linearized Navier-Stokes around a generic stationary boundary layer profile. This is done via a spectral analysis approach and a careful study of the Orr-Sommerfeld equations, or equivalently the Navier-Stokes resolvent operator $(\lambda - L)^{-1}$. The earlier work (\cite{GrN1,GrN2}) treats the Orr-Sommerfeld equations away from critical layers: this is the case when the phase velocity is away from the range of the background profile or when $\lambda$ is away from the Euler continuous spectrum. In this paper, we study the critical case: the Orr-Sommerfeld equations near critical layers, providing pointwise estimates on the Green function as well as carefully studying the Dunford's contour integral near the critical layers.

As an application, we obtain pointwise estimates on the Green function and sharp bounds on the semigroup of the linearized Navier-Stokes problem near monotonic boundary layers that are spectrally stable to the Euler equations, complementing \cite{GrN1,GrN2} where unstable profiles are considered.


\end{abstract}

\newpage
\tableofcontents


\section{Introduction}

In continuation of \cite{GrN1,GrN2}, we are interested in the study of linearized Navier Stokes around a given fixed profile $U_s = (U(z),0)^{tr}$ in the inviscid limit; namely, we consider the following set of equations
\beq \label{Nlin1}
\partial_t v + U_s \cdot \nabla v + v\cdot  \nabla U_s + \nabla p - \sqrt \nu \Delta v = 0,
\eeq
\beq \label{Nlin2}
\nabla\cdot  v = 0,
\eeq
with $0<\nu \ll1$, posed on the half plane $x \in \RR$, $z \ge 0$,
with the no-slip boundary conditions
\beq \label{Nlin3}
v = 0 \quad \hbox{on} \quad z = 0 .
\eeq

Throughout this paper, $U(z)$ is strictly monotonic, real analytic, $U(0) =0$, 
$$U_+ = \lim_{z\to \infty} U(z) < \infty,$$
and the convergence is exponentially fast. We  are interested in the case when $U(z)$ is spectrally stable for the linearized Euler equations, that is, equations \eqref{Nlin1}-\eqref{Nlin2} when $\nu =0$. For instance, all the profiles that do not have an inflection point are stable, thanks to the classical Rayleigh's stability condition. In this case, strikingly, the viscosity has a destabilizing effect. That is, {\em all shear profiles are unstable for large Reynolds numbers.} There are lower and upper marginal stability branches $\alpha_\mathrm{low}(\nu)\sim \nu^{1/8}$ and  $ \alpha_\mathrm{up}(\nu)\sim \nu^{1/12}$, for generic stable shear flows, so that whenever the horizontal wave number $\alpha$ of perturbation belongs to $[\alpha_\mathrm{low}(\nu),\alpha_\mathrm{up}(\nu)]$, the linearized Navier-Stokes equations about the shear profile have a growing eigenfunction and an unstable eigenvalue $\lambda_\nu$ with 
\begin{equation}\label{growthrate0} \Re \lambda_\nu \sim \nu^{1/4}\end{equation}  
Heisenberg \cite{Hei,HeiICM}, then Tollmien and C. C. Lin \cite{Lin0,LinBook} were among the first physicists to use asymptotic expansions to study the (spectral) instability; see also Drazin and Reid \cite{Reid,Schlichting} for a complete account of the physical literature on the subject, and \cite{GGN2,GGN3} for a complete mathematical proof of the instability. 

The aim of this paper
is to bound the solution $v$ of (\ref{Nlin1})-(\ref{Nlin3}), uniformly as $\nu$ goes to $0$. Roughly speaking, we will prove
that there
exist positive constants $C_0,\theta_0$ such that the semigroup $e^{Lt}$ associated to the linearized problem \eqref{Nlin1}-\eqref{Nlin3} satisfies 
\beq \label{bornev1}
\| e^{Lt} v_0 \|_{X} \le C_0 \nu^{-1/4} e^{\theta_0 \nu^{1/4} t} \| v_0 \|_X
\eeq
uniformly for any positive $t$ and for $\nu \ll1$, for some norm $\|\cdot \|_X$ to be precise, below. We are interested in a sharp semigroup estimate in term of exponential growth in time. Standard energy estimates only yield a semigroup bound which grows in time of order $e^{\|\nabla U_s\|_{L^\infty} t}$, a bound that is far from being sharp. The interest in deriving such a sharp bound on the linearized Navier-Stokes problem is pointed out in \cite{Gr1,GGN1,GGN2,GrN1,GrN2}. See also \cite{DGV2}, where the authors derive semigroup bounds in Gevrey spaces.  

\subsection{Spectral approach}

We shall derive semigroup bounds via a spectral approach. First, we study the equation for vorticity $\omega = \nabla \times v = \partial_z v_1 - \partial_x v_2$, which reads 
\begin{equation}\label{EE-vort} (\partial_t + U\partial_x) \omega + v U'' - \sqrt\nu \Delta \omega = 0\end{equation}
together with $v = \nabla^\perp \phi$ and $\Delta \phi = \omega$. The no-slip boundary condition \eqref{Nlin3} becomes $\phi = \partial_z \phi =0$ on $\{z=0\}$. 
We then denote the linearized vorticity operator 
$$ L\omega : = \sqrt\nu \Delta \omega - U\partial_x \omega - v U'' .$$
Since $U''$ decays exponentially as $z$ tends to infinity, the operator $L$ is a compact perturbation of the Laplace operator $\sqrt\nu \Delta$, and so we can write the semigroup $e^{Lt}$ in term of the resolvent solutions; namely, there holds the Dunford's integral representation:
\beq \label{int1}
e^{L t} \omega = {1 \over 2 i \pi} \int_\Gamma e^{\lambda t}  (\lambda - L)^{-1} \omega  \, d \lambda
\eeq
where $\Gamma$ is a contour on the right of the spectrum of $L$. 

Since the boundary layer profile $U(z)$ depends only on $z$, we can take the Fourier transform in the $x$ variable, $\alpha$ being 
the dual Fourier variable, 
leading to \beq \label{int2}
e^{L t} \omega = \int_\RR e^{i\alpha x} e^{L_\alpha t} \omega_\alpha \; d\alpha 
\eeq
with $$e^{L_\alpha t} \omega_\alpha : = {1 \over 2 i \pi}
\int_{\Gamma_\alpha} e^{\lambda t } 
 (\lambda - L_\alpha)^{-1} \omega_\alpha \, d\lambda
$$  
having the contour $\Gamma_\alpha$ lie on the right of the spectrum of $L_\alpha$. In this formula, $\omega_\alpha$ is the Fourier transform of $\omega$ in tangential variables, $L_\alpha$ is the Fourier transform of $L$: 
\begin{equation}\label{def-Lalpha}L_\alpha\omega : = \sqrt\nu\Delta_\alpha \omega - Ui\alpha\omega + i\alpha U'' \phi, \qquad \Delta_\alpha \phi = \omega,\end{equation}
with the notation $$\Delta_\alpha = \partial_z^2  - \alpha^2.$$ The linear operator $L_\alpha$ is studied together with the no-slip boundary conditions: $\phi =\phi' = 0$. 

Since the operator $L_\alpha$ is a compact perturbation of the Laplace operator $\sqrt\nu \Delta_\alpha$ (with respect to the usual $L^2$ space), the unstable spectrum of $L_\alpha$ thus consists of precisely finitely many point spectrum $\lambda$, with $\Re \lambda >0$, satisfying 
$$ (\lambda - L_\alpha) \omega = 0,$$
 with the zero boundary conditions on the corresponding stream function $\phi$. By multiplying the above equation by $\phi^*$, the complex conjugate of $\phi$, and taking integration by parts, it follows that for each $\alpha \in \RR$, the point spectrum $\lambda$ of $L_\alpha$ must lie in the region 
\begin{equation}\label{def-secLa}S_{\alpha,\nu}: = \Big\{ \lambda \in \mathbb{C}~:~\Re \lambda +\alpha^2 \sqrt \nu \le C_0 |\alpha|, \quad |\Im \lambda |  \le C_0 |\alpha| \Big \}, \end{equation}
for some universal constant $C_0$ (depending only on $U$).



\subsection{Orr-Sommerfeld equations}
In order to construct the resolvent solutions $ (\lambda - L_\alpha)^{-1} \omega_\alpha$,
we set \begin{equation}\label{resolvent-EE}
\theta_\alpha = (\lambda - L_\alpha)^{-1} \omega_\alpha .
\end{equation}
Let $\phi_\alpha$ be the corresponding stream function, defined through the elliptic equation, 
$$
 \Delta_\alpha \phi_\alpha= \theta_\alpha.
$$
Then the stream function solves the following $4^{th}$-order ODEs, the classical Orr-Sommerfeld equations, 
\beq \label{OS}
\OS(\phi_\alpha) := - \eps \Delta_\alpha^2 \phi_\alpha + (U - c) \Delta_\alpha \phi_\alpha - U'' \phi_\alpha = {\omega_\alpha \over i \alpha},
\eeq
together with the boundary conditions
\beq \label{OS2}
{\phi_\alpha}_{\vert_{z=0}} = \partial_z{ \phi_\alpha}_{\vert_{z=0}} = 0, \qquad \lim_{z\to \infty}\phi_\alpha(z) =0.
\eeq
In the above, for convenience, we have denoted 
\begin{equation}\label{lambda-alpha}
\eps = {\sqrt\nu \over i \alpha} 
, \qquad c =  - \frac{\lambda}{i\alpha} \end{equation}
both of which are complex numbers. Occasionally, we write $\OS_{\alpha,c}(\cdot)$ in place of $\OS(\cdot)$ to stress the dependence on $\alpha$ and $c$.

Solving the resolvent equation \eqref{resolvent-EE} is thus equivalent to solve the Orr-Sommerfeld problem \eqref{OS}-\eqref{OS2}. We shall solve the latter problem via constructing its Green function. Precisely, for each fixed $\alpha \in \RR_+$ and $c\in \CC$, we let $G_{\alpha,c}(x,z)$ be the corresponding Green kernel of the OS problem. 
By definition, for each $x\in \RR$ and $c\in \CC$, $G_{\alpha,c}(x,z)$ solves the Orr-Sommerfeld equations in the sense
$$ \OS(G_{\alpha,c} (x,\cdot)) = \delta_x (\cdot)$$
on $z\ge 0$, together with the boundary conditions:
$$G_{\alpha,c}(x,0) = \partial_z G_{\alpha,c} (x,0) =0, \qquad \lim_{z\to \infty} G_{\alpha,c}(x,z) =0.$$
That is, for $z\not = x$, the Green function $G_{\alpha,c} (x,z)$ solves the homogenous Orr-Sommerfeld equations, together with the following jump conditions across $z=x$:
$$ [\partial_z^kG_{\alpha,c}(x,z)]_{\vert_{z=x}} = 0, \qquad [\epsilon\partial_z^3 G_{\alpha,c}(x,z)]_{\vert_{z=x}} = -1$$ 
for $k=0,1,2$. Here, the jump $[f(z)]_{\vert_{z=x}}$ across $z=x$ is defined to be the value of the right limit subtracted by that of the left limit as $z\to x$.  

Finally, we let $G_\alpha(z,t;x)$ be the corresponding temporal Green function, defined by 
\begin{equation}\label{def-Ga} G_\alpha(z,t;x) : = \frac{1}{2\pi i} \int_{\Gamma_\alpha} e^{\lambda t} G_{\alpha,c}(x,z) \; \frac{d\lambda}{i\alpha}
\end{equation}
in which $c = i \alpha^{-1} \lambda$. Then, the Navier-Stokes semigroup $e^{L_\alpha t}$ for vorticity is constructed by 
\begin{equation}\label{def-eLt}
e^{L_\alpha t} \omega_\alpha (z) :=\int_0^\infty \Delta_\alpha G_\alpha(z,t;x)  \omega_\alpha (x) \; dx.
\end{equation}
We stress that by construction, the corresponding velocity of $e^{L_\alpha t} \omega_\alpha$ satisfies the no-slip boundary condition at $z=0$.  Our goal is to derive pointwise bounds on the temporal Green function and bounds on the semigroup $e^{L_\alpha t}$.

\subsection{Spectrum of the Orr-Sommerfeld problem}\label{sec-OSspectrum}

To construct the Green function of the Orr-Sommerfeld problem, we first need to analyze its spectrum, or by definition, the complement of the range of $c$ over which the Orr-Sommerfeld problem \eqref{OS}-\eqref{OS2} is solvable for each initial data $\omega_\alpha$. For this, we study the corresponding homogenous Orr-Sommerfeld problem 
\beq \label{OS-h}
\OS_{\alpha,c}(\phi) =0
\eeq
together with the boundary conditions $\phi = \phi'=0$ on $z=0$. 

In what follows, we focus on the case $\alpha>0$; the other case being similar. The unstable spectrum then corresponds to the case when $\Im c>0$. By view of \eqref{def-secLa} and the relation $\lambda = -i\alpha c$, the spectrum $c$ must satisfy $|\Re c|\le C_0$ and $\Im c + \alpha \sqrt \nu \le C_0$, for the same constant $C_0$ as in \eqref{def-secLa}.  

Next, since $U(z)$ converges to a finite constant $U_+$ as $z\to \infty$, solutions to the homogenous Orr-Sommerfeld equation \eqref{OS-h} converge to solutions of the limiting, constant-coefficient equations
\begin{equation}\label{OS-plus} \OS_+(\phi) = - \eps \Delta_\alpha^2 \phi_\alpha + (U_+ - c) \Delta_\alpha \phi_\alpha =0.\end{equation}
Clearly, \eqref{OS-plus} has four independent solutions $e^{\pm\mu_s z}$ and $e^{\pm \mu_f^+ z}$, with 
\begin{equation}\label{def-sfrate}\mu_s = \alpha, \qquad \mu_f =  \epsilon^{-1/2}\sqrt{U-c+\alpha^2 \epsilon},\qquad \mu_f^+  = \lim_{z\to \infty}\mu_f,\end{equation}
in which the square root takes the positive real part. Observe that as long as $|U(z)-c| \gg \epsilon$, or equivalently $|\lambda + i\alpha U(z)|\gg \sqrt \nu$, we have $\mu_s \ll \mu_f$. That is, solutions to the Orr-Sommerfeld equation consist of ``slow behavior'' $e^{\pm\mu_s z}$ and ``fast behavior'' $e^{\pm \mu_f z}$, asymptotically near the infinity. By view of \eqref{OS-h} and \eqref{OS-plus}, the two slow modes are perturbations from the Rayleigh solutions $(U-c)\Delta_\alpha\phi - U''\phi =0$ and the two fast modes are linked to so-called Airy-type solutions $(-\eps \Delta_\alpha + U-c)\Delta_\alpha \phi =0$. 

Our goal is to construct all four independent solutions to the Orr-Sommerfeld equation \eqref{OS-h}, with the aforementioned slow and fast behavior at infinity. We divide into two cases:  

\begin{itemize}

\item Away from critical layers: $|\mathrm{Range}(U)-c|\gtrsim 1$
  
\item Near critical layers: $|\mathrm{Range}(U)-c|\ll1$. 

\end{itemize}  
  By a critical layer, we are referred to the complex point $z = z_c$ at which $c = U(z_c)$. The former case when $c$ is away from the critical layers has been treated in our recent paper \cite{GrN1}, whereas the 
latter case is the main subject of this present paper.

To understand the Orr-Sommerfeld spectrum, we let $\phi_{\alpha,c}^s, \phi_{\alpha,c}^f$ be two independent, slow and fast decaying solutions of the Orr-Sommerfeld equations $\OS_{\alpha,c}(\phi) =0$, with their normalized amplitude $\|\phi_{\alpha,c}^s\|_{L^\infty} = \| \phi_{\alpha,c}^f \|_{L^\infty} = 1$. Set 
\begin{equation}\label{def-Evans} 
D(\alpha,c) : = \mu_f^{-1} \mathrm{det} \begin{pmatrix}  \phi_{\alpha,c}^s & \phi_{\alpha,c}^f \\ \partial_z  \phi_{\alpha,c}^s & \partial_z  \phi_{\alpha,c}^f \end{pmatrix}_{\vert_{z=0}} 
,\end{equation}
which is often referred to as the Evans function. Clearly, there are non trivial solutions $(\alpha,c,\phi)$ to the Orr-Sommerfeld problem if and only if $D(\alpha,c) =0$. In addition, the Orr-Sommerfeld solutions and hence the Evans function $D(\alpha,c)$ is analytic in $c$ away from the critical layers. As a consequence, there are at most finitely many zeros $c$ on $\Im c>0$, and eigenfunction corresponding to each unstable eigenvalue $c$ (if exists) is of the form 
\begin{equation}\label{eigen-mode000} \phi =  \phi_{\alpha,c}^s - a(\alpha,c) \phi_{\alpha,c}^f , \qquad a(\alpha,c) = \frac{\phi_{\alpha,c}^s(0)}{\phi_{\alpha,c}^f(0)}.\end{equation}

In estimating the Green function through the contour integral \eqref{def-Ga}, we can move the contour across the discrete spectrum by adding corresponding projections on the eigenfunction. However, we cannot move the contour of integration $\Gamma_\alpha$ across the Euler continuous spectrum $-i\alpha \mathrm{Range}(U)$ (or equivalently, $c$ across the range of $U$), since the Orr-Sommerfeld solutions are singular near the critical layers. In addition, there are unstable eigenvalues that exist near the critical layers and that vanish in the inviscid limit (see Section \ref{sec-spectral}, below). One of the contributions of this paper is to carefully study the contour integral near the critical layers and thus to provide sharp bounds on the Navier-Stokes semigroup. 

\subsection{Critical layers}\label{sec-crlayers}

As mentioned, to derive sharp semigroup bounds, we need to study the resolvent equations with the temporal frequency $\lambda$ to be arbitrarily near the continuous spectrum of the Euler operator, or equivalently, to study the Orr-Sommerfeld equations for $c$ arbitrarily close to the range of $U(z)$. For this reason, we will construct solutions to Orr-Sommerfeld equations in the regime when 
$$
\Im c\ll1
$$
and $\Re c$ lies within the range of $U(z)$. As a consequence, there exists a complex number $z_c$ (which is unique, since $U(z)$ is strictly monotonic) such that 
$$
U(z_c) = c.
$$
Clearly, $\Im z_c \ll1$. The presence of critical layers greatly complicates the analysis of constructing Orr-Sommerfeld solutions and deriving uniform bounds for the corresponding Green function. Let us briefly explain this issue. 

Roughly speaking, there are two independent solutions to the Orr-Som\-mer\-feld equations that are approximated by the Rayleigh solutions, solving 
$$\Ray_\alpha (\phi) := (U-c)\Delta_\alpha \phi - U'' \phi =0.$$ 
However, one cannot view the Orr-Som\-mer\-feld operator as a perturbation of the Rayleigh operator, or 
equivalently, $\epsilon \Delta_\alpha^2 \Ray_\alpha^{-1}$ is not a good iteration operator, for $\eps \ll1$. Precisely, when $z\approx z_c$, Rayleigh solutions experience a singularity of the form $(z-z_c)\log (z-z_c)$ and therefore $\epsilon \Delta_\alpha^2 \Ray_\alpha^{-1}$
 has a singularity of order $(z-z_c)^{-3}$.  
 
 To deal with the singularity, we need to examine the leading operator 
 in the Orr-Sommerfeld equations near the singular point $z = z_c$. Indeed, we introduce the blow-up variable: 
 $Z = (z-z_c)/\delta$ and search for the Ansatz solution $\phi = \phi_\mathrm{cr}(Z)$, 
leading to 
$$
\partial_Z^4 \phi_\mathrm{cr} \approx Z\partial_Z^2 \phi_\mathrm{cr}
$$
with the critical layer size $\delta \approx \epsilon^{1/3}$. That is, within the critical layer, $\partial_z^2 \phi_\mathrm{cr}$ solves the classical Airy equation. 

Following \cite{GGN2,GGN3}, we construct Orr-Sommerfeld solutions via an iterative approach. Precisely, we introduce the following iterative operator $$ 
\Iter : = \underbrace{\Airy^{-1}}_{\mbox{critical layer}} \circ \quad \underbrace{\epsilon \Delta^2_\alpha}_{\mbox{error}} 
\quad \circ \quad \underbrace{\Ray_\alpha^{-1}}_{\mbox{inviscid}}
$$
in which $\Airy(\cdot)$ denotes the following modified Airy operator
\begin{equation} \label{opAiry}
\Airy(\phi) := \epsilon \Delta_\alpha ^2 \phi - (U-c) \Delta_\alpha \phi .
\end{equation}
The Airy operator governs the behavior of Orr-Sommerfeld solutions near critical layers. Like the classical Airy operator, $\Airy(\cdot)$ has smoothing effect, which will be studied in great details. 
In the sequel, we shall introduce suitable function spaces on which the Iter operator is contractive. The inverses $\Airy^{-1}$ and $\Ray_\alpha^{-1}$ will be constructed appropriately via their corresponding Green functions. 

\subsection{Spectral instability}\label{sec-spectral}

In this section, we recall the following spectral instability theorem, proved in \cite{GGN3}, providing unstable eigenvalues for generic shear flows.  

\begin{theorem}[Spectral instability; \cite{GGN3}]\label{theo-spectralinstablity}
Let $U(z)$ be an arbitrary shear profile with $U(0)=0$ and $U'(0) > 0$ and satisfy 
$$
\sup_{z \ge 0} | \partial^k_z (U(z) - U_+) e^{\eta_0 z} | < + \infty, \qquad k=0,\cdots ,4,
$$ for some constants $U_+$ and $\eta_0 > 0$. Set $\alpha_\mathrm{low}\sim \nu^{1/8}$ and  $ \alpha_\mathrm{up}\sim \nu^{1/12}$ be the lower and upper stability branches. 

Then, there is a critical Reynolds number $R_c = \frac{1}{\sqrt{\nu_c}}$ so that for all positive $\nu\le \nu_c$ and all $\alpha \in (\alpha_\mathrm{low}, \alpha_\mathrm{up})$, there exists
a nontrivial solution $(c_\nu, \phi_\nu)$, with $\mathrm{Im} ~c_\nu >0$, to the Orr-Sommerfeld problem \eqref{OS-h} such that $v_\nu: = e^{i\alpha(x-c_\nu t) } \nabla^\perp_\alpha \phi_\nu (z)$ solves the linearized Navier-Stokes problem 
\eqref{Nlin1}-\eqref{Nlin3}. In the case of instability, there holds the following estimate for the growth rate of the unstable solutions:
$$ \Re \lambda_\nu = \alpha \Im c_\nu \quad \approx\quad  \nu^{1/4}$$
in the inviscid limit as $\nu \to 0$. 
\end{theorem}

\begin{remark}\label{rem-18}
By construction, the stream function $\phi_\nu$ is constructed through asymptotic expansions and of the form 
 $$\phi_\nu(z) =  \phi_{in,0}(z) +\delta_\bl \phi_{bl,0} ( \delta_\bl^{-1} z) + \delta_{cr} \phi_{cr,0} (\delta_{cr}^{-1} \eta(z)) $$ 
for some boundary layer function $\phi_{bl,0}$ and some critical layer function $\phi_{cr,0}$, with the Langer's variable $\eta(z) \approx z-z_c$ near the critical layer. The critical layer thickness is of order 
$$
\delta_{cr}  = (\epsilon/U_c')^{1/3} \approx \nu^{1/6} \alpha^{-1/3} , \qquad \nu^{1/8} \lesssim \alpha \lesssim \nu^{1/12},
$$
and the boundary sublayer thickness is of order 
$$
\delta_{bl} = \Bigl( \frac{\sqrt \nu}{\alpha_\nu (U_0-c_\nu) }\Bigr)^{1/2}  \approx  \nu^{1/8}.
$$
In particular, for the lower instability branch $\alpha \sim \nu^{1/8}$, both critical layer and boundary sublayer have the same thickness and are of order $\nu^{1/8}$.  
\end{remark}

\section{Main results}

\subsection{Green function for Orr-Sommerfeld}

We shall construct the Green function $G_{\alpha,c}(x,z)$ for the Orr-Sommerfeld problem
\begin{equation}
\label{OS-problem}
\begin{aligned}
\OS(\phi) = - \eps \Delta_\alpha^2 \phi + (U - c) \Delta_\alpha \phi - U'' \phi &=0,
\\
\phi_{\vert_{z=0}} = \partial_z \phi_{\vert_{z=0}} =0, \qquad \lim_{z\to \infty}\phi(z)&=0.
\end{aligned}\end{equation}
As discussed in Section \ref{sec-OSspectrum}, the homogenous equations $\OS(\phi)=0$ have four independent solutions, two of which are slow modes and linked with the Rayleigh operator $\Ray_\alpha = (U - c) \Delta_\alpha - U'' $, and the other two are fast modes and linked with the Airy operator $\Airy = \eps \Delta_\alpha^2  - (U - c) \Delta_\alpha $. The standard conjugation method for ODEs does not apply directly to construct these slow and fast modes, due to the dependence on various parameters in the problem. 
 
In \cite{GrN1}, we initiated an analytical program to construct the Green function for the Orr-Sommerfeld problem. There, we consider the case when $\alpha |c - \Range(U)| \gtrsim 1$ or equivalently the temporal frequency $\lambda$ remains away from the essential spectrum of the corresponding linearized Euler operator. We recall the following theorem, proved in \cite{GrN1}.

\begin{theorem}[\cite{GrN1}]\label{theo-GreenOS} For any fixed positive $\epsilon_0$, we set 
$$A_{\epsilon_0} := \Big\{ (\alpha,c)\in \RR_+ \times \CC~:~|c - \Range (U)|\ge \frac{\epsilon_0}{1+\alpha}\Big\} $$
and 
\begin{equation}\label{def-mMf}
m_f = \inf_{z} \Re \mu_f (z) , \qquad M_f = \sup_z \Re \mu_f(z) , \end{equation}
with $\mu_f$ defined as in \eqref{def-sfrate}. 
Let $G_{\alpha,c}(x,z)$ be  the Green function of the Orr-Sommerfeld problem \eqref{OS-problem} and let $D(\alpha,c)$ be the corresponding Evans function \eqref{def-Evans}. Then, there are universal positive constants $\theta_0, C_0$ so that
\begin{equation}\label{est-GrOS}
\begin{aligned}
 |G_{\alpha,c}(x,z)|  &\le C_0 [ D(\alpha,c)]^{-1} \Big(  \frac{1}{\mu_s} e^{-\theta_0 \mu_s (|z| + |x|)}  +  \frac{1}{ m_f }  e^{-\theta_0 m_f (|z|+|x|)}\Big)
\\&\quad +  C_0 \Big( \frac{1}{\mu_s}e^{-\theta_0\mu_s |x-z|} + \frac{1}{ m_f} e^{- \theta_0 m_f|x-z|} \Big) 
\end{aligned}
\end{equation}
uniformly in $(\alpha,c)$, within $A_{\epsilon_0}$, and uniformly for all $x,z\ge 0$. Similar bounds hold for the derivatives. 
\end{theorem}

Next, we consider the case when $|c - \Range(U)| \ll 1$. The critical layers appear, as discussed in Section \ref{sec-crlayers}. It suffices for the derivation of semigroup bounds to consider the following range of $\alpha$:
\begin{equation}\label{range-aGrOS} \sqrt \nu \ll \alpha \ll \nu^{-1/4} .
\end{equation}
Indeed, we observe that when $\alpha^2\sqrt \nu \gtrsim 1$, the linearized Navier-Stokes problem has a nonvanishing spectral gap in the inviscid limit, and thus, the boundedness of the corresponding semigroup can be obtained, without having to go into the continuous spectrum of Euler. Whereas, in the case when $\alpha \lesssim \sqrt\nu$, the Navier-Stokes equations are simply a regular perturbation of the heat equation and the semigroup bounds are easy to establish.   

Our first main result in this paper is as follows. 

\begin{theorem}\label{theo-GreenOS-stable} Let $\alpha$ satisfy \eqref{range-aGrOS}, and let $c$ be sufficiently close to the Range of $U$ so that $z_c$ exists and is finite. Assume in addition that $(\alpha,c)$ satisfy 
\begin{equation}\label{small-c} \epsilon^{1/8} |\log \Im c | \ll1.\end{equation}
Let $G_{\alpha,c}(x,z)$ be  the Green function of the Orr-Sommerfeld problem \eqref{OS-problem}, $\mu_s, \mu_f$ be defined as in \eqref{def-sfrate}, and let $D(\alpha,c)$ be the corresponding Evans function \eqref{def-Evans}. 
Then, there are universal positive constants $\theta_0, C_0$ so that the followings hold.

(i) In the case when $\alpha \gtrsim 1$, we have 
\begin{equation}\label{est-GrOS-stable}
\begin{aligned}
 |G_{\alpha,c}(x,z)|  &\le 
  C_0\Big( \mu_s^{-1}e^{-\theta_0\mu_s |x-z|} + \delta \langle Z\rangle^{-5/4} \langle X \rangle^{3/4} e^{ -  \int_{x}^z \Re \mu_f(y)\; dy}\Big)
\\
& \quad+ C_0 [ D(\alpha,c)]^{-1} \Big(  \mu_s^{-1}e^{-\theta_0 \mu_s (|z| + |x|)}  
\\&\quad+ \delta
\langle Z\rangle^{-5/4} \langle X \rangle^{3/4} e^{ -  \int_{z_c}^z \Re \mu_f(y)\; dy}e^{ -  \int_{z_c}^x \Re \mu_f(y)\; dy} \Big)
\end{aligned}
\end{equation}
uniformly for all $x,z\ge 0$. 

(ii) In the case when $\alpha \ll1$, the slow behavior of the Green function satisfies the following: the two terms $\mu_s^{-1}e^{-\theta_0\mu_s |x\pm z|}$ in \eqref{est-GrOS-stable} are replaced by 
\begin{equation}\label{Gslow-small-a} \frac{C_0}{|U-c|} \Big( |U-c| + \cO(\alpha)\Big) e^{- \theta_0 \mu_s |x\pm z|} ,\end{equation}
uniformly for all $x,z\ge 0$. 

Similar bounds hold for the derivatives. In the above, $X,Z$ denote the classical Langer's variables
\begin{equation}\label{Lar-XZ}\begin{aligned}
X = \Big( \frac 32 \int_{z_c}^x \mu_f(y) \; dy\Big)^{2/3} , \qquad Z  =\Big( \frac 32 \int_{z_c}^z \mu_f(y) \; dy\Big)^{2/3} .
 \end{aligned}\end{equation}
 \end{theorem}

Let us briefly explain the Green function bounds. The slow behavior in \eqref{est-GrOS-stable} is  due to the behavior of Rayleigh solutions, which are of the form $e^{\pm \mu_s z}$ for large $z$. As for the fast behavior, Orr-Sommerfeld solutions are essentially the second primitives, denoted by $Ai(2,\cdot)$ and $Ci(2,\cdot)$, of the classical Airy functions $Ai(\cdot)$ and $Ci(\cdot)$, corresponding to decaying and growing solutions at infinite. Asymptotically, there hold 
$$
\begin{aligned}Ai(2,Z) &\le C_0 \langle Z \rangle^{-5/4}   e^{- \sqrt{2|Z|} Z/3} ,
\\Ci(2,Z) &\le C_0 \langle Z \rangle^{-5/4}   e^{ \sqrt{2|Z|} Z/3},
\end{aligned} 
$$
in which $Z = \delta^{-1}\eta(z)$ denotes the Langer's variable. In particular, $Z$ is of order $\epsilon^{-1/3}|z-z_c|$ near the critical layers $z = z_c$ and of order 
$\epsilon^{-1/3}\langle z\rangle^{2/3}$ for large $z$. 
By view of \eqref{Lar-XZ}, we have 
$$  e^{\pm \sqrt{2|Z|} Z/3} = e^{\pm\int_{z_c}^z\mu_f(y) \; dy } $$
which explains the fast behavior in the Green function estimates \eqref{est-GrOS-stable}. 

Like the case away from critical layers (\cite{GrN1}), our construction of the Green function for the Orr-Sommerfeld problem follows closely the analytical approach introduced by Zumbrun-Howard in the seminal paper \cite{ZH}. The main difficulty is to construct independent solutions to the homogenous Orr-Sommerfeld equations, having uniform bounds with the parameters $\alpha, \epsilon, c$. In particular, critical layers appear; see Section \ref{sec-crlayers}. Certainly, the standard conjugation method for ODEs (see, for instance, \cite{ZH,MZ1}) does not apply directly due to the dependence on the various parameters in the equation. Our construction of slow modes of the OS equations, especially near critical layers, is based on the recent iterative approach introduced in \cite{GGN2,GGN3}.

\subsection{Green function for vorticity}

The Green function $G_{\alpha,c}(x,z)$ of the Orr-Sommerfeld problem \eqref{OS-problem} describes the behavior of the stream function $\phi$. In this section, we derive bounds for $\Delta_\alpha G_{\alpha,c}(x,z)$, which is the Green function for vorticity. It follows from the Orr Sommerfeld equations that  $\Delta_\alpha G_{\alpha,c}(x,z)$ solves
$$ \Big( -\epsilon \Delta_\alpha + U - c \Big) \Delta_\alpha G_{\alpha,c} (x,z) = \delta_x(z) + U'' G_{\alpha,c} (x,z) .$$   
This shows that to leading order, the vorticity $ \Delta_\alpha G_{\alpha,c} (x,z)$ is governed by the $2^{th}$-order modified Airy operator 
$$\cA: = -\epsilon \Delta_\alpha + U - c.$$
To describe this, let us set $\mathcal{G}(x,z)$ to be the Green function of $-\epsilon \Delta_\alpha + U-c$, which consists of precisely the fast behavior. We then write the Green function for vorticity $
\Delta_\alpha G_{\alpha,c}(x,z) =  \mathcal{G}(x,z) + \mathcal{R}_G(x,z) $. 
It follows that the residual Green function $ \mathcal{R}_G(x,z)$ can be computed by 
\begin{equation}\label{def-RGGG} \mathcal{R}_G(x,z) := \int_0^\infty \mathcal{G}(y,z) U''(y) G_{\alpha,c} (x,y) \; dy .\end{equation}
Roughly speaking, the integration gains an extra factor of ${\mu_f}^{-1}$, which is precisely the size of the fast oscillation in the Green function 
$\mathcal{G}(x,z)$. That is, the residual $\mathcal{R}_G(x,z)$ is of order of $(\epsilon \mu_f^2)^{-1}$ times $G_{\alpha,c}(x,z)$.

Precisely, we obtain the following theorem. 

\begin{theorem}\label{theo-vort-stable} Under the same assumptions as in Theorem \ref{theo-GreenOS-stable}, the Green function for vorticity $\Delta_\alpha G_{\alpha,c}(x,z) $ can be written as 
\begin{equation}\label{def-vortG12}  \Delta_\alpha G_{\alpha,c}(x,z) = \mathcal{G}(x,z) + \mathcal{R}_G(x,z) \end{equation}
in which $\mathcal{G}(x,z)$ denotes the Green function of $-\epsilon \Delta_\alpha + U - c$. Let $\delta = \eps^{1/3} / (U_c')^{1/3} $ be the critical layer thickness. 
There holds 
\begin{equation}\label{estGa-11}
\begin{aligned}
|\D_z^\ell \D_x^k \mathcal{G}(x,z) | 
\le 
C_{x,z} \delta^{-2-k-\ell}\langle Z \rangle^{(2k-1)/4}\langle X \rangle^{(2\ell-1)/4}
e^{ -  \int_{x}^z \Re \mu_f(y)\; dy}
\end{aligned}
\end{equation} 
for all $k,\ell \ge 0$, in which $C_{x,z}\lesssim  \langle z\rangle^{k/3}\langle x\rangle^{(1-\ell)/3}$. In addition, the remainder  $ \mathcal{R}_G(x,z)$ satisfies 
$$
\begin{aligned}
 |\mathcal{R}_G(x,z)|  &\le 
  C_0\delta \langle Z \rangle^{-1/2}\Big( \mu_s^{-1}e^{-\theta_0\mu_s |x-z|} + \delta \langle Z\rangle^{-1/2} e^{ -  \int_{x}^z \Re \mu_f(y)\; dy}\Big)
\\
& \quad+ C_0 \delta \langle Z \rangle^{-1/2}[ D(\alpha,c)]^{-1} \Big(  \mu_s^{-1}e^{-\theta_0 \mu_s (|z| + |x|)}  
\\&\quad+ \delta
\langle Z\rangle^{-1/2} e^{ -  \int_{z_c}^z \Re \mu_f(y)\; dy}e^{ -  \int_{z_c}^x \Re \mu_f(y)\; dy} \Big),
\end{aligned}
$$
for $x,z\ge 0$ and for $\alpha \gtrsim 1$. In the case when $\alpha \ll1$, the slow behavior is again replaced by the bound \eqref{Gslow-small-a}, exactly as done in Theorem \ref{theo-GreenOS-stable}.   
\end{theorem}

\subsection{Semigroup bounds in boundary layer spaces}

Our next main result yields a sharp bound on the semigroup $e^{L_\alpha t}$, constructed by the Dunford's integral
\begin{equation}\label{Dunford1}e^{L_\alpha t} \omega_\alpha = {1 \over 2 i \pi}
\int_{\Gamma_\alpha} e^{\lambda t } 
 (\lambda - L_\alpha)^{-1} \omega_\alpha \, d\lambda,
\end{equation}
uniformly in the small viscosity and in the Fourier frequency $\alpha \in \RR_+$, for $\Gamma_\alpha$ lying on the right of the spectrum of $L_\alpha$. Recall that the unstable spectrum of $L_\alpha$ consists precisely of discrete spectrum, having eigenvalues and eigenmodes solve the Orr-Sommerfeld problem; see Section \ref{sec-OSspectrum}.

Away from the critical layers, eigenmodes are smooth in $z$, analytic in $\lambda$, and of the form $\nabla^\perp (e^{i\alpha (x-ct)} \phi(z))$, with $\phi$ being defined as in \eqref{eigen-mode000}. Thus, we can move the contour of integration in \eqref{Dunford1} pass the point spectrum, by adding projections on the eigenmodes. This can be continued, as long as the contour of integration remains away from the Euler continuous spectrum $-i\alpha \mathrm{Range(U)}$, or equivalently $c$ is away from the critical layers. 

Near the critical layers, solutions to the Orr-Sommerfeld equations are singular with singularity of the form $(z-z_c)\log (z-z_c)$, and we are no longer able to take the contour cross the critical layers. The main contribution of this paper is to allow the contour to stay sufficiently close to the critical layers (precisely, within a distance of order $\nu^{1/4}$). This is sufficient for us to obtain sharp bounds on the semigroup $e^{L_\alpha t}$.

In this paper, in coherence with the physical literature (\cite{Reid}), we assume that the unstable eigenvalues found in the spectral instability result, Theorem \ref{theo-spectralinstablity}, are maximal eigenvalues. Precisely, let us set $\lambda_{\alpha,\nu}$ to be the maximal unstable eigenvalue of $L_\alpha$ (or zero, if none exists), for each $\alpha >0$ and $\nu >0$, and introduce 
\begin{equation}\label{def-ga0max0}
\gamma_0 : = \lim_{\nu \to 0} \sup_{\alpha \in \RR}  \nu^{-1/4}\Re \lambda_{\alpha,\nu}  .
\end{equation}
The existence of unstable eigenvalues in Theorem \ref{theo-spectralinstablity} implies that $\gamma_0$ is positive. Our spectral assumption is that $\gamma_0$ is finite (that is, the eigenvalues in Theorem \ref{theo-spectralinstablity} are maximal).

Our spectral assumption can be stated in term of the Evans function as follows: 
\begin{assumption}\label{assump-Evans} Define $\gamma_0$ as in \eqref{def-ga0max0} and the Evans function $D(\alpha,c)$ as in \eqref{def-Evans}. For any $\gamma_1>\gamma_0$, there holds 
$$D(\alpha,c) \not =0$$
for all $\alpha$ and for $c = -\lambda/i\alpha$ with $\Re \lambda > \gamma_0 \nu^{1/4}$.   
\end{assumption}

Our semigroup bounds will be obtained with respect to the following boundary layer norms for vorticity 
\begin{equation}\label{assmp-wbl-stable} 
\| \omega_\alpha\|_{ \beta, \gamma, p} : = \sup_{z\ge 0} \Big [ \Bigl( 1 +  \sum_{q=1}^p\delta^{-q} \phi_{P-1+q} (\delta^{-1} z)  \Bigr)^{-1} e^{\beta z} |\omega_\alpha (z)| \Big]\end{equation}
with $p\ge 0$, $\beta>0$, and with the boundary layer thickness 
$$\delta = \gamma \nu^{1/8}$$
for some $\gamma>0$. Here, the order $\nu^{1/8}$ is dictated by the structure of the unstable eigenmodes; see Remark \ref{rem-18}. We stress that the boundary layer norm is introduced to capture the large behavior of vorticity near the boundary. In the case when $p=0$,  $\| \omega_\alpha\|_{ \beta, \gamma, p} $ reduces to the usual exponentially weighted $L^\infty$ norm $ \| \omega_\alpha\|_{L^\infty_\beta}$. We introduce the boundary layer space ${\cal B}^{\beta,\gamma,p}$ to consist of functions whose $\|\cdot \|_{ \beta, \gamma, p} $ norm is finite, and write $L^\infty_\beta = {\cal B}^{\beta,\gamma,0}$. Clearly, $L^\infty_\beta \subset {\cal B}^{\beta,\gamma,q} \subset {\cal B}^{\beta,\gamma,p}$ for $0\le q\le p$.

Our final main result is as follows. 

\begin{theorem} \label{theo-eLt-stable} Let $\omega_\alpha \in {\cal B}^{\beta,\gamma,1}$ for some positive $\beta$, and $\gamma_0$ be defined as in \eqref{def-ga0max0}. 
Then, for any $\gamma_1>\gamma_0$, there are positive constants $C_0,\theta_0$ so that 
$$\begin{aligned}
\| e^{L_\alpha t}\omega_\alpha\|_{ \beta, \gamma, 1} &\le C_\nu e^{\gamma_1 \nu^{1/4} t }  e^{- \frac14 \alpha^2 \sqrt\nu t}  \| \omega_\alpha\|_{ \beta, \gamma, 1}
\\
\| \partial_ze^{L_\alpha t}\omega_\alpha\|_{ \beta, \gamma, 1} &\le C_\nu\Big( \nu^{-1/8}+ (\sqrt \nu t)^{-1/2} \Big) e^{\gamma_1 \nu^{1/4} t }  e^{- \frac14 \alpha^2 \sqrt\nu t}  \| \omega_\alpha\|_{ \beta, \gamma, 1} ,
\end{aligned}$$
in which the constant $C_\nu$ is defined by 
\begin{equation}\label{def-Cnu} C_\nu : = C_0 \Big( 1 + \alpha^2 \nu^{-1/4} \chi_{\{\alpha \ll1\}} \Big)\end{equation}
for some universal constant $C_0$ and for $\chi_{\{\cdot\}}$ being the characteristic function. 
\end{theorem}

The main result in Theorem \ref{theo-eLt-stable} is a continuation of \cite{GrN1} to provide sharp and uniform semigroup bounds in the inviscid limit. It also provides {\em a stable semigroup estimate} with respect to the boundary layer norm for (stable) boundary layers. As was pointed out in \cite{DG,DGV2,Gr1,GGN1,GGN2,GrN1,GrN2,HKN2}, such a stable sharp bound is crucial for the nonlinear instability of stable boundary layers, which we shall address in a forthcoming paper. Its proof relies on a very careful and 
detailed construction and analysis of the
Green function of linear Navier-Stokes equations, constructed in \cite{GrN1}, and the Fourier-Laplace approach (\cite{Z1, Z2,GR08,GrN1,GrN2}).

\begin{remark}
Comparing to the case when $\lambda$ is away from critical layers (\cite{GrN1,GrN2}), we note that there is a loss of $\nu^{-1/4}$ in the semigroup bound. This loss is precisely due to the presence of singularity (or critical layers) in the resolvent equation, when inverting the linearized Euler operator $\partial_t + i\alpha U - i\alpha U'' (\Delta_\alpha)^{-1}$, within the temporal frequency $\Re \lambda \sim \nu^{1/4}$ and the spatial frequency $\alpha \ll1$; see Remark \ref{rem-lossGnu}. 
\end{remark}

\begin{remark}
In view of the spectral instability, Theorem \ref{theo-spectralinstablity}, the unstable solutions occur precisely when the spatial frequency $\alpha$ satisfies $\nu^{1/8}\lesssim \alpha \lesssim \nu^{1/12}$. In particular, if we assume in addition that the maximal unstable eigenvalue corresponds to $\alpha \sim \nu^{1/8}$, then the semigroup $e^{L_\alpha t}$ with respect to the boundary layer norms is bounded by
$$C_0 \Big[ e^{\gamma_1 \nu^{1/4} t } + \nu^{\beta - \frac14} e^{\gamma_\beta \nu^{1/4} t} + \nu^{-1/4} e^{\gamma_2 \nu^{1/4} t} \Big] e^{- \frac14 \alpha^2 \sqrt\nu t} $$
for $\frac14 \le \beta \le \frac16$, for some $\gamma_2 \ll \gamma_1$, and for $\gamma_\beta \le \gamma_1$. In the instability analysis, we expect that at the instability time $t_\nu$, the maximal growth dominates the remaining terms in the semigroup bound. 
\end{remark}

In view of the vorticity decomposition \eqref{def-vortG12}, we write the semigroup $e^{L_\alpha t}$ as
\begin{equation}\label{de-eLta-stable}e^{L_\alpha t}  = \mathcal{S}_\alpha +  \mathcal{R}_{\alpha} \end{equation} 
with 
\begin{equation}\label{def-SRa}\begin{aligned}
\mathcal{S}_\alpha  \omega_\alpha (z): &= \frac{1}{2\pi i}  \int_0^\infty \int_{\Gamma_{\alpha}} e^{\lambda t} \mathcal{G}(x,z) \omega_\alpha (x) \; \frac{d \lambda dx }{i\alpha},
\\
\mathcal{R}_{\alpha} \omega_\alpha (z): &= \frac{1}{2\pi i} \int_0^\infty  \int_{\Gamma_{\alpha}} e^{\lambda t}  \mathcal{R}_G(x,z) \omega_\alpha (x) \; \frac{d \lambda dx}{i\alpha}.
\end{aligned}\end{equation}
Theorem \ref{theo-eLt-stable} is a combination of semigroup estimates for $\mathcal{S}_\alpha $ and $\mathcal{R}_\alpha$, which we shall derive in the following sections.


\section{Rayleigh solutions near critical layers}\label{sec-Rayleigh}


In this section, we are aimed to construct an exact inverse for the Rayleigh operator $\Ray_\alpha(\cdot)$ and thus solve the inhomogenous Rayleigh problem 
\begin{equation}\label{Ray-smalla}
\Ray_\alpha(\phi) = (U-c)\Delta_\alpha \phi - U'' \phi =f, \qquad z\ge 0,
\end{equation}
in the presence of critical layers. We consider the case when $c$ is sufficiently close to the range of $U(z)$ so that 
$$c = U(z_c)$$ for some complex number $z_c$. 

The case when $\alpha \ll1$ has been studied carefully in our recent study (\cite{GGN2,GGN3}), which we shall recall for sake of completeness. The construction of Rayleigh solutions for $\alpha =0$ is done via an explicit Green function. 
We then use this inverse to construct an approximate inverse to $\Ray_\alpha$ operator through the construction 
of an approximate Green function. Finally, the construction of the exact inverse of $\Ray_\alpha$ follows by  an iterative procedure.


\subsection{Function spaces}\label{sec-space}


The natural framework to solve (\ref{Ray-smalla}) is to work in the complex universal cover $\widetilde  \Gamma$ of $\Gamma_{\sigma,r}\setminus\{z_c\}$, with $\Gamma_{\sigma,r}$ being a complex neighborhood of $\RR_+$. 
The Rayleigh equation has the vanishing coefficient of the highest derivative term at $z=z_c$, since $U(z_c) = c$. Using the classical results on ordinary differential equations in the complex
place we get that (\ref{Ray-smalla}) has a "multivalued" solution $\widetilde  \phi$ defined on the universal cover of $\Gamma_{\sigma,r}$. 


In constructing solutions, we will use the function spaces $X^{p,\eta}$, for $p\ge 0$, to denote the spaces consisting of holomorphic 
functions $f = f(z)$ on $\widetilde  \Gamma$ such that the norm 
$$
\| f\|_{X^{p,\eta}} : = \sup_{|z-z_c|\le 1 } \sum_{k=0}^p|  (z-z_c)^k \partial_z^k f(z) | 
+ \sup_{|z-z_c|\ge 1} \sum_{k=0}^p|  e^{\eta \Re z} \partial_z^k f(z) |  
$$ 
is bounded. In the case $p=0$, we simply write $X^\eta, \|\cdot \|_\eta$ in places of
 $X^{0,\eta}, \|\cdot \|_{X^{0,\eta}}$, respectively. We also introduce the function spaces $Y^{p,\eta} \subset X^{p,\eta}$, $p\ge 0$, 
 such that for any $f\in Y^{p,\eta}$, the function $f$ additionally satisfies 
$$
|f(z)| \le C, \quad | \partial_z f(z) | \le C (1 + | \log (z - z_c) | ) , 
$$
$$| \partial_z^k f(z) | \le C (1 + | z - z_c |^{1 - k} ),
$$
for all $|z-z_c|\le 1$ and for $2\le k \le p$. The best constant $C$ in the previous bounds, plus $\| f\|_{X^{p,\eta}}$, defines the norm $\| f \|_{Y^{p,\eta}}$.


\subsection{Rayleigh equation: $\alpha = 0$}

In this section, we recall the study (\cite{GGN2,GGN3}) of the Rayleigh operator $\Ray_0$ when
$\alpha = 0$. More precisely, we solve  
\begin{equation}\label{Ray0} \Ray_0 (\phi) = (U-c) \partial_z^2 \phi - U'' \phi = f.\end{equation}
We have the following lemma whose proof is given in \cite[Lemma 3.2]{GGN3}.

\begin{lemma}[\cite{GGN2,GGN3}] \label{lem-defphi012} Assume that $\Im c \not =0$. There are two independent solutions $\phi_{1,0},\phi_{2,0}$ of $\Ray_0(\phi) =0$ with the Wronskian determinant 
$$ 
W(\phi_{1,0}, \phi_{2,0}) := \partial_z \phi_{2,0} \phi_{1,0} - \phi_{2,0} \partial_z \phi_{1,0} = 1.
$$
Furthermore, there are holomorphic functions $P_1(z), P_2(z), Q(z)$ on $\widetilde  \Gamma$ 
with $P_1(z_c) = P_2(z_c) = 1$ and $Q(z_c)\not=0$ so that the asymptotic descriptions 
\begin{equation}\label{asy-phi012} 
\phi_{1,0}(z) = (z-z_c) P_1(z) ,\qquad \phi_{2,0}(z) = P_2(z) + Q(z) (z-z_c) \log (z-z_c)
\end{equation}
hold for $z$ near $z_c$, and  \begin{equation}\label{decay-phi012} 
 \phi_{1,0}(z) - V_+ \in Y^{p,\eta_0} , \qquad \partial_z \phi_{2,0}(z)  - \frac{1}{V_+}   \in  Y^{p,\eta_1}
\end{equation}
for some positive constants $C,\eta_0$, and $V_+ = U_+ - c$, any $\eta_1<\eta_0$, and any $p\ge 0$.
 \end{lemma}

To go back to $\Gamma_{\sigma,r}$ we have to make
a choice of the logarithm. We choose to define it on $\mathbb{C} - \{ z_c + \mathbb{R}_-  \}$.
Then $\phi_{2,0}$ is holomorphic on 
$$
\Gamma' = \Gamma_{\sigma,r} \cap \Bigl(  \mathbb{C} - \{ z_c + \mathbb{R}_-  \} \Bigr) .
$$
Note that because of the singularity at $z_c$, $\phi_{2,0}$ cannot be made holomorphic on 
$\Gamma_{\sigma,r}$.

In particular if $\Im z_c = 0$, $\phi_{2,0}$ is holomorphic in $z$ excepted on the half line $z_c + \mathbb{R}_- $.
For $z \in \mathbb{R}$, $\phi_{2,0}$ is holomorphic as a function of $c$ excepted if $z - z_c$ is real and negative,
namely excepted if $z < z_c$. 
For a fixed $z$, $\phi_{2,0}$ is an holomorphic function of $c$ provided $z_c$ does not cross
$\mathbb{R}_+$, and provided $z - z_c$ does not cross $\mathbb{R}_-$.

Let $\phi_{1,0},\phi_{2,0}$ be constructed as in Lemma \ref{lem-defphi012}. Then 
for real arguments $x$ and $z$, the Green function $G_{R,0}(x,z)$ of the $\Ray_0$ operator can be defined by 
$$
G_{R,0}(x,z) = \left\{ \begin{array}{rrr} (U(x)-c)^{-1} \phi_{1,0}(z) \phi_{2,0}(x), 
\quad \mbox{if}\quad z>x,\\
(U(x)-c)^{-1} \phi_{1,0}(x) \phi_{2,0}(z), \quad \mbox{if}\quad z<x.\end{array}\right.
$$ 
We will denote by $G_{R,0}^\pm$ the expressions of $G_{R,0}$ for $x > z$ and $x < z$. 
Here we note that $c$ is complex with $\Im c \not=0$ and so the Green function $G_{R,0}(x,z)$ is a well-defined function in $(x,z)$, continuous across $x=z$, and its first derivative has a jump across $x=z$. Let us now introduce the inverse of $\Ray_0$ as 
\begin{equation}\label{def-RayS0}
\begin{aligned}
RaySolver_0(f) (z)  &: =  \int_0^{+\infty} G_{R,0}(x,z) f(x) dx.
\end{aligned}
\end{equation}

The following lemma asserts that the operator $RaySolver_0(\cdot)$ is in fact well-defined from $X^{0,\eta}$ to $Y^{2,0}$, which in particular shows that $RaySolver_0(\cdot)$ gains two derivatives, but losses the fast decay at infinity.

\begin{lemma}\label{lem-RayS0} 
Assume that $\Im c \not =0$. For any $f\in {X^{0,\eta}}$,  the function $RaySolver_0(f)$ is a solution to the Rayleigh problem \eqref{Ray0}, defined on $\widetilde  \Gamma$. In addition, $RaySolver_0(f) \in Y^{2,0}$, and there holds  
$$
\| RaySolver_0(f)\|_{Y^{2,0}} \le C (1+|\log \Im c|) \|f\|_{{X^{0,\eta}}},
$$ 
for some universal constant $C$. 
\end{lemma}
\begin{proof} As long as it is well-defined, the function $RaySolver_0(f)(z)$ solves the equation \eqref{Ray0} at once by a direct calculation, upon noting that 
$$ \Ray_0 (G_{R,0}(x,z) ) = \delta_x(z),$$
for each fixed $x$. 

Next, by scaling, we assume that $ \| f\|_{X^{0,\eta}} = 1$. By using \eqref{decay-phi012} in Lemma \ref{lem-defphi012}, it is clear that $\phi_{1,0}(z)$ and $\phi_{2,0}(z)/(1+z)$
  are uniformly bounded. Thus, considering the cases $z>x$ and $z<x$, we obtain 
\begin{equation}\label{est-Gr0}|G_{R,0}(x,z)| \le  C \max\{ (1+x), |x-z_c|^{-1} \}.\end{equation}
That is, $G_{R,0}(x,z)$ grows linearly in $x$ for large $x$ and has a singularity of order $|x-z_c|^{-1}$ when $x$ is near $z_c$, for arbitrary $z \ge 0$.  Since $|f(z)|\le e^{-\eta z}$, the integral \eqref{def-RayS0} is well-defined and satisfies 
$$|RaySolver_0(f) (z)| \le  C \int_0^\infty e^{-\eta x} \max\{ (1+x), |x-z_c|^{-1} \}  \; dx\le C (1+|\log \Im c|),$$
in which we used the fact that $\Im z_c \approx \Im c$. 
  
Finally, as for derivatives, we need to check the order of singularities for $z$ near $z_c$. We note that $|\partial_z \phi_{2,0}| \le C (1+|\log(z-z_c)|)$, and hence 
  $$|\partial_zG_{R,0}(x,z)| \le  C \max\{ (1+x), |x-z_c|^{-1} \} (1+|\log(z-z_c)|).$$
Thus, $\partial_z RaySolver_0(f)(z)$ behaves as $1+|\log(z-z_c)|$ near the critical layer. In addition, from the $\Ray_0$ equation, we have \begin{equation}\label{identity-R0f} \partial_z^2 (RaySolver_0(f)) = \frac{U''}{U-c} RaySolver_0(f) + \frac{f}{U-c}.\end{equation}
This proves that $RaySolver_0(f) \in Y^{2,0}$, by definition of $Y^{2,0}$. \end{proof}

\begin{lemma}\label{lem-derRayS0}Assume that $\Im c \not =0$. Let $p$ be in $\{0,1,2\}$. For any $f \in X^{p,\eta}$, we have 
$$
\begin{aligned}
 \| RaySolver_0(f) \|_{Y^{p+2,0}} \le C \|f\|_{X^{p,\eta}}(1+|\log (\Im c)| )
 \end{aligned}$$ 
\end{lemma}
\begin{proof} This is Lemma \ref{lem-RayS0} when $p=0$. When $p=1$ or $2$, the lemma follows directly from the identity \eqref{identity-R0f}.
\end{proof}


\subsection{Approximate Green function for Rayleigh: $\alpha \ll1$}


Let $\phi_{1,0}$ and $\phi_{2,0}$ be the two solutions of $\Ray_0(\phi) = 0$ that are constructed above in Lemma \ref{lem-defphi012}. We note that  solutions of $\Ray_0(\phi) = f$ tend to a constant
value as $z \to + \infty$ since $\phi_{1,0} \to U_+-c$. We now construct solutions to the Rayleigh equation \eqref{Ray-smalla} with $\alpha \not=0$.
To proceed, let us introduce
\begin{equation}\label{def-phia12}
\phi_{1,\alpha } = \phi_{1,0} e^{-\alpha z} ,\qquad \phi_{2,\alpha} = \phi_{2,0} e^{-\alpha z}.
\end{equation}
A direct calculation shows that the Wronskian determinant 
$$
W[\phi_{1,\alpha},\phi_{2,\alpha}] =  \partial_z \phi_{2,\alpha} \phi_{1,\alpha} - \phi_{2,\alpha} \partial_z \phi_{1,\alpha}  = e^{-2\alpha z}
$$ is  non zero. In addition, we can check that 
\begin{equation}\label{Ray-phia12}
\Ray_\alpha(\phi_{j,\alpha}) = - 2 \alpha (U-c) \partial_z \phi_{j,0} e^{-\alpha z} 
\end{equation}
We are then led to introduce an approximate Green function $G_{R,\alpha}(x,z)$ defined by 
$$
G_{R,\alpha}(x,z) = \left\{ \begin{array}{rrr} (U(x)-c)^{-1} e^{-\alpha (z-x)}  \phi_{1,0}(z) \phi_{2,0}(x), \quad \mbox{if}\quad z>x\\
(U(x)-c)^{-1} e^{-\alpha (z-x)}  \phi_{1,0}(x) \phi_{2,0}(z), \quad \mbox{if}\quad z< x.\end{array}\right.
$$
Again, like $G_{R,0}(x,z)$, the Green function $G_{R,\alpha}(x,z)$ is ``singular'' near $z = z_c$ with two sources 
of singularities: one arising from $1/ (U(x) - c)$ for $x$ near $ z_c$ and the other coming from the $(z - z_c) \log (z- z_c)$ singularity
of $\phi_{2,0}(z)$. By a view of \eqref{Ray-phia12}, it is clear that 
\begin{equation}\label{id-Gxz}
\Ray_\alpha (G_{R,\alpha}(x,z)) = \delta_{x} -2\alpha (U- c) E_{R,\alpha}(x,z),
\end{equation}
for each fixed $x$. Here the error term $E_{R,\alpha}(x,z)$ is defined by  
$$
E_{R,\alpha}(x,z) = 
 \left\{ \begin{array}{rrr}
 (U(x)-c)^{-1} e^{-\alpha (z-x)} \partial_z \phi_{1,0}(z) \phi_{2,0}(x), \quad \mbox{if}\quad z>x\\
 (U(x)-c)^{-1}e^{-\alpha (z-x)}\  \phi_{1,0}(x) \partial_z \phi_{2,0}(z), \quad \mbox{if}\quad z< x.\end{array}\right.
$$
We then introduce an approximate inverse of the operator $\Ray_\alpha$ defined by
\begin{equation}\label{def-RaySa}
RaySolver_\alpha(f)(z) 
:= \int_0^{+\infty} G_{R,\alpha}(x,z) f(x) dx
\end{equation}
and the error remainder 
\begin{equation}\label{def-ErrR}
Err_{R,\alpha}(f)(z) := 2\alpha (U(z) - c) \int_0^{+\infty} E_{R,\alpha}(x,z) f(x) dx
\end{equation}

\begin{lemma}\label{lem-RaySa} Assume that $\Im c \not =0$, and let $p$ be $0,1,$ or $2$. For any $f\in {X^{p,\eta}}$,  with $\alpha<\eta$, the function $RaySolver_\alpha(f)$ is well-defined in $Y^{p+2,\alpha}$, satisfying 
$$ \Ray_\alpha(RaySolver_\alpha(f)) = f + Err_{R,\alpha}(f).$$
Furthermore, there hold  
\begin{equation}\label{est-RaySa}
\| RaySolver_\alpha(f)\|_{Y^{p+2,\alpha}} \le C (1+|\log \Im c|) \|f\|_{{X^{p,\eta}}},
\end{equation}
and 
\begin{equation}\label{est-ErrRa} 
\|Err_{R,\alpha}(f)\|_{Y^{p,\eta}} \le C\alpha  (1+|\log (\Im c)|)  \|f\|_{X^{p,\eta}} ,
\end{equation}
for some universal constant $C$. 
\end{lemma}
\begin{proof}  The proof follows similarly to that of Lemmas \ref{lem-RayS0} and \ref{lem-derRayS0}. Indeed, since $G_{R,\alpha}(x,z)  = e^{-\alpha (z-x)} G_{R,0}(x,z)$, the behavior near the critical layer $z=z_c$ is the same for both the Green function, and hence the proof of \eqref{est-RaySa} and \eqref{est-ErrRa} near the critical layer follows identically from that of Lemmas \ref{lem-RayS0} and \ref{lem-derRayS0}.  

Let us check the right behavior at infinity. Consider the case $p=0$ and assume $\|f \|_{X^{0,\eta}} =1$. Similarly to the estimate \eqref{est-Gr0}, Lemma \ref{lem-defphi012} and the definition of $G_{R,\alpha}$ yield
  $$|G_{R,\alpha}(x,z)| \le  C e^{-\alpha|x-z|} \max\{ (1+x), |x-z_c|^{-1} \}.$$
Hence, by definition, 
$$ |RaySolver_\alpha (f)(z) |\le C e^{-\alpha z} \int_0^\infty e^{\alpha x} e^{-\eta x}\max\{ (1+x), |x-z_c|^{-1} \}\; dx $$ 
which is  bounded by $C(1+|\log \Im c|) e^{-\alpha z}$, upon recalling that $\alpha<\eta$.  This proves the right exponential decay of $RaySolver_\alpha (f)(z)$ at infinity, for all $f \in X^{0,\eta}$.

Next, by definition, we have 
$$\begin{aligned}
Err_{R,\alpha}(f)(z) &= -2\alpha (U(z) - c)  \partial_z \phi_{2,0}(z)   \int_z^\infty  e^{-\alpha (z-x)} \phi_{1,0}(x)\frac{f(x)}{U(x)-c}\; dx  
\\ & \quad - 2\alpha (U(z) - c)  \partial_z \phi_{1,0}(z) \int_0^ze^{-\alpha (z-x)} \phi_{2,0}(x) {f(x) \over U(x) - c} dx .
\end{aligned}$$
Since $f(z), \partial_z \phi_{1,0}(z)$ decay exponentially at infinity, the exponential decay of $Err_{R,\alpha}(f)(z)$ follows directly from the above integral representation. It remains to check the order of singularity near the critical layer. Clearly, for bounded $z$, we have 
$$ |E_{R,\alpha}(x,z) | \le C (1+ |\log (z-z_c)| ) e^{\alpha x} \max \{ 1, |x-z|^{-1}\} . $$
The lemma then follows at once, using the extra factor of $U-c$ in the front of the integral \eqref{def-ErrR} to bound the $\log (z-z_c)$ factor. The estimates for derivatives follow similarly. \end{proof}

\subsection{The exact solver for Rayleigh: $\alpha \ll1$}


We are ready to construct the exact solver for the Rayleigh problem. Precisely, we obtain the following.

\begin{proposition}\label{prop-exactRayS}
 Let $p$ be in $\{0,1,2\}$ and $\eta>0$. Assume that $\Im c \not =0$ and $\alpha |\log \Im c|$ is sufficiently small. Then, there exists an operator $RaySolver_{\alpha,\infty} (\cdot) $ from $X^{p,\eta}$ to $Y^{p+2,\alpha}$ (defined by \eqref{def-exactRayS}) so that 
\begin{equation}\label{eqs-RaySolver}
\begin{aligned}
 \Ray_\alpha (RaySolver_{\alpha,\infty} (f)) &= f.
\end{aligned} \end{equation} 
In addition, there holds 
$$\| RaySolver_{\alpha,\infty}(f)\|_{Y^{p+2,\alpha}} \le C \|f\|_{X^{p,\eta}}(1+|\log (\Im c)|) ,$$
for all $f \in X^{p,\eta}$.
\end{proposition}

\begin{proof} The proof follows by iteration. Let us denote 
$$ S_0(z) : = RaySolver_\alpha(f)(z),\qquad E_0(z): = Err_{R,\alpha}(f)(z).$$
It then follows that $\Ray_\alpha (S_0) (z) = f(z) + E_0(z)$. Inductively, we define
$$ S_n(z): = - RaySolver_\alpha(E_{n-1})(z), \qquad E_n(z): = - Err_{R,\alpha}(E_{n-1})(z) ,$$
for $n\ge 1$. It is then clear that for all $n\ge 1$, 
\begin{equation}\label{eqs-appSn} \Ray_\alpha  \Big( \sum_{k=0}^n S_k(z)\Big) = f(z) + E_n(z).\end{equation}
This leads us to introduce the exact solver for Rayleigh defined by 
\begin{equation}\label{def-exactRayS} RaySolver_{\alpha,\infty} (f) := RaySolver_\alpha(f)(z) - \sum_{n\ge 0} (-1)^n RaySolver_{\alpha} (E_n)(z).\end{equation}
By a view of \eqref{est-ErrRa}, we have 
$$\| E_n \|_\eta = \| (Err_{R,\alpha})^n (f)\|_\eta \le C^n\alpha^n  (1+|\log (\Im c)|)^n  \|f\|_\eta,$$
which implies that $E_n\to 0$ in $X_\eta$ as $n \to \infty$ as long as $\alpha \log \Im c$ is sufficiently small. In addition, by a view of  \eqref{est-RaySa},
$$\|RaySolver_{\alpha} (E_n)\|_{Y^\alpha_2} \le C C^n\alpha^n  (1+|\log (\Im c)|)^n  \|f\|_\eta .$$
This shows that the series
$$\sum_{n\ge 0} (-1)^n RaySolver_{\alpha} (E_n)(z)$$
converges in $Y_2^\alpha$, assuming that $\alpha \log \Im c$ is small.

Next, by taking the limit of $n\to \infty$ in \eqref{eqs-appSn}, the equation \eqref{eqs-RaySolver} holds by definition at least in the distributional sense. The estimates when $z$ is near $z_c$ follow directly from the similar estimates on $RaySolver_\alpha(\cdot)$; see Lemma \ref{lem-RaySa}. The proof of Proposition \ref{prop-exactRayS} is thus complete. 
\end{proof}

\subsection{Exact Rayleigh solutions: $\alpha  \ll1$}\label{sec-exactRayleigh}
We shall construct two independent exact Rayleigh solutions by iteration, starting from the approximate Rayleigh solutions $\phi_{j,\alpha}$ defined as in \eqref{def-phia12}. 

\begin{lemma}\label{lem-exactphija} For $\alpha$ small enough so that $\alpha |\log \Im c| \ll 1$, 
there exist two independent functions $\phi_{Ray,\pm} \in e^{\pm \alpha z}L^\infty$ such that
$$
\Ray_\alpha ( \phi_{Ray,\pm} ) = 0, \qquad W[\phi_{Ray,+},\phi_{Ray,-}](z) = \alpha.
$$
Furthermore, we have the following expansions in $L^\infty$: 
$$
\begin{aligned}
\phi_{Ray,-} (z)&=  e^{-\alpha z} \Big (U-c + O(\alpha )\Big).
\\
\phi_{Ray,+} (z)&=  e^{\alpha z} \mathcal{O}(1),
\end{aligned}$$
as $z\to \infty$. At $z = 0$, there hold
$$ 
\begin{aligned}
\phi_{Ray,-}(0) &= U_0 - c + \alpha (U_+-U_0) ^2  \phi_{2,0}(0) + \mathcal{O}(\alpha(\alpha + |z_c|))
\\
\phi_{Ray,+}(0) &= \alpha  \phi_{2,0}(0) +  \mathcal{O}(\alpha^2)
\end{aligned}$$
with $\phi_{2,0}(0) =  {1 \over {U'_c}}  -  {2U''_c\over {U'_c}^2 } z_c \log z_c + \mathcal{O}(z_c)$. 
 \end{lemma}
\begin{proof} 
Let us start with the decaying solution $\phi_{Ray,-}$, which is now constructed by induction. Let us introduce 
$$ \psi_{0} =  e^{-\alpha z} (U-c), \qquad e_{0} = - 2\alpha (U-c) U' e^{-\alpha z},$$
and inductively for $k \ge 1$, 
$$ \psi_{k} = - RaySolver_\alpha (e_{k-1}), \qquad e_{k} = - Err_{R,\alpha} (e_{k-1}).$$
We also introduce 
$$ \phi_{N} = \sum_{k=0}^N \psi_{k} .$$
By definition, it follows that 
$$ \Ray_\alpha (\phi_{N}) = e_{N}, \qquad \forall ~N\ge 1.$$
%
We observe that $\|e_{0}\|_{\eta+\alpha} \le C \alpha$ and $\|\psi_{0}\|_\alpha \le C$. Inductively for $k\ge 1$, by the estimate \eqref{est-ErrRa}, we have 
$$\| e_{k}\|_{\eta + \alpha}  \le C\alpha  (1+|\log (\Im c)|)  \|e_{k-1}\|_{\eta+\alpha} \le C \alpha(C\alpha  (1+|\log (\Im c)|))^{k-1} ,
$$
and by Lemma \ref{lem-RaySa}, 
$$\| \psi_{k} \|_\alpha \le C(1+|\log (\Im c)|)  \| e_{k-1}\|_{\eta + \alpha} \le (C\alpha  (1+|\log (\Im c)|))^{k} .$$
Thus, for sufficiently small $\alpha$, the series $\phi_{N}$ converges in $X_\alpha$ and the error term $e_{N}\to 0$ in $X_{\eta+\alpha}$. This proves the existence of the exact decaying Rayleigh solution $\phi_{Ray,-}$ in $X_\alpha$, or in $e^{-\alpha z}L^\infty$.

As for the growing solution, we simply define 
$$ \phi_{Ray,+}  =\alpha \phi_{Ray,-}(z) \int_{1/2} ^ z \frac{1}{\phi^2_{Ray,-} (y) }\; dy,$$
for which we have $W[\phi_{Ray,+},\phi_{Ray,-}](z) = \alpha$. By definition, $\phi_{Ray,+} $ solves the Rayleigh equation identically. Next, since $\phi_{Ray,-}(z)$ tends to $e^{-\alpha z} (U_+ - c + \mathcal{O}(\alpha))$, $\phi_{Ray,+} $ is of order $e^{\alpha z}$ as $z \to \infty$.

Finally, at $z=0$, we have 
$$
\begin{aligned}
\psi_1(0) &= - RaySolver_\alpha(e_0) (0) = - \phi_{2,\alpha}(0) \int_0^{+\infty} e^{2\alpha x}\phi_{1,\alpha}(x) {e_0(x) \over U(x) - c} dx 
\\
&=  2 \alpha \phi_{2,0}(0) \int_0^{+\infty}  U' (U-c)dz = \alpha (U_+-U_0) (U_+ + U_0 - 2c) \phi_{2,0}(0)
\\
&  = \alpha (U_+-U_0)^2  (U_+ + U_0 - 2c) \phi_{2,0}(0) +  2\alpha (U_+-U_0) (U_0 - c) \phi_{2,0}(0).
\end{aligned}
$$
From the definition, we have $\phi_{Ray,-}(0) = U_0 -c + \psi_1(0) + \mathcal{O}(\alpha^2)$. 
This proves the lemma, upon using that $U_0 - c = \mathcal{O}(z_c)$. 
\end{proof}

%
%

\subsection{Rayleigh equation: $\alpha \approx1$}

In this subsection, we treat the case when $\alpha \approx 1$ and when the critical layer is finite:
$$z_c \not = \infty.$$  
Precisely, we shall prove the following lemma.

\begin{lemma}\label{lem-mid-alpha} Let $r,R$ be arbitrary positive constants, and let $z_c$ be the critical layer. Assume that $z_c \not =\infty$. For all $\alpha \in [r,R]$, there are two exact solutions $\phi_{1,\alpha}, \phi_{2,\alpha}$ to the Rayleigh equations so that 
$W[\phi_{1,\alpha}, \phi_{2,\alpha}] = 1$, 
$$|\phi_{1,\alpha}(z)| \le C e^{-\alpha z}, \qquad |\phi_{2,\alpha}(z)| \le C e^{\alpha z},$$
for $z$ away from the critical layers, and $\phi_{j,\alpha}$ are both in $Y^{2,\alpha}$
(in particular, both behave as $1+(z-z_c)\log(z-z_c)$ for $z$ near the critical layers $z=z_c$). 
\end{lemma} 

The two exact solutions $\phi_{1,\alpha}, \phi_{2,\alpha}$ to the Rayleigh equation defines the Green function:  
$$
G_{R,\alpha}(x,z) =  \frac{1}{U(x)-c}\left\{ \begin{array}{rrr} \phi_{1,\alpha}(z) \phi_{2,\alpha}(x), \quad \mbox{if}\quad z>x\\
\phi_{1,\alpha}(x) \phi_{2,\alpha}(z), \quad \mbox{if}\quad z< x.\end{array}\right.
$$
and the inverse of the Rayleigh operator 
$$Ray^{-1}_{\alpha} (f) = \int_0^\infty G_{R,\alpha}(x,z) f(x)\; dx .$$ 
This yields the following proposition. 

\begin{proposition}\label{prop-exactRayS-mid}
Let $r,R$ be arbitrary positive constants, and let $z_c$ be the critical layer. Assume that $z_c \not =\infty$ and $\Im c \not =0$. Then, for all $\alpha \in [r,R]$ and for positive $\eta <\alpha$, the inverse $Ray^{-1}_{\alpha}(\cdot)$ of the Rayleigh operator is well-defined from $X^{p,\eta}$ to $Y^{p+2,\eta}$. In addition, there holds 
$$\| Ray^{-1}_{\alpha} (f)\|_{Y^{p+2,\eta}} \le C \|f\|_{X^{p,\eta}}(1+|\log (\Im c)|) ,$$
for all $f \in X^{p,\eta}$, and for $p = 0,1,2$.
\end{proposition}
\begin{proof}
The proof is identical to that of Lemma \ref{lem-RaySa}. 
\end{proof}

\begin{proof}[Proof of Lemma \ref{lem-mid-alpha}]
We start our construction with the following approximate Rayleigh solutions
\begin{equation}\label{def-phia12-mid}
\phi_{1,\alpha,0 } = e^{-\alpha z}  (U-c),\qquad \phi_{2,\alpha,0} = \phi_{1,\alpha,0} \int_{1+\Re z_c }^z \frac{1}{\phi_{1,\alpha,0}^2} \; dy.
\end{equation}
By construction, the Wronskian determinant $W[\phi_{1,\alpha,0},\phi_{2,\alpha,0}] 
=1$. In addition, we can check that 
\begin{equation}\label{Ray-phia12-mid}
\Ray_\alpha(\phi_{j,\alpha,0}) = - 2 \alpha U'(z) \phi_{j,\alpha,0}(z) .
\end{equation}
%
%
We then introduce an approximate Green function $G_{R,\alpha}(x,z)$ defined by 
$$
G_{R,\alpha}(x,z) =  \frac{1}{U(x)-c}\left\{ \begin{array}{rrr} \phi_{1,\alpha,0}(z) \phi_{2,\alpha,0}(x), \quad \mbox{if}\quad z>x\\
\phi_{1,\alpha,0}(x) \phi_{2,\alpha,0}(z), \quad \mbox{if}\quad z< x.\end{array}\right.
$$
for each fixed $x$, and the error term $E_{R,\alpha}(x,z)$ caused by the approximation
$$
E_{R,\alpha}(x,z) = -2\alpha U'(z) G_{R,\alpha}(x,z).
$$

\subsubsection*{Large $z$: $z\ge M$.} 
We shall construct Rayleigh solutions in the region $z\ge M$, for sufficiently large $M$ so that $M\ge 1+\Re z_c$. The solutions are constructed iteratively in the following form:
$$
\widetilde  \phi_{j,\alpha} = \phi_{j,\alpha,0} + e^{(-1)^j \alpha z}\psi_{j,0}$$
in which $\psi_{j,0}$ is defined via the contraction mapping principle 
$$
\begin{aligned}
 \psi_j =  T_j \psi_j: &= 2 \alpha e^{(-1)^{1+j} \alpha z} \int_M^\infty G_{R,\alpha} (x,z) U'(x) e^{(-1)^j \alpha x} \psi_{j,0}(x) \; dx
 \\
 &\quad + 2\alpha \int_M^\infty G_{R,\alpha} (x,z) U'(x) \phi_{j,\alpha,0}(x)\; dx.
 \end{aligned}$$
 We shall prove that the mapping $T_j$ is indeed contractive in $L^\infty(M,\infty)$ for sufficiently large $M$. It then follows by the construction that 
$$ Ray_\alpha(\widetilde  \phi_{j,\alpha} ) =  - 2 \alpha U'(z) e^{(-1)^j \alpha z}\psi_{j,0}. 
$$ 
We shall show that $e^{(-1)^j \alpha z}\psi_{j,0}$ is exponentially small of order $e^{-\eta_0M}$ for sufficiently large $M$. We then correct the previous error by introducing the next order solution $\psi_{j,1} = T_j \psi_{j,0}$, leaving a smaller error of the form $2 \alpha U'(z) e^{(-1)^j \alpha z}\psi_{j,1}$. This iterative construction yields exact Rayleigh solutions defined on $[M,\infty)$.

It remains to show the contractive of $T_j$ in $L^\infty(M,\infty)$. For all $x,z \ge M$, there holds 
$$ | G_{R,\alpha}(x,z) | \le C_0 e^{-\alpha |x-z|}.$$
Hence, for $j=1$, we have 
$$ 
\begin{aligned}
|T_1 \psi - T_1 \widetilde  \psi| 
&\le 
C \alpha \| \psi - \widetilde  \psi \|_{L^\infty} e^{\alpha z}\int_M^\infty  e^{-\alpha |x-z|} e^{-\eta_0 x} e^{-\alpha x} \; dx
\\
&\le 
C \alpha \| \psi - \widetilde  \psi \|_{L^\infty} e^{-\eta_0 z}
\\
&\le 
C \alpha e^{-\eta_0 M}\| \psi - \widetilde  \psi \|_{L^\infty} .\end{aligned}$$
Similarly, for $j=2$, we compute 
$$ 
\begin{aligned}
|T_2 \psi - T_2 \widetilde  \psi| 
&\le 
C \alpha \| \psi - \widetilde  \psi \|_{L^\infty} e^{-\alpha z}\int_M^\infty  e^{-\alpha |x-z|} e^{-\eta_0 x} e^{\alpha x} \; dx
\\
&\le 
C \alpha \| \psi - \widetilde  \psi \|_{L^\infty} z e^{-\eta_0 z}
\\
&\le 
C \alpha M e^{-\eta_0 M}\| \psi - \widetilde  \psi \|_{L^\infty} .\end{aligned}$$
Since $\alpha $ is bounded, these computations prove that $T_j$ are contractive on $L^\infty (M,\infty)$ for sufficiently large $M$. The exact Rayleigh solutions $\phi_j \sim \phi_{j,\alpha}$ for large $z$ follow.

\subsubsection*{Near critical layers: $z \in  [\Re z_c - r, \Re z_c + r]$.} 

For $r$ sufficiently small, we shall construct Rayleigh solutions in $I_c: = [\Re z_c - r, \Re z_c + r]$. As in the previous case, it suffices to show that 
$$  T \psi: = 2 \alpha \int_{I_c} G_{R,\alpha} (x,z) U'(x) \psi (x) \; dx$$
is contractive in $X$ that consists of functions $f(z)$ so that 
$$\| f\|_X : = \sup_{z\in I_c} |f(z)| \Big(1+ |(z-z_c) \log (z-z_c)|\Big)^{-1}$$
is finite. We note that for $x,z\in I_c$, there holds 
$$
|G_{R,\alpha}(x,z) | \lesssim \left\{ \begin{array}{rrr} |x-z_c|^{-1}|z-z_c| \Big( 1+ |(x-z_c) \log (x-z_c)|\Big), \quad \mbox{if}\quad z>x
\\ 1+ |(z-z_c) \log (z-z_c)|, \quad \mbox{if}\quad z< x.\end{array}\right.
$$
For all $\psi, \widetilde  \psi$ in $X$, we compute 
$$ 
\begin{aligned}
|(T \psi - T \widetilde  \psi )(z)| 
&\le 
 \| \psi - \widetilde  \psi \|_X \int_{I_c} |G_{R,\alpha}(x,z) |\Big( 1+ |(x-z_c) \log (x-z_c)|\Big)\; dx.
\end{aligned}$$
For $x >z$, the integral is bounded by 
$$\begin{aligned}
 \Big( &1+ |(z-z_c) \log (z-z_c)|\Big) \int_{I_c}\Big( 1+ |(x-z_c) \log (x-z_c)|\Big)\; dx 
 \\&\le C r \Big( 1+ |(z-z_c) \log (z-z_c)|\Big) .\end{aligned}$$
 Similarly, for $x<z$, the integral is estimated by 
 $$\begin{aligned}
  \int_{I_c} &|x-z_c|^{-1}|z-z_c| \Big( 1+ |(x-z_c) \log (x-z_c)|\Big)\; dx 
 \\&\le C r |(z-z_c) \log (z-z_c)|.\end{aligned}$$
Combining the two estimates yields 
$$ 
\begin{aligned}
|(T \psi - T \widetilde  \psi )(z)| 
&\le 
C r \| \psi - \widetilde  \psi \|_X 
\end{aligned}$$
which proves the contraction of $T$ as a map from $X$ to itself, for sufficiently small $r$. This yields the exact Rayleigh solutions near the critical layer.

\subsubsection*{Away from critical layers: $z\in [0,\Re z_c - r]\cup [\Re z_c + r,M] $.} 
Since the Rayleigh equations are regular away from the critical layers, the previous solutions constructed on $[M,\infty)$ can be extended for all $z\ge \Re z_c + r$, for arbitrarily small positive constant $r$. We then solve the Rayleigh problem on $I_c = [\Re z_c - r, \Re z_c + r]$, having the boundary conditions at $\Re z_c + r$ to be the same as the Rayleigh solutions constructed on $z\ge \Re z_c + r$. Finally, on $[0,\Re z_c - r]$, the Rayleigh equations are regular, and thus the solutions are extended to all $z\ge 0$ by the standard ODE theory.  

This completes the proof of Lemma \ref{lem-mid-alpha}. 
\end{proof}


\subsection{Rayleigh equation: $\alpha \gg1$}\label{sec-Rayleighla}


In this section, we study the Rayleigh equation when $\alpha \gg1$, but $\Im c$ can be arbitrarily close to the imaginary axis as long as $\alpha^{-1} \log \Im c$ remains sufficiently small. It turns out that we can avoid dealing with critical layers as done in the previous section. Precisely, we obtain the following lemma. 

\begin{lemma}\label{lem-large-alpha} For sufficiently large $\alpha$ so that $\alpha^{-1} \log \Im c$ remains sufficiently small, there are two exact solutions $\phi_{\alpha,\pm}$ to the Rayleigh equations so that 
$W[\phi_{\alpha,-}, \phi_{\alpha,+}] = \alpha$, and 
$$\phi_{\alpha,\pm}(z) =  e^{\pm \alpha z} \Big( 1 + \cO(\alpha^{-1}\log \Im c) \Big) .$$
In addition, the Green function $G_{R,\alpha}(x,z)$ of the Rayleigh operator satisfies 
$$ |G_{R,\alpha}(x,z)|\le C \alpha^{-1} |U(x)-c|^{-1} e^{-\alpha |x-z|}\Big( 1 + \cO(\alpha^{-1}\log \Im c) \Big) .$$
\end{lemma} 

\begin{proof} Since $\alpha \gg1$, it is convenient to introduce the change of variables
$$ \widetilde  z = \alpha z, \qquad \phi(z) = \alpha^{-2}\widetilde  \phi (\alpha z).$$
Then, $\widetilde  \phi$ solves the scaled Rayleigh equation
\begin{equation}\label{def-Raysc} \Ray_{sc}(\widetilde  \phi): = (U-c) (\partial_{\widetilde  z}^2 - 1) \widetilde  \phi - \alpha^{-2}U'' \widetilde  \phi =0,\end{equation}
in which $U = U(\alpha^{-1}\tilde z)$ and $U'' = U''(\alpha^{-1}\tilde z)$. 
This way, we can treat the $\alpha^{-2} U''$ as a regular perturbation, and thus the inverse of $\Ray_{sc}(\cdot)$ is constructed by induction, starting from the inverse of $(U-c) (\partial_{\widetilde  z}^2 - 1)$. The iteration operator is defined by  
$$ T \widetilde  \phi : = \alpha^{-2} \Big[ (U-c) (\partial_{\widetilde  z}^2 - 1)\Big]^{-1} \circ U'' \widetilde  \phi .$$ 
It suffices to show that $T$ is contractive with respect to the $e^{\widetilde  z}$ or $e^{-\widetilde  z}$ weighted sup norm, denoted by $\|\cdot \|_{L^\infty_\pm}$, respectively for decaying or growing solutions. Indeed, recalling that $(\partial_{\widetilde  z}^2 - 1)^{-1} f= e^{-|\widetilde  z|} \star f$, we check that  
$$\begin{aligned}
 |T \widetilde  \phi  - T\widetilde  \psi|(z) 
 &\le  \alpha^{-2}  \int_0^\infty  \Big| \frac{U''}{U-c}\Big| e^{-|\widetilde  x - \widetilde  z|} e^{\pm|\widetilde  x|} \| \widetilde  \phi  - \widetilde  \psi\|_{L^\infty_\pm}\; d\widetilde  x
 \\
& \le  \alpha^{-2} e^{\pm|\widetilde  z|}  \| \widetilde  \phi  - \widetilde  \psi\|_{L^\infty_\pm} \int_0^\infty  \Big| \frac{U''(\alpha^{-1}\widetilde  x)}{U(\alpha^{-1}\widetilde  x)-c}\Big| \; d\widetilde  x
 \\
& \le  \alpha^{-1} e^{\pm|\widetilde  z|}  \| \widetilde  \phi  - \widetilde  \psi\|_{L^\infty_\pm} \int_0^\infty  \Big| \frac{U''(x)}{U(x)-c}\Big| \; d x,
 \end{aligned}$$ 
 in which the triangle inequality was used to deduce $e^{-|\tilde x - \tilde z|} e^{\pm|\tilde x|} \le e^{\pm|\tilde z|} $. As for the integral term, when $c$ is away from the range of $U$, it follows that 
$$ \int_0^\infty  \Big| \frac{U''(x)}{U(x)-c}\Big| \; dx \le C_0 \int_0^\infty  |U''|\; dx \le C.$$ 
In the case, when $c$ is near the range of $U$, the integral of $1/(U-c)$ is bounded by $\log(U-c) \lesssim \log \Im c$. This proves that 
$$\begin{aligned}
 |T \tilde \phi  - T\tilde \psi|(z) 
& \le C \alpha^{-1} |\log \Im c| e^{\pm|\tilde z|}  \| \tilde \phi  - \tilde \psi\|_{L^\infty_\pm}.
 \end{aligned}$$

Hence, $T$ is contractive as long as $\alpha^{-1}\log \Im c$ is sufficiently small. Similarly, the same bound holds for derivatives, since the derivative of the Green kernel of $\partial_{\widetilde  z}^2 - 1$ satisfies the same bound. This yields two independent Rayleigh solutions of the form 
\begin{equation}\label{phi0-la} \partial_{\widetilde  z}^k\widetilde  \phi^0_{s,\pm} (\widetilde  z) = (-1)^k e^{\pm \widetilde  z} \Big( 1 + \cO(\alpha^{-1}\log \Im c) \Big), \qquad k=0,1.\end{equation}
The lemma follows. 
\end{proof}

\section{The $4^{th}$-order Airy operator}   
In this section, we study the Airy operator, defined as in \eqref{opAiry}. Let us introduce the following so-called modified Airy operator 
\begin{equation}\label{def-cA}\mathcal{A}: = \eps \Delta_\alpha - (U-c) .\end{equation}
Then, the Airy operator simply reads 
\begin{equation}\label{Airy-de}\Airy = \cA \Delta_\alpha .\end{equation}
That is, the inverse of $\Airy(\cdot)$ is simply a combination of the inverse of $\cA$ and the inverse of $\Delta_\alpha$. We observe that near the critical layer $z = z_c$, the coefficient $U - c$ is approximately $z - z_c$, 
and so $\cA$ can be approximated by the classical Airy operator $\partial_z^2 - z$.

\subsection{The modified Airy operator}\label{sec-mAiry}

This section is devoted to study the modified Airy operator $\cA$, defined in \eqref{def-cA}. Precisely, we shall construct the Green function of $\cA$. The construction involves the classical Airy functions and the Langer's variables. To proceed, let us introduce the following Langer's transformation, which is a joint transformation in the variable $z \to \eta$ and in the function 
$\Phi \to \Psi$: 
\begin{equation}\label{def-Langer}(z,\Phi) \mapsto (\eta,\Psi)\end{equation}
with $\eta = \eta(z)$ defined by
\begin{equation}\label{var-Langer}
\eta (z) = \Big[ \frac 32 \int_{z_c}^z \Big( \frac{U-c+\eps\alpha^2}{U'_c}\Big)^{1/2} \; dz \Big]^{2/3} 
\end{equation}
and $\Psi = \Psi(\eta)$ defined by the relation
\begin{equation}\label{phi-Langer}
 \Phi (z) = \dot z ^{1/2} \Psi(\eta).
 \end{equation}
Here, $
 \dot z = \frac{d z( \eta)}{d \eta} 
 $  and $z = z(\eta)$ is the inverse of the map $\eta = \eta(z)$. It is useful to note that $(U-c+\alpha^2 \epsilon) \dot z^2 = U_c' \eta$. 

 The Langer's transformation links the classical Airy equation to the modified Airy equation. Precisely, we have the following lemma whose proof follows from a direct calculation. 
 
\begin{lemma}[\cite{GGN2,GGN3}] \label{lem-Langer} Let $(z,\Phi) \mapsto (\eta, \Psi)$ be the Langer's transformation defined as in \eqref{def-Langer}. The function $\Psi(\eta)$ solves the classical Airy equation: 
\begin{equation}\label{cl-Airy}\epsilon   \partial^2_\eta \Psi  - U_c' \eta \Psi   = f(\eta)\end{equation}
if and only if the function $\Phi =  \dot z ^{1/2} \Psi(\eta (z))$ solves
\begin{equation}\label{id-mAiry} 
\cA( \Phi) =  \dot z ^{-3/2} f(\eta(z))+ \eps \D_z^2 \dot z^{1/2} \dot z^{-1/2}  \Phi (z) .\end{equation}
\end{lemma}

In view of \eqref{cl-Airy}, let us denote 
\begin{equation}\label{def-dcl}
\delta = \Bigl( { \eps \over U_c'} \Bigr)^{1/3}  = e^{-i \pi / 6} (\alpha R U_c')^{-1/3}
\end{equation}
to be the thickness of the critical layer $z=z_c$. We also introduce the critical layer variables 
\begin{equation}\label{def-XZ}X = \delta^{-1} \eta(x), \qquad Z = \delta^{-1} \eta(z).\end{equation}
where $\eta(z)$ is the Langer's variable defined as in \eqref{var-Langer}.

In this section, we prove the following proposition, providing pointwise estimates on the Green function of the modified Airy operator.

\begin{proposition}\label{prop-exactGrmAiry}
Let $\cA$ be defined as in \eqref{def-cA} and let $\mathcal{G}(x,z)$ be the corresponding Green function to the modified Airy operator $\cA$ on the half-line $z\ge 0$. That is, for each fixed $x$, there holds
\begin{equation}\label{mA-Green}
\cA \mathcal{G}(x,z)  = \delta_x(z) ,  
\end{equation} 
together with the jump conditions across $z=x$:
\begin{equation}\label{GrmAiry-jump}[\mathcal{G}(x,z)]_{\vert_{z=x}} = 0 , \qquad [\epsilon \partial_z\mathcal{G}(x,z)]_{\vert_{z=x}} = 1.\end{equation}
Then, the Green function $\mathcal{G}(x,z)$ exists and satisfies the following pointwise estimates
\begin{equation}\label{mAiry-Greenest}
\begin{aligned}
|\D_z^\ell \D_x^k \mathcal{G}(x,z) | 
&\le 
C \delta^{-2-k-\ell} \langle z\rangle^{(1-k-\ell)/3}\langle Z \rangle^{(k+\ell-1)/2} \mathcal{T}(x,z)  ,
\end{aligned}
\end{equation} 
for all $k,\ell \ge 0$ and all $x,z\ge 0$. Here, $\mathcal{T}(x,z)$ is defined by 
\begin{equation}\label{def-TT}\mathcal{T}(x,z): = \left\{ \begin{aligned} 
e^{\frac{\sqrt2}{3}|X|^{3/2}} e^{-\frac{\sqrt2}{3}|Z|^{3/2}} & \quad \qquad \mbox{if} \quad z>x>\Re z_c,
\\
e^{-\frac{\sqrt2}{3}|X|^{3/2}} e^{-\frac{\sqrt2}{3}|Z|^{3/2}} & \quad \qquad \mbox{if} \quad z>\Re z_c>x,
\\e^{- \frac{\sqrt2}{3}|X|^{3/2}} e^{\frac{\sqrt2}{3}|Z|^{3/2}} & \quad \qquad \mbox{if} \quad \Re z_c>z>x,
\end{aligned}\right.\end{equation}
and $\mathcal{T}(x,z) = \mathcal{T}(z,x)$. In particular, there holds
\begin{equation}\label{Gcl-xnearz}
\begin{aligned}
\mathcal{T}(x,z) \le 
e^  {-  {\sqrt{2} \over 3} \sqrt{|Z|}|X-Z| }  .
\end{aligned}
\end{equation} 
In the above, $\delta$ denotes the critical layer thickness and $X,Z$ the critical layer variables defined as in \eqref{def-dcl} and \eqref{def-XZ}, respectively. 
\end{proposition}

In view of the Langer's transformation and Lemma \ref{lem-Langer}, we recall the following lemma on the classical Airy solutions, which is needed in the proof of Proposition \ref{prop-exactGrmAiry}. 

\begin{lemma}[\cite{Airy,Vallee}]\label{lem-classicalAiry} The classical Airy equation $\Psi'' - z \Psi  =0$
has two independent solutions $Ai(z)$ and 
$Ci(z)$ so that the Wronskian determinant of $Ai$ and $Ci$ equals to one. 
In addition, $Ai(e^{i \pi /6} x)$ and $Ci(e^{i \pi /6} x)$ converge to $0$ as $x\to \pm \infty$
($x$ being real), respectively. Furthermore, there hold asymptotic bounds: 
$$
\Bigl| Ai(k, e^{i \pi / 6} x) \Bigr| \le
 {C| x |^{-k/2-1/4} } e^{-\sqrt{2 | x|} x / 3}, \qquad k\in \mathbb{Z}, \quad x\in \mathbb{R},
 $$ and
 $$\Bigl| Ci(k, e^{i \pi / 6} x) \Bigr| \le
 {C| x |^{-k/2-1/4} } e^{\sqrt{2 | x|} x / 3}, \qquad k\in \mathbb{Z},\quad x\in \mathbb{R},
 $$in which $Ai(0,z) = Ai(z)$, $Ai(k,z) = \partial_z^{-k} Ai(z)$ for $k\le 0$, and $Ai(k,z)$ is the $k^{th}$ primitives of $Ai(z)$, for $k\ge 0$. The functions $Ci(k,z)$ are defined in the same way. 
\end{lemma}

%
%

We first construct a good approximate Green function of $\cA$.

\begin{lemma}\label{lem-GrmAiry}
There is an approximate Green function $G_\mathrm{a}(x,z)$ (defined in \eqref{def-Gcl}) to the modified Airy operator $\cA$ in the following sense: for each $x$, the function $G_\mathrm{a}(x,z)$ solves 
\begin{equation}\label{mA-Ga}
\cA G_\mathrm{a}(x,z)  = \delta_x(z) + E_\mathrm{a}(x,z),  
\end{equation} 
for some remainder term $E_\mathrm{a}(x,z)$ (defined in \eqref{eqs-Gcl}), together with the jump conditions \eqref{GrmAiry-jump}. 
In addition, for $k,\ell\ge 0$ and $x,z\ge 0$, there hold the following uniform estimates
$$
\begin{aligned}
|\D_x^\ell \D_z^k G_\mathrm{a}(x,z) | 
&\le 
C \delta^{-2-k-\ell} \langle x\rangle^{(1-\ell)/3} \langle z \rangle^{-k/3} \langle X\rangle^{(2\ell-1)/4} \langle Z \rangle^{(2k-1)/4} \mathcal{T}(x,z),
\\
|\D_x^\ell \D_z^k E_\mathrm{a}(x,z) | 
&\le 
C \delta^{-k-\ell} \langle x\rangle^{(1-\ell)/3}\langle z \rangle^{(-4-k)/3} \langle X\rangle^{(2\ell-1)/4}\langle Z \rangle^{(2k+1)/4} \mathcal{T}(x,z),
\end{aligned}
$$
in which $\mathcal{T}(x,z)$ is defined as in \eqref{def-TT}. In particular, $$\mathcal{T}(x,z) \le 
e^  {-  {\sqrt{2} \over 3} \sqrt{|Z|}|X-Z| } .$$ 
\end{lemma}
\begin{proof} Using Lemma \ref{lem-classicalAiry}, we can define an approximate Green function for the Airy equation:
\begin{equation}\label{def-Gcl} 
G_\mathrm{a}(x,z) =   \delta \eps^{-1}  \dot x\left\{ \begin{array}{rrr}  Ai(X)Ci(Z), \qquad &\mbox{if}\qquad x>z,\\
 Ai (Z) Ci(X) , \qquad &\mbox{if}\qquad x<z,
\end{array}\right.
\end{equation}with $\dot x = 1/\eta'(x)$. It is clear that $[G_\mathrm{a}(x,z)]_{\vert_{z=x}} = 0$. Next, recalling that the Wronskian determinant of $Ai$ and $Ci$ is equal to one, $\dot x = 1/\eta'$, and $\frac{d}{dz}Z = \delta^{-1} \eta'(z)$, we have 
$$[\partial_zG_\mathrm{a}(x,z)]_{\vert_{z=x}} =  \delta  \eps^{-1}  \dot x \frac{d}{dz}Z \Big[ Ai(Z) Ci'(Z) - Ai'(Z) Ci(Z) \Big]= \epsilon^{-1}.$$
That is, the jump conditions in \eqref{GrmAiry-jump} hold for $G_\mathrm{a}(x,z)$. In addition, for $x>z$, we compute 
$$ 
\begin{aligned}
\eps \partial_z^2 G_\mathrm{a}(x,z) 
&= \delta^{-1}\dot x (\eta'(z))^2 Ai(X) Ci''(Z) + \dot x \eta''(z) Ai(X) Ci'(Z) 
\\
&=\delta^{-1} \dot x (\eta'(z))^2 Z Ai(X) Ci(Z) + \dot x \eta''(z) Ai(X) Ci'(Z) 
\\
&= \epsilon \delta^{-2} (\eta'(z))^2 Z G_\mathrm{a}(x,z)+ \dot x \eta''(z) Ai(X) Ci'(Z) .
\end{aligned}
$$
By definition, we have $ \epsilon = \delta^3 U_c'$ and $(U-c+\alpha^2 \epsilon) \dot z^2 = U_c' \eta  = \delta U_c' Z$.  This proves 
$$ 
\begin{aligned}
\eps \Delta_\alpha G_\mathrm{a}(x,z) 
&= (U-c) G_\mathrm{a}(x,z)+ \dot x \eta''(z) Ai(X) Ci'(Z),
\end{aligned}
$$
for $x>z$. A similar calculation holds for $x<z$, yielding 
\begin{equation}\label{eqs-Gcl}
\eps \Delta_\alpha G_\mathrm{a}(x,z)  - (U-c) G_\mathrm{a}(x,z)  = \delta_x(z) + E_\mathrm{a}(x,z),  
\end{equation} 
with 
\begin{equation}\label{def-EEEa}
E_\mathrm{a}(x,z) = 
\eta''(z) \dot x\left\{ \begin{array}{rrr}  Ai(X)Ci'(Z), \qquad &\mbox{if}\qquad x>z,\\
 Ai'(Z) Ci(X) , \qquad &\mbox{if}\qquad x<z.
\end{array}\right.
\end{equation}

It remains to derive the pointwise estimates as stated in the proposition. It suffices to consider the case: $x<z$. By the estimates on the Airy functions obtained from Lemma \ref{lem-classicalAiry}, we get 
\begin{equation}\label{asymp-bd}
\begin{aligned}
| \D_z^k  &Ai(e^{i \pi / 6} z)  \D_x^\ell Ci(e^{i \pi / 6} x) | 
\\&\le 
{C|z|^{k/2-1/4} |x|^{\ell/2-1/4}} 
\exp \Bigl( {1\over 3} \sqrt{2 |x|} x -  {1\over 3} \sqrt{2 |z|} z \Bigr) ,
\end{aligned}
\end{equation} for $k \ge 0$, and for $x,z$ bounded away from zero. In particular, as $z > x$,
$$
\exp \Bigl( {1\over 3} \sqrt{2 |x|} x -  {1\over 3} \sqrt{2 |z|} z \Bigr) \le
\exp \Bigl( {\sqrt{2} \over 3} \sqrt{|z|} | x - z | \Bigr) .
$$
We remark that the polynomial growth in $x$ in the above estimate can be replaced by the growth in $z$, 
up to an exponentially decaying term. Similar bounds can be obtained for the case $x>z$.

The lemma thus follows directly from \eqref{asymp-bd}, upon noting that the pre-factor in terms of the lower case $z$ is due to the Langer's change of variables, with 
$$ |\eta(z)| \sim  \langle z \rangle^{2/3}, \qquad |\eta'(z)| \sim  \langle z \rangle^{-1/3} $$
with $\langle z \rangle = \sqrt{1+z^2}$, for some universal constant $C$. 
\end{proof}

\begin{lemma}  \label{lem-ConvAiry-cl}  Let $G_\mathrm{a}(x,z)$ be the approximate Green function defined as in \eqref{def-Gcl}, and $E_a(x,z)$ as defined in \eqref{eqs-Gcl}. Also let $f\in X_\eta$, for some $\eta>0$.  
Then, for $k\ge 0$, there is some constant $C$ so that 
\begin{equation}\label{conv-loc1}\begin{aligned}
\Big\|\int_{0}^{\infty}  \partial_z^kG_\mathrm{a}(x,\cdot)  f(x)   dx \Big\|_\eta &\le  C \delta^{-1-k}  \|f\|_\eta 
\end{aligned}\end{equation}
and \begin{equation}\label{conv-nonloc}\begin{aligned}
\Big \|\int_{0}^{\infty} \partial_z^kE_\mathrm{a}(x,\cdot)  f(x)   dx\Big \|_\eta &\le C \delta^{1-k}  \|f\|_\eta .
\end{aligned}\end{equation}  
\end{lemma} 


\begin{proof} Without loss of generality, we assume $\|f \|_\eta =1$. Using the bounds from Lemma \ref{lem-GrmAiry}, and noting that $Z = \eta(z)/\delta \approx (1+|z|)^{2/3}/\delta$ as $z$ becomes large, for $k=0,1,2$, we obtain
$$\begin{aligned}
\int_{0}^{\infty} &|  \partial_z^kG_\mathrm{a}(x,z) f(x)|   dx
\\ & \le C \delta^{-2-k} \int_0^\infty (1+ z)^{(1-k)/3} e^{-\eta z} \langle Z \rangle^{(k-1)/2} \mathcal{T}(x,z) \; dx 
\\&\le C \delta^{-2-k} (1+ z)^{(1-k)/3} e^{-\eta z} \langle Z \rangle^{(k-1)/2}  \int_0^\infty (1+x)^{-1/3} e^{-  {\sqrt{2} \over 3} \sqrt{|Z|}|X-Z| }   \delta dX 
 \\&\le C \delta^{-1-k} (1+ z)^{(1-k)/3} \langle Z \rangle^{(k-2)/2} e^{-\eta |z|}
  \\&\le C \delta^{-1-k}  e^{-\eta |z|},
\end{aligned}$$
upon noting that $(1+ z)^{(1-k)/3} \langle Z \rangle^{(k-2)/2} \lesssim 1$. In the above, we have used the change of variable $dx = \delta \dot z^{-1} dX$ with $\dot z \approx (1+|x|)^{1/3}$. For higher derivatives, we use the equation \eqref{eqs-Gcl} for $\epsilon \partial_z^2 G_\mathrm{a}(x,z)$. Similar estimates hold for $E_\mathrm{a}(x,z)$. This completes the proof of the lemma.
\end{proof}

\begin{lemma} \label{lem-locConvAiry}  Let $G_\mathrm{a}(x,z), E_a(x,z)$ be defined as in \eqref{def-Gcl} and \eqref{eqs-Gcl}, and let $f = f(X)$ satisfy $|f(X)|\le  C_f e^{- \theta_0 |X|^{3/2}} $, for some positive constant $\theta_0<\frac{\sqrt 2}{3}$. 
Then there is some constant $C$ so that 
$$
\begin{aligned}
\int_{0}^\infty |\partial_z^k G_\mathrm{a}(x,z)   f(X)|   dx &\le  C C_f \delta^{-1-k} \langle z\rangle^{(1-k)/3} \langle Z \rangle^{(k-2)/2}e^{-\theta_0 |Z|^{3/2}},
\\
\int_{0}^\infty | \partial_z^kE_\mathrm{a}(x,z)  f(X)|   dx &\le  C C_f  \delta^{1-k} \langle z\rangle^{(-3-k)/3} \langle Z \rangle^{(k-2)/2}e^{-\theta_0 |Z|^{3/2}},
\end{aligned}
$$ for $k \ge 0$. 
\end{lemma} 
\begin{proof} Lemma \ref{lem-GrmAiry} yields 
$$\begin{aligned}
&\int_{0}^{\infty} |  \partial_z^kG_\mathrm{a}(x,z) f(X)|   dx
\\&\le C_0C_f\delta^{-2-k} \langle z\rangle^{(1-k)/3} \langle Z \rangle^{(k-1)/2}  \int_0^\infty  \mathcal{T}(x,z) e^{-\theta_0 |X|^{3/2}} \; dx
\\&\le C_0C_f\delta^{-1-k} \langle z\rangle^{(1-k)/3} \langle Z \rangle^{(k-1)/2}  \int_0^\infty  \mathcal{T}(x,z) e^{-\theta_0 |X|^{3/2}} \; dX,
\end{aligned}$$ in which $\mathcal{T}(x,z)$ is defined as in \eqref{def-TT}. 
We then estimate the above integral by considering two cases $|X|\le |Z|$ and $|X|\ge |Z|$. Precisely, we have 
$$\begin{aligned}
\int_{\{ |X|\ge |Z|\}}  \mathcal{T}(x,z) e^{-\theta_0 |X|^{3/2}} \; dX
&\le  e^{-\theta_0 |Z|^{3/2}} \int_0^\infty e^{-  {\sqrt{2} \over 3} \sqrt{|Z|}|X-Z| } \; dX 
\\& \lesssim  \langle Z \rangle^{-1/2}  e^{- \theta_0 |Z|^{3/2}}.\end{aligned}$$  
As for the case when $|X|\le |Z|$, we have $$\begin{aligned}
\int_{\{ |X|\le |Z|\}}  \mathcal{T}(x,z)e^{-\theta_0 |X|^{3/2}} \; dX
&\le  \int_0^{|Z|}  e^{-  {\sqrt{2} \over 3} |Z|^{3/2}} e^{{\sqrt{2} \over 3}|X|^{3/2} }  e^{-\theta_0 |X|^{3/2}} \; dX 
\\&\le  \int_0^{|Z|}  e^{-  {\sqrt{2} \over 3} |Z|^{3/2}} e^{({\sqrt{2} \over 3} - \theta_0)\sqrt{|Z|} |X| }  \; dX 
\\& \lesssim  \langle Z \rangle^{-1/2}  e^{- \theta_0 |Z|^{3/2}}.\end{aligned}$$  
The estimates for $E_\mathrm{a}(x,z) $ follow identically. \end{proof}

Finally, we end the section by proving Proposition \ref{prop-exactGrmAiry}. 
\begin{proof}[Proof of Proposition \ref{prop-exactGrmAiry}] The proof follows from the standard iteration from the approximate Green function. Indeed, let $x$ be fixed and let us introduce 
$$ G_0(x,z) = G_\mathrm{a} (x,z), \qquad E_0(x,z) = E_\mathrm{a}(x,z). $$
Lemma \ref{lem-GrmAiry} yields 
$$ \cA G_0(x,z) = \delta_x(z) + E_0(x,z).$$
By induction, for $n\ge 1$, we introduce 
$$
\begin{aligned}
 G_n(x,z) &=  - G_\mathrm{a}(x,\cdot) \star E_{n-1}(x,\cdot) (z), 
 \\
 E_n(x,z) &= - E_\mathrm{a}(x,\cdot) \star E_{n-1}(x,\cdot) (z),
 \end{aligned}$$
in which $\star$ denotes the usual convolution (in $z$). By construction, $ G_n(x,z)$ and their first derivatives are smooth, and 
the sum 
$$\mathcal{G}_n (x,z)= \sum_{k=0}^n G_n(x,z)$$
solves 
$$ \cA \mathcal{G}_n(x,z) = \delta_x(z) + E_n(x,z).$$
By Lemma \ref{lem-GrmAiry}, there holds 
$$
\begin{aligned}
|E_0(x,z) |  
\le C_0 e^{-  {\sqrt{2} \over 3} \sqrt{|Z|}|X-Z| } . 
\end{aligned}
$$
The convolution estimates from Lemma \ref{lem-locConvAiry}, and the inductive argument, then yield  
$$
\begin{aligned}
|E_n(x,z)|
\le (C_0\delta)^n e^{-  \theta_0\sqrt{|Z|}|X-Z| } 
\end{aligned}
$$ for some positive constant $\theta_0$. 
Since $\delta \ll1$, as $n\to \infty$, $E_n(x,z) $ converges pointwise to zero, and hence the series $\mathcal{G}_n(x,z)$ converges pointwise to $\mathcal{G}(x,z)$.  The proposition follows. 
\end{proof}


\subsection{An approximate Green kernel for Airy operator}


In order to study the smoothing effect of the Airy operator, we need more information on the Green function of the Airy operator. 
Recalling that $G_\mathrm{a}(x,z)$ the Green function of the modified Airy operator $\cA$, defined as in \eqref{eqs-Gcl}, 
we then introduce $G(x,z)$ so that 
\begin{equation} \label{def-doubleG}
\Delta_\alpha G(x,z) = G_\mathrm{a}(x,z) . 
\end{equation} 
By construction, we have \begin{equation}\label{eqs-2GAiry}
\Airy (G(x,z) ) = \delta_x(z) + E_\mathrm{a}(x,z),  
\end{equation}
with the same remainder $E_\mathrm{a}(x,z)$ as in \eqref{mA-Ga} for the modified Airy equation. Using \eqref{eqs-2GAiry} we see that $G$ satisfies the following jump conditions at $z=x$: 
\begin{equation}\label{def-jumpG0}
\begin{aligned}
~[G(x,z)]_{\vert_{x=z}} =  [\partial_z G(x,z)]_{\vert_{x=z}}  = [\partial_z^2 G(x,z)]_{\vert_{x=z}}  =0
\end{aligned}\end{equation}
and 
\begin{equation}\label{def-jumpG1}
\begin{aligned}
~[\epsilon \partial_z^3G(x,z)]_{\vert_{x=z}} =  1.
\end{aligned}\end{equation}
This yields an approximate Green function for the Airy operator. Certainly, we might take $G(x,z) = \Delta_\alpha^{-1} G_{a}(x,z)$,
in which $\Delta_\alpha^{-1}$ is defined through the convolution with its Green function $\alpha^{-1} e^{-\alpha |x-z|} $. This however will not work due to the fact that the oscillation of the Laplacian is of order $\alpha^{-1}$, which is much greater than the localization of the singularity of order $\delta \ll \alpha^{-1}$ near the critical layer $z=z_c$. Instead, we solve \eqref{def-doubleG} by integrating with respect to $z$, up to linear terms in $z$ which we shall correct below. For sake of presentation, we introduce 
$$ \widetilde Ci(1,z) =\delta^{-1} \int_0^z e^{-\alpha y} Ci(\delta^{-1}\eta(y))\; dy, \qquad \widetilde Ci(2,z) = \delta^{-1}\int_0^z  \widetilde Ci(1,y)\; dy$$
and 
$$ \widetilde Ai(1,z) = \delta^{-1}\int_\infty^z e^{-\alpha y} Ai(\delta^{-1}\eta(y))\; dy, 
\qquad \widetilde Ai(2,z) =\delta^{-1} \int_\infty^z \widetilde Ai(1,y)\; dy.$$
Then, in view of \eqref{def-Gcl}, we set 
\begin{equation}\label{def-tG}
\widetilde G(x,z) =  \delta^3 \epsilon^{-1}  \dot x^{3/2} e^{\alpha z} 
\left\{ \begin{array}{rrr}  Ai(\delta^{-1}\eta(x)) \widetilde Ci(2,z) , &\mbox{if }x>z,\\
Ci(\delta^{-1}\eta(x)) \widetilde Ai(2,z) ,  &\mbox{if } x<z
\end{array} \right. 
\end{equation} 
in which the factor $\dot x ^{3/2}$ was added simply to normalize the jump of $G(x,z)$. It follows from direct computations that $\widetilde G(x,z)$ solves
 \eqref{def-doubleG}.

 It remains to correct the jump conditions at $z=x$. The jump conditions on $\partial_z^2G(x,z)$ 
 and $\partial_z^3 G(x,z)$ follow from those of $G_\mathrm{a}(x,z)$. To correct the first two jump conditions in \eqref{def-jumpG0}, 
 we introduce 
\begin{equation}
\label{def-Gr}
G(x,z) = \widetilde G(x,z) + E(x,z) \end{equation}
in which $E(x,z)$ is of the form 
\begin{equation}\label{def-E}
E(x,z) =  \delta^3 \epsilon^{-1}  \dot x^{3/2}  e^{\alpha z} \left\{ \begin{array}{rrr}  
\delta^{-1} a_1 (x) (z-x) +  a_2(x) , &\mbox{if }x>z,\\
 0,  &\mbox{if } x<z.
\end{array} \right.
\end{equation}
Here, the coefficients $a_1(x), a_2(x)$ are defined by the jump conditions, which lead to
\begin{equation}\label{def-ta12}
\begin{aligned}   a_1 (x)  &=
 Ci(\delta^{-1}\eta(x)) \widetilde Ai(1,x) - Ai(\delta^{-1}\eta(x)) \widetilde Ci(1,x)  ,\\
a_2(x)  
&=  Ci(\delta^{-1}\eta(x))\widetilde Ai(2,x) - Ai(\delta^{-1}\eta(x)) \widetilde Ci(2,x)  .
\end{aligned}
\end{equation}
Then, it is clear that $G(x,z)$ solves \eqref{def-doubleG} and satisfies the jump conditions \eqref{def-jumpG0} and \eqref{def-jumpG1}. 
Thus, $G(x,z)$ is an approximate Green function of $\Airy$. Let us give some bounds on the approximate Green function, 
using the known bounds on $Ai(\cdot)$ and $Ci(\cdot)$.

We have the following proposition. 

\begin{proposition} \label{prop-ptGreenbound} 
Let $G(x,z) = \widetilde G(x,z) + E(x,z)$ be the approximate Green function defined as in \eqref{def-Gr}, 
and let $X = \eta(x)/\delta$ and $Z = \eta(z)/\delta$. There hold pointwise estimates
\begin{equation}\label{mG-xnearz}
\begin{aligned}
|\D_z^\ell \D_x^k \widetilde G(x,z) | 
\le 
C \delta^{-k-\ell} \langle z \rangle ^{(4-k-\ell)/3}\langle Z \rangle^{(k+\ell-3)/2}
\mathcal{T}(x,z) ,
\end{aligned}
\end{equation} 
with $\mathcal{T}(x,z)$ defined as in \eqref{def-TT}. In particular, $\mathcal{T}(x,z) \le e^  {-  {\sqrt{2} \over 3} \sqrt{|Z|}|X-Z| } $. 
Similarly, for the non-localized term, we have 
\begin{equation}\label{non-locE}
\begin{aligned}
|E(x,z) | &\le   C e^{-\alpha |x-z|}\Big[\langle x\rangle^{4/3} \langle X\rangle^{-3/2} 
+ \delta^{-1}\langle x\rangle^{1/3}\langle X \rangle^{-1}|x-z| \Big],
\\
|\partial_zE(x,z) | &\le  C\delta^{-1}\langle x \rangle^{1/3} e^{-\alpha |x-z|}\langle X \rangle^{-1} 
\end{aligned}\end{equation}
for $x>z$. In addition, $\partial_z^2 E(x,z) =0$.
\end{proposition}
\begin{proof} We recall that for $z$ near $z_c$, we can write $ \dot z (\eta(z) ) = 1 + \cO(|z-z_c|),$ which in particular yields that 
$$
\frac 12 \le \dot z(\eta(z)) \le \frac 32 
$$ 
for $z$ sufficiently near $z_c$. In addition, by a view of the definition \eqref{var-Langer}, $\eta(z)$ grows like  $ (1+|z|)^{2/3}$ as $z\to \infty$. 

Let us assume that $z\gg 1$. It suffices to give estimates on $\widetilde Ai(k,z), \widetilde Ci(k,z)$. With notation $Y = \eta(y)/|\delta|$ and the fact that $\eta'(y) \sim \langle y \rangle^{-1/3}$, we have 
$$\begin{aligned}
 |\widetilde Ai(1,z) | &\le  \delta^{-1}\int_z^\infty | e^{-\alpha y}Ai(e^{i\pi/6}Y) |\; dy 
 \\&\le  C\delta^{-1}\int_z^\infty e^{-\alpha y} (1+|Y|)^{-1/4} e^{-\sqrt{2|Y|} Y/3} \; dy 
 \\& \le  C\int_Z^\infty e^{-\alpha y} (1+|Y|)^{-1/4} e^{-\sqrt{2|Y|} Y/3} \; (1+y)^{1/3}dY 
 \\& \le  C\int_Z^\infty (1+y)^{1/3} e^{-\alpha y} (1+|Y|)^{-1/4} e^{-\sqrt{2|Z|} Y/3} \; dY 
 \\& \le  C(1+z)^{1/3}e^{-\alpha z} (1+|Z|)^{-3/4} e^{-\sqrt{2|Z|} Z/3}  ,
  \end{aligned}
$$
and 
$$\begin{aligned}
 |\widetilde Ai(2,z) | &\le  \delta^{-1}\int_z^\infty |\widetilde Ai(1,y) |\; dy 
 \\&\le  C\delta^{-1}\int_z^\infty (1+y)^{1/3}e^{-\alpha y} (1+|Y|)^{-3/4} e^{-\sqrt{2|Y|} Y/3} \; dy 
 \\& \le  C\int_Z^\infty (1+y)^{2/3} e^{-\alpha y} (1+|Y|)^{-3/4} e^{-\sqrt{2|Z|} Y/3} \; dY 
 \\& \le  C(1+z)^{2/3} e^{-\alpha z} (1+|Z|)^{-5/4} e^{-\sqrt{2|Z|} Z/3}  .
  \end{aligned}
$$
Similarly, we have 
$$ \begin{aligned}
|\widetilde Ci(1,z)| &\le \delta^{-1} \int_0^z e^{- \alpha y} |Ci(e^{i\pi/6}Y)|\; dy
\\&\le C\delta^{-1}\int_0^z e^{-\alpha y} (1+|Y|)^{-1/4} e^{\sqrt{2|Y|} Y/3} \; dy
\\& \le C\int_0^z (1+z)^{1/3} e^{-\alpha z} (1+|Y|)^{-1/4} e^{\sqrt{2|Y|} Y/3} \; dY
\\& \le C (1+z)^{1/3}e^{-\alpha z} (1+|Z|)^{-3/4} e^{\sqrt{2|Z|} Z/3} 
\end{aligned}$$
and 
$$ \begin{aligned}
|\widetilde Ci(2,z)| &\le \delta^{-1} \int_0^z |\widetilde Ci(1,y)|\; dy
\\&\le C\int_0^z (1+y)^{2/3} e^{-\alpha y} (1+|Y|)^{-3/4} e^{\sqrt{2|Y|} Y/3} \; dY
\\& \le C (1+z)^{2/3} e^{-\alpha z}(1+|Z|)^{-5/4} e^{\sqrt{2|Z|} Z/3}. 
\end{aligned}$$

In the case that $z\lesssim 1$, the above estimates remain valid. Indeed, first consider the case that $z\ge \R z_c$.
 In this case, we still have $|Y|\ge |Z|$ whenever $y\ge z$, and so the estimates on $\widetilde Ai(k,z)$ follow in the same way as done above. Whereas, in the case that $z\le \Re z_c$, we write 
 $$\begin{aligned}
\widetilde Ai(1,z)  =  - \delta^{-1} \Big[ \int_z^{\Re z_c} + \int_{\Re z_c}^\infty\Big]  e^{-\alpha y}Ai(e^{i\pi/6}Y) \; dy 
  \end{aligned}
$$
in which the last integral is equal to $\widetilde Ai(1,\Re z_c)$ and is in particular bounded. As for the first integral, we compute   
 $$\begin{aligned}
& \delta^{-1} \int_z^{\Re z_c} | e^{-\alpha y}Ai(e^{i\pi/6}Y) |\; dy 
 \\&\le  C\delta^{-1}\int_z^{z_c} e^{-\alpha y} (1+|Y|)^{-1/4} e^{-\sqrt{2|Y|} Y/3} \; dy .
  \end{aligned}
$$
Since $z\le y\le \Re z_c$, we have $|Y|\le |Z|$ and hence  
 $$\begin{aligned}
 \delta^{-1} \int_z^{\Re z_c} | e^{-\alpha y}Ai(e^{i\pi/6}Y) |\; dy 
 &\le  C\delta^{-1}\int_z^{z_c}  (1+|Y|)^{-1/4} e^{\frac{\sqrt2}{3} |Y|^{3/2}} \; dy 
\\
 &\le  C\int_0^{|Z|}  (1+|Y|)^{-1/4} e^{\frac{\sqrt2}{3}\sqrt{|Z|} |Y|} \; dY
  \\&\le C \langle Z\rangle^{-3/4} e^{\sqrt{2} |Z|^{3/2}/3}.
  \end{aligned}
$$
Similar computations can be done for $|\widetilde Ai(2,z) |$, yielding 
$$\begin{aligned}
 |\widetilde Ai(1,z) | \le \langle Z\rangle^{-3/4} e^{\sqrt{2} |Z|^{3/2}/3} , \qquad  |\widetilde Ai(2,z) | \le  C\langle Z\rangle^{-5/4} e^{\sqrt{2}|Z|^{3/2} /3}  .
  \end{aligned}
$$
That is, like $Ai(\eta(z)/\delta)$, the functions $\widetilde Ai(k,z)$ grow exponentially fast as $z$ is increasingly away from the critical layer and $z\le \Re z_c$.  Similarly, we also have 
$$ \begin{aligned}
|\widetilde Ci(1,z)| \le C\langle Z\rangle^{-3/4} e^{ - \sqrt{2}|Z|^{3/2}/3}, \qquad 
|\widetilde Ci(2,z)| \le C \langle Z\rangle^{-5/4}e^{ - \sqrt{2}|Z|^{3/2}/3},
\end{aligned}$$
for $z\le \R z_c$. The estimates become significant when the critical layer is away from the boundary layer, that is when $\delta \ll |z_c|$. 

By combining together these bounds and those on $Ai(\cdot)$, $Ci(\cdot)$, the claimed bounds on $\widetilde G(x,z)$ follow. Derivative bounds are also obtained in the same way. 
Finally, using the above bounds on $\widetilde Ai(k,z)$ and $\widetilde Ci(k,z)$, we get 
\begin{equation}\label{ak12-bound}\begin{aligned}  
 |\partial_x^k a_1 (x)|  &\le C\delta^{-k}\langle x\rangle^{1/2-k/3} e^{-\alpha x} \langle X\rangle^{k/2-1}
 \\ 
 |\partial_x^ka_2(x)| &\le    C\delta^{-k}\langle x\rangle^{5/6-k/3}  e^{-\alpha x}\langle X\rangle^{k/2-3/2} ,
\end{aligned}
\end{equation}
upon noting that the exponents in $Ai(\cdot)$ and $Ci(\cdot)$ are cancelled out identically. 

This completes the proof of the lemma. 
\end{proof}

%

\subsection{Convolution estimates}

\begin{lemma}  \label{lem-ConvAiry} Let $G(x,z)= \widetilde G(x,z) + E(x,z)$ be the approximate Green function of $\Airy$, and let $f\in X_\eta$, $\eta>0$.  
Then there is some constant $C$ so that 
\begin{equation}\label{conv-loc}\begin{aligned}
\Big\| \int_{0}^{\infty} \partial_z^k\widetilde G(x,\cdot)  f(x)    dx \Big\|_{\eta'} &\le  \frac {C \delta^{1-k} }{\eta - \eta'}  \|f\|_\eta ,
\end{aligned}\end{equation}
and \begin{equation}\label{conv-nonloc}\begin{aligned}
\Big\|\int_{0}^{\infty} \partial_z^kE(x,\cdot )  f(x)   dx \Big \|_{\eta'}&\le \frac {C\delta^{-k}}{\eta - \eta'}\|f\|_\eta ,
\end{aligned}\end{equation} for $k = 0,1,2$ and for $\eta'<\eta$, with noting that $\partial_z^2 E(x,z)\equiv 0$.  
\end{lemma} 


\begin{proof} Without loss of generality, we assume $\|f \|_\eta =1$. First, consider the case $|z|\le 1$. Using the pointwise bounds obtained in Proposition \ref{prop-ptGreenbound}, we have 
$$\begin{aligned}
\int_{0}^{\infty} |  \widetilde G(x,z) f(x)|   dx &\le C_0 \int_0^\infty  e^{-  {\sqrt{2} \over 3} \sqrt{|Z|}|X-Z| } e^{-\eta x}\; dx \le C \delta  ,
\end{aligned}$$
upon noting that $dx = \delta \dot z^{-1}(\eta (x)) dX$ with $\dot z(\eta(x)) \approx (1+|x|)^{1/3}$. Here the growth of $\dot z(\eta(x))$ in $x$ is clearly controlled by $e^{ -\eta x } $. Similarly, since $|E(x,z)|\le C(1+x)^{4/3}$, we have 
$$\begin{aligned}
\int_z^\infty | E(x,z) f(x)|   dx &\le C \int_z^\infty (1+x)^{4/3} e^{-\eta x} dx\le C   ,
\end{aligned}$$
which proves the estimates for $|z|\le 1$. Next, consider the case $z\ge 1$, and $k = 0,1,2$. Again using the bounds from Proposition \ref{prop-ptGreenbound} and noting that $Z = \eta(z)/\delta \approx (1+|z|)^{2/3}/\delta$ as $z$ becomes large, we obtain 
$$\begin{aligned}
\int_{0}^{\infty} |  \partial_z^k\widetilde G(x,z) f(x)|   dx
 &\le C \delta^{-k}\int_0^\infty  (1+z)^{(4-k)/3}  e^{-\eta x}e^{-  {\sqrt{2} \over 3} \sqrt{|Z|}|X-Z| } \; dx 
\\&\le C \delta^{-k}  (1+z)^{1-k/3}e^{-\eta z}   \int_0^\infty e^{-  {\sqrt{2} \over 3} \sqrt{|Z|}|X-Z| }   \delta dX 
\\&\le C\delta^{1-k}  (1+z)e^{-\eta |z|} .
\end{aligned}$$
Here again we have used the change of variable $dx = \delta \dot z^{-1} dX$ with $\dot z \approx (1+|x|)^{1/3}$. 

Let us now consider the nonlocal term $E(x,z)$, which is nonzero for $x>z$, and consider the case $z\ge 1$. Note that $|X|\ge |Z|$ and so we have from \eqref{non-locE} 
$$
|E(x,z) |\le   C(1+x)^{4/3} (1+|X|)^{-3/2} + C(1+x)^{1/3}.
$$
In addition, we note that $X = \eta(x)/\delta \approx (1+|x|)^{2/3}/\delta$, and so the first term is bounded by the second term in $E(x,z)$. We thus have 
$$\begin{aligned}
\int_{z}^\infty | E(x,z) f(x)|   dx &\le C \int_{z}^\infty (1+x)^{1/3}e^{-\eta|x|} dx 
\\&\le C  (1+z)^{1/3} e^{-\eta |z|} .
\end{aligned}$$
Also, we have 
$$\begin{aligned}
\int_{z}^\infty | \partial_z E(x,z) f(x)|   dx &\le C \delta^{-1}\int_{z}^\infty (1+x)^{1/3}e^{-\eta|x|} dx 
\\&\le C \delta^{-1} (1+z)^{1/3} e^{-\eta |z|} .
\end{aligned}$$
This completes the proof of the lemma.
\end{proof}

Finally, when $f$ is very localized, we obtain a better convolution estimate as follows. 

\begin{lemma}  \label{lem-locConvAiry} Let $G(x,z)$ be the approximate Green function of the Airy operator, and let $f = f(X)$ satisfy $|f(X)|\le  C_fe^{-\theta_0  |X|^{3/2}} $ for some positive $\theta_0$.  
Then there is some constant $C$ so that 
\begin{equation}\label{conv-locsource}\begin{aligned}
\int_{0}^\infty | \partial_z^k\widetilde G(x,z)  f(X)|   dx &\le  C C_f |\delta|^{1-k} \langle Z \rangle^{k-2}e^{-\theta_0 |Z|^{3/2}},
\\
\int_{0}^\infty | \partial_z^k E(x,z)  f(X)|   dx &\le  C C_f |\delta| ^{1-k},
\end{aligned}\end{equation}
for $\ell \ge 0$. 
\end{lemma} 
\begin{proof} The proof is straightforward, following those from Lemma \ref{lem-ConvAiry}. For instance, we have 
$$\begin{aligned}
&\int_{0}^\infty | \partial_z^k \widetilde G(x,z) f(X)|   dx 
\\&\le C_0C_f|\delta|^{-k} \int_0^\infty \Big[ \langle Z \rangle^{k-3/2} e^{-  {\sqrt{2} \over 3} \sqrt{|Z|}|X-Z| }  
+e^{ - \frac 14 |Z|^{3/2}} e^{ - \frac 14 |X|^{3/2} } \Big] e^{-\theta_0 |X|^{3/2}} \; dx
\\&\le C C_f |\delta|^{1-k}  \langle Z \rangle^{k-2}e^{-\theta_0 |Z|^{3/2}}.
\end{aligned}$$
Here, note that the first integration was taken over the region where $|X|\approx |Z|$. The estimates for $E(x,z)$ follow similarly. \end{proof}

\subsection{The inhomogenous Airy equation}

Using the approximate Green function, we can now construct the exact inverse of $\Airy$ and solve the inhomogenous equation:
 $$\Airy(\phi) = f.$$
We obtain the following proposition. 

\begin{proposition}\label{prop-exactAiry} Let $\eta'<\eta$ be positive numbers and let $\alpha$ satisfy $\nu^{1/2}\ll \alpha \ll \nu^{-1/4}$. Then, the inverse of $\Airy$ operator exists and is well defined from $X_\eta$ to $X_{\eta'}$ so that 
$$ \Airy (\Airy^{-1} (f))= f, \qquad \forall ~ f\in X_\eta.$$
In addition, for $k\ge 0$, there holds 
$$ \| \partial_z^k\Airy^{-1} f \|_{\eta'} \le C \delta^{-k}\| f\|_\eta , \qquad \forall ~ f\in X_\eta.$$
\end{proposition}
\begin{proof} The proposition follows from the standard iteration. Indeed, for $f \in X_\eta$, we set 
\begin{equation}\label{def-star-Gf} \phi_0 (z): = G \star f(z) = \int_0^\infty G(x,z) f(x) \; dx.\end{equation}
By construction, it follows that 
$$ \Airy (\phi_0) = f + E_0, \qquad E_0 := E_\mathrm{a} \star f$$
in which the remainder kernel $E_\mathrm{a}(x,z)$ is defined as in \eqref{eqs-2GAiry}. Lemmas \ref{lem-ConvAiry}  and \ref{lem-ConvAiry-cl} respectively give 
\begin{equation}\label{E00-est} \| \phi_0 \|_{\eta'} \le C_0 \| f\|_\eta, \qquad \| E_0 \|_\eta \le C_0 \delta \|f \|_\eta .\end{equation}
That is, the error term $E_0$ is indeed of order $\cO(\delta)$ in $X_\eta$. Recalling that the assumption on $\alpha$ in particular implies that $\delta \alpha \ll 1$. 

We may now define by iteration an exact solver for the Airy operator $\mathcal{A}(\cdot)$. Let us start with a fixed $f \in X_\eta$. Let us define
\begin{equation}\label{iter-Aphi}
\begin{aligned}
\phi_n &= - \mathcal{A}_\mathrm{a}^{-1}(E_{n-1})
\\
E_n &=  - Err_\mathrm{a}(E_{n-1}) \end{aligned}
\end{equation}
for all $n \ge 1$, with $E_0 = f$. Let us also denote 
$$
S_n = \sum_{k=1}^n \phi_k .
$$
It  follows by induction that 
$$ \mathcal{A} (S_n) = f + E_n,$$
for all $n\ge 1$. By Lemma \ref{lem-ConvAiry-cl}, we have 
$$ \| E_n\|_\eta \le C \delta \|E_{n-1}\|_\eta \le C^n  \delta ^n \| f\|_\eta.$$ 
Recalling that $\delta\to 0$, the above proves that $E_n \to 0$ in $X_\eta$ as $n \to \infty$. Again, Lemma \ref{lem-ConvAiry-cl} yields
$$ \| \phi_n\|_{\eta} \le C \delta^{-1} \| E_{n-1}\|_\eta \le C \delta^{-1} \delta ^{n-1} \|f\|_\eta.$$
This shows that $\phi_n$ converges to zero in $X_{\eta}$ as $n \to \infty$, and furthermore the series 
$$ S_n \to S_\infty$$
 in $X_{\eta}$ as $n \to \infty$, for some $S_\infty \in X_{\eta}$. We then denote $\mathcal{A}^{-1}(f) = S_\infty$, for each $f \in X_\eta$. In addition, we have $ \mathcal{A} (S_\infty) = f,$ that is, $\mathcal{A}^{-1}(f) $ is the exact solver for the modified Airy operator. A similar estimate follows for derivatives, using \ref{conv-loc1}. The proposition is proved.   
\end{proof}

%
%


\subsection{Smoothing effect of Airy operator}

In this section, we study the smoothing effect of the modified Airy function. Precisely, let us consider the Airy equation with a singular source:
\begin{equation}\label{Airyp-singular}\Airy(\phi) = \epsilon \partial_x^4 f(z)\end{equation}
in which $f \in Y^{4,\eta}$, that is $f(z)$ and its derivatives decay exponentially at infinity and behaves as $(z-z_c)\log(z-z_c)$ near the critical layer $z=z_c$. Precisely, we assume that 
\begin{equation}\label{assump-f01} |\dz^k f(z)| \le  C e^{-\eta z} ,\qquad k = 0, \cdots, 4,\end{equation}
 for $z$ away from $z_c$, and 
\begin{equation}\label{assump-f02} 
\begin{aligned}
&|f(z)|\le C, \quad | \dz f(z) | \le C (1 + | \log (z - z_c) | ) , 
\\
  &| \dz^k f(z) | \le C (1 + | z - z_c |^{1 - k} )
,
\end{aligned}\end{equation}
for $z$ near $z_c$ and for $k = 2,3,4$, for some constant $C$. The singular source $\epsilon \partial_x^4 f$ arises as an error of the inviscid solution when solving the full viscous problem. The key for the contraction of the iteration operator lies in the following proposition:

\begin{proposition}\label{prop-mAiry} Assume that $\nu^{1/2}\ll  \alpha \ll \nu^{-1/4}$. Let $\Airy^{-1}$ be the inverse of the Airy operator constructed as in Proposition \ref{prop-exactAiry} and let $f\in  Y^{4,\eta}$, for any positive number $\eta$. Then, there holds the following uniform estimate: 
\begin{equation}\label{smooth-Airy}
\begin{aligned}
\Big\| \Airy^{-1}( \epsilon \partial_x^4 f ) \Big\|_{X^{2,\eta'}} \le  C\|f\|_{Y^{4,\eta}}  \Big[ \delta +  \min\{ |z_c|, \alpha^{-1} \} \Big] |\log \delta| 
.
\end{aligned}\end{equation}
\end{proposition}

In the case when the critical layers $z=z_c$ are away from the boundary, the bound in the above proposition is not sufficiently small. Instead, we obtain the following proposition: 

\begin{proposition}\label{prop-mAiry2} Assume that $\nu^{1/2}\ll  \alpha \ll \nu^{-1/4}$ and $|z_c|\gtrsim 1$. For any positive number $r>0$, we set $\Omega_r := \{z\ge \Re z_c - r\}$. Then, there holds the estimate: 
\begin{equation}\label{smooth-Airy}
\begin{aligned}
\Big\| \Airy^{-1}( \epsilon \partial_x^4 f ) \Big\|_{X^{2,\eta'} (\Omega_r)} \le  C\|f\|_{Y^{4,\eta}}\Big[ \delta +  \min\{r, \alpha^{-1} \} \Big] |\log \delta| 
\end{aligned}\end{equation}
for arbitrary $\eta' < \eta$. 
\end{proposition}

\subsubsection{Pointwise bounds}
Roughly speaking, the Airy inverse of $\epsilon \partial_z^4  f$ is computed through the convolution with the (approximate) Green function $G(x,z)$:  
$$\Airy^{-1}(\epsilon \partial_z^4 f ) \approx G \star \epsilon \partial_z^4 f \approx  - \epsilon \partial_z^3 G \star \partial_z f ,$$
in which $\epsilon \partial_z^3 G$ is bounded and is localized near the critical layer, whose thickness is of order $\delta$. This indicates the bound by a constant of order $\delta \log \delta$ as stated in the estimate \eqref{smooth-Airy}. 
Precisely, we obtain the following lemma.

\begin{lemma}\label{lem-mAiry} Assume that $\nu^{1/2}\ll  \alpha \ll \nu^{-1/4}$.  Let $G(x,z)$ be the approximate Green function of $\Airy$, constructed as in \eqref{def-Gr}, with $E_\mathrm{a}(x,z)$ being the error of the approximation, and let $f\in Y^{4,\eta}$, for any positive number $\eta$. For $k\ge 0$, there holds the following uniform convolution estimates: 
\begin{equation}\label{mphi-bound}
\begin{aligned}
\Big| &(U(z)-c)^k \dz^k G\star \epsilon \partial_z^4f (z)\Big| 
\\&\le     C\|f\|_{Y^{4,\eta}}  (1+|\log \delta|) (\delta+\chi_{[0,\Re z_c]} \min \{|z-z_c|,\alpha^{-1}\}  )e^{-\eta' z}
\end{aligned}\end{equation}
for any $\eta' < \eta$, and in addition, 
\begin{equation}\label{mErrA-bound}
\begin{aligned}
\Big|(U(z)-c)^k \dz^k E_\mathrm{a} \star \epsilon \partial_z^4  f(z) \Big| 
&\le      C \|f\|_{ Y^{4,\eta}} e^{-\eta z} \delta^2 (1+|\log \delta|)
\end{aligned}\end{equation}
for all $z\ge 0$. 
\end{lemma}

\begin{proof} Without loss of generality, we assume that $ \|f\|_{ Y^{4,\eta}}  = 1$. To begin our estimates, let us recall the decomposition of $G(x,z)$ into the localized and non-localized part as 
$$G(x,z) = \widetilde G(x,z) + E(x,z),$$
where $\widetilde G(x,z) $ and $E(x,z) $ satisfy the pointwise bounds in Proposition \ref{prop-ptGreenbound}. Using the continuity of the Green function, we can integrate by parts to get  
\begin{equation}\label{mphi-integral}\begin{aligned}
\phi (z)&=  - \epsilon \int_{0}^{\infty}\D_x^3  (\widetilde G + E) (x,z) \D_x f(x) \; dx 
+ \mathcal{B}_0(z)\\
&= I_\ell (z) + I_e(z) +  \mathcal{B}_0(z)
\end{aligned}
\end{equation}
Here, $I_\ell(z) $ and $I_e(z)$ denote the corresponding integral that involves $\widetilde G(x,z)$ and $E(x,z)$ respectively, and $\mathcal{B}_0(z)$ is introduced to collect the boundary terms at $x=0$ and is defined by 
\begin{equation}\label{def-mBdry}
\mathcal{B}_0(z): = - \epsilon \sum_{k=0}^2 (-1)^k \D_x^k G(x,z) \D^{3-k}_x(f(x))\vert{_{x=0}} .
\end{equation}
By a view of the definition of $E(x,z)$, we further denote 
$$\begin{aligned}  
I_{e,1} (z) : &=   \delta^2 e^{\alpha z} \int_{z}^{\infty} \D_x^3(\dot x^{3/2} a_1 (x) (z-x) ) \D_x f(x) \; dx,
\\ I_{e,2} (z) :& =  \delta^3 e^{\alpha z} \int_{z}^{\infty} \D_x^3 ( \dot x^{3/2} a_2(x)  ) \D_x f(x) \; dx  
\end{aligned}$$
We have $I_e(z) = I_{e,1}(z) + I_{e,2}(z)$. 


\bigskip
\noindent
{\bf Estimate for the integral $I_\ell (z)$.}  Using the bound \eqref{mG-xnearz} on the localized part of the Green function, we can give bounds on the integral term $I_\ell $ in \eqref{mphi-integral}. Consider the case $|z-z_c|\le \delta$. In this case, we note that $\eta'(z) \approx \dot z(\eta(z)) \approx 1$.  By splitting the integral into two cases according to the estimate \eqref{mG-xnearz}, we get
$$\begin{aligned}
|I_\ell (z)| &=\Big| \epsilon \int_{0}^{\infty} \D_x^3  \widetilde G(x,z) \D_x f(x) \; dx \Big| 
\\& \le  \epsilon \int_{\{|x-z_c|\le \delta\}} |\D_x^3  \widetilde G(x,z) \D_x f(x)| \; dx +  \epsilon \int_{\{|x-z_c|\ge \delta\}} |\D_x^3  \widetilde G(x,z) \D_x f(x)| \; dx,
 \end{aligned}$$ 
in which since $\epsilon \D_x^3 \widetilde G(x,z)$ is uniformly bounded, the first integral on the right is bounded by 
 $$ C \int_{\{|x-z_c|\le \delta\}} | \D_x f(x)| \; dx  \le C   \int_{\{|x-z_c|\le \delta\}} (1+|\log (x-z_c)|) \; dx \le C \delta (1+|\log \delta|) .$$ 
For the second integral on the right, we note that in this case since $X$ and $Z$ are away from each other, there holds $e^  {-  {\sqrt{2} \over 3} \sqrt{|Z|}|X-Z| } \le C e^{ - \frac 16 |X|^{3/2} }e^{ - \frac 16 |Z|^{3/2} }$. We get 
$$\begin{aligned}
\epsilon &\int_{\{|x-z_c|\ge \delta\}}  |\D_x^3  \widetilde G(x,z) \D_x f(x)| \; dx 
\\& \le C   \int_{\{|x-z_c|\ge \delta\}}  e^{ - \frac 16 |X|^{3/2} } e^{-\eta  x } (1+|\log (x-z_c)|)\; dx 
\\ & \le C   (1+|\log\delta|)\int_\RR   e^{ - \frac 16 |X|^{3/2} }  \; dx 
\\& \le C   \delta (1+|\log\delta|),
\end{aligned}$$
in which the second-to-last inequality was due to the crucial change of variable $X = \delta^{-1}\eta(x)$ and so $dx = \delta \dot z (\eta (x)) dX$ with $|\dot z(\eta (x))| \le C(1+|x|)^{1/3}$.

Let us now consider the case $|z-z_c|\ge \delta$. Here we note that as $z\to \infty$, $Z = \delta^{-1}\eta(z)$ also tends to infinity since $|\eta(z)| \approx (1+|z|)^{2/3}$ as $z$ is sufficiently large. We again split the integral in $x$ into two parts $|x-z_c|\le \delta $ and $|x-z_c|\ge \delta$. For the integral over $\{|x-z_c|\le \delta\}$, as above, with $X$ and $Z$ being away from each other, we get 
$$\begin{aligned}
\epsilon \int_{\{|x-z_c|\le \delta\}} |\D_x^3  \widetilde G(x,z) \D_x f(x)| \; dx 
& \le C   
e^{ - \frac 16 |Z|^{3/2}}  \int_{\{|x-z_c|\le \delta\}}  (1+|\log (x-z_c)|)\; dx 
\\&  \le C   e^{-\eta z}\delta (1+|\log \delta|).
\end{aligned}$$
Here the exponential decay in $z$ was due to the decay term $e^{ - \frac 16 |Z|^{3/2}}$ with $Z \approx (1+z)^{2/3}$. 
Next, for the integral over $\{ |x-z_c|\ge \delta\}$, we use  the bound \eqref{mG-xnearz} and the assumption $|\D_x f(x)| \le C   e^{-\eta  x } (1+|\log \delta|)$ to get 
$$\begin{aligned}
\epsilon \int_{\{|x-z_c|\ge \delta\}} & |\D_x^3  \widetilde G(x,z) \D_x f(x)| \; dx  
\\&\le C    (1+|\log \delta|) (1+z)^{1/3}  \int e^{-\eta x}   e^  {- \sqrt{2|Z|}|X-Z| /3} \; dx 
\\&\le C (1+|\log \delta|) \delta (1+z)^{1/3}  e^{-\eta z }  |Z|^{-1/2}
 \end{aligned}$$ 
If $z\le 1$, the above is clearly bounded by $C (1+|\log \delta|) \delta$. Consider the case $z\ge 1$. We note that $|Z| \gtrsim |z|^{2/3} / \delta$. This implies that $(1+z)^{1/3}|Z|^{-1/2} \lesssim 1$ and so the above integral is again bounded by $C (1+|\log \delta|) \delta e^{-\eta z } $.

 Therefore in all cases,  we have $|I_\ell (z)|\le C  e^{-\eta z} \delta (1+|\log \delta|)$ or equivalently, 
\begin{equation}  \Big| \int_{0}^{\infty}\epsilon \D_x^3  \widetilde G (x,z) \D_x f(x) \; dx \Big| 
\le C  e^{-\eta z} \delta (1+|\log \delta|)\end{equation}
 for all $z\ge 0$. 
 
 \bigskip
\noindent
{\bf Estimate for $I_{e,2}$.} Again, we consider several cases depending on the size of $z$. For $z$ away from the critical and boundary layer (and so is $x$): $z\ge |z_c| + \delta$, we apply integration by parts to get  
$$\begin{aligned}
 I_{e,2} (z) &=   \delta^3 e^{\alpha z}\int_{z}^{\infty} \D_x^3 ( \dot x^{3/2} a_2(x)  ) \D_x f(x) \; dx  
 \\& =  - \delta^3 e^{\alpha z}\int_{z}^{\infty} \D_x^2 (\dot x^{3/2}a_2(x)) \D^2_x f(x) \; dx  -  \delta^3  e^{\alpha x}\D_x^2 (\dot x^{3/2}a_2(x)) \D_x f(x)\vert_{x=z} .
 \end{aligned}$$
Here for convenience, we recall the bound \eqref{ak12-bound} on $a_2(x)$: 
\begin{equation}\label{bound-ma2}
 |\partial_x^ka_2(x)| \le    C\delta^{-k}(1+x)^{5/6-k/3} e^{-\alpha x} (1+|X|)^{k/2-3/2} .
\end{equation}
Now by using this bound and the fact that $|Z| \gtrsim |z|^{2/3}/\delta$ (recalling that $z$ is away from the critical layer), the boundary term in $I_{e,2}$ is bounded by 
$$\begin{aligned}
&|\delta^3  e^{\alpha z}\D_z^2 (\dot z^{3/2}a_2(z)) \D_z f(z)|
\\&\le C\delta (1+|z|)^{2/3}  (1+|Z|)^{-1/2} (1+|\log (z-z_c)|) e^{-\eta z}  
 \\&  \le C  e^{-\eta z} \delta (1+|\log \delta|) (1+\delta^{1/2}|z|^{1/3})
. \end{aligned}$$
Whereas, the integral term in $I_{e,2}$ is estimated by 
$$\begin{aligned}
\Big|\delta^3 e^{\alpha z} &\int_{z}^{\infty} \D_x^2 (\dot x^{3/2}a_2(x)) \D^2_x f(x) \; dx \Big| 
\\&\le C \delta  \int_{z}^{\infty} (1+x)^{2/3}|X|^{-1/2} |x-z_c|^{-1}   e^{-\eta x} e^{-\alpha |x-z|} \; dx
\\&\le C \delta (1+z)^{2/3} |Z|^{-1/2}   e^{-\eta z} (1+|\log \delta|)
\\&\le C   \delta(1+|\log \delta|) (1+\delta^{1/2} |z|^{1/3})e^{-\eta z} .
\end{aligned}$$
Thus we have 
\begin{equation}\label{mIe2-est}\Big|I_{e,2} (z) \Big| \le  C   \delta(1+|\log \delta|) (1+\delta^{1/2} |z|^{1/3})e^{-\eta z} , \end{equation}
for all $z\ge |z_c|+\delta$. 

Next, for $z\le |z_c| + \delta$, we write the integral $I_{e,2} (z)$ into 
$$ \delta^3 e^{\alpha z} \Big[\int_{\{|x-z_c| \ge \delta\}} + \int_{\{|x-z_c|\le \delta\}} \Big]\D_x^3 (\dot x^{3/2}a_2(x)) \D_x f(x) \; dx , $$ 
where the first integral can be estimated similarly as done in \eqref{mIe2-est}. For the last integral, using \eqref{bound-ma2} for bounded $X$ yields 
$$\begin{aligned}
\Big|&\delta^3e^{\alpha z} \int_{\{|x-z_c|\le \delta\}}\D_x^3 (\dot x^{3/2}a_2(X)) \D_x f(x) \; dx  \Big| 
\\&\le C   \int_{\{|x-z_c|\le \delta\}} e^{-\alpha |x-z|}(1+|\log (x-z_c)|)\; dx
\\& \le C   \delta (1+|\log \delta|) .\end{aligned}$$

Thus, we have shown that 
\begin{equation}\label{mIe2-bound} \Big| I_{e,2} (z) \Big| \le C   \delta(1+|\log \delta|) (1+\delta^{1/2} |z|^{1/3})e^{-\eta z},\end{equation}
for all $z\ge 0$. 

\bigskip
\noindent
{\bf Estimate for $I_{e,1}$.} Following the above estimates, we can now consider the integral 
 $$\begin{aligned}  
I_{e,1} (z) =  \delta^2  e^{\alpha z} \int_{z}^{\infty} \D_x^3(\dot x^{3/2} a_1 (x) (z-x) ) \D_x f(x) \; dx,
\end{aligned}$$
Let us recall the bound \eqref{ak12-bound} on $a_1(x)$: 
\begin{equation}\label{bound-ma1} |\partial_x^k a_1 (x)|  \le C\delta^{-k}(1+x)^{1/2-k/3} e^{-\alpha x}(1+|X|)^{k/2-1}.\end{equation}
To estimate the integral $I_{e,1}(z)$, we again divide the integral into several cases. First, consider the case $z\ge |z_c| + \delta$. Since in this case $x$ is away from the critical layer, we can apply integration by parts three times to get  
$$\begin{aligned} 
I_{e,1} (z)&=  - \delta^2 e^{\alpha z} \int_{z}^{\infty} \D_x^2(\dot x^{3/2} a_1 (x) (z-x)) \D_x^2 f(x) \; dx 
\\&\quad -  \delta^2  e^{\alpha x}\D_x^2(\dot x^{3/2} a_1 (x) (z-x)) \D_x f(x) \vert_{x=z} 
\\  & = -  \delta^2  e^{\alpha z} \int_{z}^{\infty}  \dot x^{3/2} a_1 (x) (z-x) \D_x^4 f(x) \; dx  
\\&\quad  +  \delta^2  e^{\alpha x} \Big( \D_x(\dot x^{3/2} (z-x)a_1) \D_x^2 f(x)  -\D_x^2(\dot x^{3/2} (z-x)a_1) \D_x f\Big) \vert_{x=z} 
 \end{aligned}$$
in which the boundary terms are bounded by $   e^{-\eta |z|}\delta (1+|\log \delta|)$ times 
$$\begin{aligned}
 (1+z)^{1/6}&\Big[ \delta  (1+z)^{5/6} |Z|^{-1} |z-z_c|^{-1} + \delta (1+z)^{-1/6} |Z|^{-1} + (1+z)^{1/2} |Z|^{-1/2}\Big]  
 \\& \le C (1+\delta^{1/2}z^{1/3}).
  \end{aligned}$$ 
  
Similarly, we consider the integral term in $I_{e,1}$. Let $M = \frac 1 \eta \log(1+z)$. By \eqref{bound-ma1}, we have 
$$\begin{aligned}
\Big| \delta^2  e^{\alpha z} &\int_{z}^{\infty}  \dot x^{3/2} a_1 (x) (z-x) \D_x^4 f(x) \; dx   \Big| 
\\& \le  C    \int_z^\infty \delta^2 (1+x)(1+|X|)^{-1}|x-z| |x-z_c|^{-3}  e^{-\eta  x } e^{-\alpha |z-x|}\; dx 
\\& \le  C    (1+z)^{1/3}\Big[ M +(1+z) e^{-\eta M}  \Big]  e^{-\eta z} \int_{\{|x-z_c|\ge \delta\}} \delta^3 |x-z_c|^{-3} \; dx 
\\& \le C   (1+z)^{1/3}\log(1+z) e^{-\eta z}\delta . 
\end{aligned}$$
Hence, we obtain the desired uniform bound $I_{e,1}(z)$ for $z \ge |z_c| + \delta $.

Next, consider the case $|z-z_c|\le \delta$ in which $Z$ is bounded. We write 
$$ I_{e,1}(z) =     \delta^2  e^{\alpha z} \Big[ \int_{\{|x-z_c|\ge \delta\}} + \int_{\{|x-z_c|\le\delta\}} \Big]  \D_x^3(\dot x^{3/2}a_1(x)(z-x)) \D_x f(x) \; dx   . $$ The first integral on the right can be estimated similarly as above, using integration by parts and noting that $|z-x| \le 2\delta$ on the boundary $|x-z_c| = \delta$. For the second integral, we use the bound \eqref{bound-ma1} for bounded $X$ to get 
$$\begin{aligned}
\Big|\delta^2 e^{\alpha z} &\int_{\{|x-z_c|\le\delta\}}\D_x^3(\dot x^{3/2}a_1(x)(z-x)) \D_x f(x)  \; dx \Big| 
\\&\le C \int_{\{|x-z_c|\le\delta\}} (1+|\log (x-z_c)|) e^{-\alpha |x-z|}\; dx
\\&  \le  C   \delta (1+|\log \delta|) .
\end{aligned}$$ 

Finally, we consider the case $0\le z\le |z_c|-\delta$. It suffices to consider the case when $\delta \ll |z_c|$, that is when the critical layer is away from the boundary layer. In this case the linear growth in $Z$ becomes significant: $|Z| \lesssim (1+|z-z_c|/\delta)$. As above, we estimate the integral $I_{e,1}$ over the regions $|x-z_c| \le \delta$ and $|x-z_c|\ge \delta$, separately. For the former case, we compute 
$$\begin{aligned}
\Big|\delta^2 e^{\alpha z} &\int_{\{|x-z_c|\le \delta \}}\D_x^3(\dot x^{3/2}a_1(x)(z-x)) \D_x f(x)  \; dx \Big| 
\\&\le C \delta^{-1} \int_{\{|x-z_c|\le \delta  \}} (1+|\log (x-z_c)|) |x-z| e^{-\alpha |x-z|}\; dx
\\&  \le  C   (1+|\log \delta |) \min \{ |z-z_c|, \alpha^{-1} \} 
\end{aligned}$$ 
in which the inequality $e^{-\alpha |x-z|} |x-z| \le \alpha^{-1}$ and the fact that $|x-z| \lesssim |z-z_c|$ were used. For the integral over $|x-z_c|\ge \delta$, we perform the integration by parts as done in the previous case, yielding the same bound. This proves 
\begin{equation}\label{large-zc}  I_{e,1}(z) \le  C   (\delta+\min\{ |z-z_c|, \alpha^{-1} \}) (1+|\log \delta|) \end{equation}
 for all $z\ge 0$.

\begin{remark}\label{rem-large-zc} We remark that on $\Omega_r  = \{ z\ge \Re z_c - r\}$ for arbitrary positive $r$, the estimate \eqref{large-zc} becomes 
\begin{equation}\label{large-zc-r}  I_{e,1}(z) \le  C   (\delta+\min\{r, \alpha^{-1} \}) (1+|\log \delta|) .\end{equation}
\end{remark}

%
%
 
 \bigskip
 
 \noindent
 {\bf Estimate for the boundary term $\mathcal{B}_0(z)$.} It remains to give estimates on
$$\mathcal{B}_0(z) = - \epsilon \sum_{k=0}^2 (-1)^k \D_x^k G(x,z) \D^{3-k}_x(f(x))\vert{_{x=0}} .$$
We note that there is no linear term $E(x,z)$ at the boundary $x=0$ since $z\ge 0$. Using the bound \eqref{mG-xnearz} for $x = 0$, we get   
$$ \begin{aligned}
|\epsilon \widetilde G(x,z) \D^3_x(f(x))\vert{_{x=0}}  &\le C    \delta^3  (1+ |z_c|^{-2})  e^{-\frac 23 |Z|^{3/2}}
\\
|\epsilon \D_x\widetilde G(x,z) \D^2_x(f(x))\vert{_{x=0}}  &\le C    \delta^2  (1+|z_c|^{-1}) e^{-\frac 23 |Z|^{3/2}}
\\
|\epsilon \D^2_x\widetilde G(x,z) \D_xf(x)\vert{_{x=0}} &\le C    \delta (1+|\log z_c|)e^{-\frac 23 |Z|^{3/2}}.
\end{aligned}$$
This together with the assumption that $\delta \lesssim z_c $ then yields 
\begin{equation}\label{B-est01}| \mathcal{B}_0(z) | \le C    \delta(1+|\log \delta|) e^{-\eta z}.\end{equation}

Combining all the estimates above yields the lemma for $k=0$. This completes the proof of estimate \eqref{mphi-bound} for $k=0$. 

\bigskip
 
 \noindent
 {\bf Proof of estimate \eqref{mphi-bound} with $k>0$.} We now prove the lemma for the case $k=2$; the case $k=1$ follows similarly. We consider the integral 
$$ \epsilon \int_{0}^{\infty} (U(z) - c)^2 \dz^2 \widetilde G(x,z) \D_x^4 f(x) dx  = I_1(z) + I_2(z),$$ with $I_1(z)$ and $I_2(z)$ denoting the 
integration over $\{ |x-z_c|\le \delta\}$ and $\{|x-z_c|\ge \delta\}$, respectively.  
Note that $(U(z)-c)\dot z^{2} =  U'(z_c) \eta(z)$ and recall that $Z = \eta(z)/\delta$ by definition. For the second integral $I_2(z)$, by using \eqref{assump-f01}, \eqref{assump-f02}, and the bounds on the Green function for $x$ away from $z$ and for $x$ near $z$, it follows that 
$$\begin{aligned} |I_2(z)| &\le C \Big[ \delta e^{-\eta z} \int_{\{|x-z_c|\ge \delta\}}  (1+|Z|)^{1/2} e^{-\frac 23 \sqrt{|Z|} |X-Z|} (1+|x-z_c|^{-1})\; dx
\\&\quad + \epsilon e^{-\frac 16 |Z|^{3/2}} \int_{\{|x-z_c|\ge \delta\}}e^{-\frac 16 |X|^{3/2}} (1+|x-z_c|^{-3}) e^{-\eta x}\; dx   \Big] .\end{aligned}$$
By using $|x-z_c|\ge \delta$ in these integrals and making a change of variable $X = \eta(x)/\delta$ to gain an extra factor of $\delta$, the integral $I_2$ is bounded by 
$$\begin{aligned} C \delta  \Big[(1+z) e^{-\eta z} \int_{\RR}  (1+|Z|)^{1/2} e^{-\frac 23 \sqrt{|Z|} |X-Z|} \; dX 
+ e^{-\frac 18 |Z|^{3/2}} \int_{\RR}e^{-\frac 16 |X|^{3/2}}\; dX  \Big] \end{aligned}$$
which is clearly bounded by $C \delta (1+z) e^{-\eta z}$. It remains to give the estimate on $I_1(z)$ over the region: $|x-z_c|\le \delta$. In this case, we take integration by parts three times. Leaving the boundary terms untreated for a moment, let us consider the integral term
$$\epsilon \int_{\{|x-z_c|\le \delta\}} (U(z) - c)^2 \dz^2 \dx^3 \widetilde G(x,z) \D_x f(x) dx.$$
We note that the twice $z$-derivative causes a large factor $\delta^{-2}$ which combines with $(U-c)^2$ to give a term of order $|Z|^2$. Similarly, the small factor of $\epsilon$ cancels out with $\delta^{-3}$ that comes from the third $x$-derivative. The integral is therefore bounded by 
$$\begin{aligned}  C  \Big[ &e^{-\eta z} \int_{\{|x-z_c|\le \delta\}}  e^{-\frac 23 \sqrt{|Z|} |X-Z|} (1+|\log (x-z_c)|)\; dx
\\&\quad + e^{-\frac 16 |Z|^{3/2}} \int_{\{|x-z_c|\le \delta\}}e^{-\frac 16 |X|^{3/2}} (1+|\log(x-z_c)|)\; dx   \Big] 
\\
&\le  C  \Big[ e^{-\eta z} + e^{-\frac 18 |Z|^{3/2}}   \Big] \int_{\{|x-z_c|\le \delta\}} (1+|\log (x-z_c)|)\; dx 
\\
&\le  C e^{-\eta z} \delta (1+|\log \I c|).
\end{aligned}$$
Finally, the boundary terms can be treated, following the above analysis and that done in the case $k=0$; see \eqref{B-est01}. This completes the proof of \eqref{mphi-bound}.

\bigskip
 
 \noindent
 {\bf Proof of estimate \eqref{mErrA-bound}.} The proof follows similarly, but more straightforwardly, the above proof for the localized part of the Green function. 
We skip the details. \end{proof}

\subsubsection{Proof of the Airy smoothing}
We are ready to give the proof of the main results in this section: Propositions \ref{prop-mAiry} and \ref{prop-mAiry2}. 

\begin{proof}[Proof of Proposition \ref{prop-mAiry}] The proposition now follows from the standard iteration, using the pointwise bounds obtained in Lemma \ref{lem-mAiry}Indeed, let us introduce 
$$ \phi_0 : = G\star \epsilon \partial_z^4f $$
Then, by construction, 
$$ \Airy (\phi_0) = \epsilon \partial_z^4f  + E_0, \qquad E_0: = E_\mathrm{a} \star \epsilon \partial_z^4  f(z).$$
Here, by Lemma \ref{lem-mAiry}, the error term $E_0$ satisfies the bound 
$$
\begin{aligned}
\| E_0 \|_\eta &\le    C \|f\|_{ Y^{4,\eta}} \delta^2  (1+|\log \delta|)
.\end{aligned}
$$
If we now set $\psi: = - \Airy^{-1} (E_0)$ with the exact inverse of $\Airy$ being obtained from Proposition \ref{prop-exactAiry}, we then have 
$$ \Airy (\phi_0 + \psi )  = \epsilon \partial_z^4f .$$
The proposition follows at once from Lemma \ref{lem-mAiry} which gives the claimed bounds on $\phi_0$ and Proposition \ref{prop-exactAiry} which provides bounds on $\psi$, the Airy inverse of $E_0$.  
\end{proof}

\begin{proof}[Proof of Proposition \ref{prop-mAiry2}] The proof is identical to that of Proposition \ref{prop-mAiry}, upon using Remark \ref{rem-large-zc}. 
\end{proof}

\section{Orr-Sommerfeld solutions near critical layers}
In this section, we construct four independent solutions to the Orr-Sommerfeld equations: 
\begin{equation}\label{rec-OS}
\OS(\phi) = - \eps \Delta_\alpha^2 \phi + (U - c) \Delta_\alpha \phi - U'' \phi =0.
\end{equation}
We study the solutions near critical layers; namely, we consider $c$ so that $\Im c \ll 1$ and there is a (unique) complex number $z_c$ so that 
$$c = U(z_c).$$
In addition, the spatial frequency $\alpha$ is assumed to be in the range 
\begin{equation}\label{def-rga}
 \nu^{1/2} \ll \alpha \ll \nu^{-1/4} .
 \end{equation}
In particular, we note that $
\epsilon = \frac{\sqrt \nu}{i\alpha} \ll1. 
$
In addition, the condition $\alpha \ll \nu^{-1/4}$ is equivalent to $\delta \alpha \ll1$, in which $\delta \approx \epsilon^{1/3}$ denotes the thickness of critical layers. The latter implies that the critical layers are thinner than the oscillation in the Laplacian $\Delta_\alpha = \partial_z^2 - \alpha^2$.  

\subsection{Main result}

Precisely, our main result in this section is as follows.  

\begin{theorem}[Independent Orr-Sommerfeld solutions]\label{theo-OS-solutions} Let $\alpha$ be within the range as described in \eqref{def-rga}, and $c$ be sufficiently close to the Range of $U$ so that $z_c$ exists and is finite. Assume that the parameters $c, z_c, \epsilon,$ and $\alpha$ satisfy either 
$$\gamma_{\epsilon,c}: =  (\epsilon^{1/3} + \min\{ |z_c|, \alpha^{-1}\} ) (1+|\log \delta \log \Im c|) \ll1$$
or in the case when $|z_c|\gtrsim 1$,
$$\gamma_{\epsilon,c}: = \epsilon^{1/4}(1+|\log \epsilon \log \Im c|)\ll1.$$
Then, there are four independent solutions to the Orr-Sommerfeld equations \eqref{rec-OS}, two of which are slow modes $\phi_{s,\pm}$ close to the Rayleigh solution, and the other two are fast modes $\phi_{f,\pm}$ close to the classical Airy functions. Precisely, there is some positive number $\eta$ so that
$$ \begin{aligned}
\phi_{s,\pm} &= \phi_{Ray,\pm} \big( 1 + \cO(\gamma_{\alpha,c}e^{-\eta z})\Big) 
\\ \phi_{f,\pm} &= \phi_{Airy,\pm} \big( 1 + \cO( \delta (1+\delta \alpha^2)e^{-\eta z})\Big) 
\end{aligned}$$
in which $\delta \sim \epsilon^{1/3}$ denotes the size of critical layers, $\cO(\cdot)$ are bounds in the usual $L^\infty$ norm, $ \phi_{Ray,\pm} $ the Rayleigh solutions constructed in Section \ref{sec-Rayleigh}, and $\phi_{Airy,\pm}$ denotes the second primitive Airy functions:  
$$\phi_{Airy,-}: = \gamma_3 Ai(2,\delta^{-1}\eta(z)) ,\qquad \phi_{Airy,+}(z) : = \gamma_4 Ci(2,\delta^{-1}\eta(z))$$
for some normalizing constants $\gamma_{3,4}$ so that $\phi_{Airy,\pm}(0) =1$. 
\end{theorem}

Let us now detail $Ai$, $Ci$ and $\eta$.
In the above theorem, $Ai(2,\cdot)$ and $Ci(2,\cdot)$ are the corresponding second primitives of the Airy solutions $Ai(\cdot)$ 
and $Ci(\cdot)$, and $\eta(z)$ denotes the Langer's variable. 
We refer to Section \ref{sec-mAiry} for more details. In particular, there hold 
$$
\begin{aligned}Ai(2,Z) &\le C_0 \langle Z \rangle^{-5/4}   e^{- \sqrt{2|Z|} Z/3} 
\\Ci(2,Z) &\le C_0 \langle Z \rangle^{-5/4}   e^{ \sqrt{2|Z|} Z/3}.
\end{aligned} 
$$
Moreover, $Z = \delta^{-1}\eta(z)$ is of order $\epsilon^{-1/3}|z-z_c|$ near the critical layers $z = z_c$ and of order 
$\epsilon^{-1/3}\langle z\rangle^{2/3}$ for large $z$. 
By view of the Langer's variable, we have 
\begin{equation}\label{def-ZZZZZ}\begin{aligned}
 Z 
 &= \Big( \frac 32 \int_{z_c}^z \eps^{-1/2}\sqrt{U-c+\alpha^2 \eps} \; dy\Big)^{2/3} =  \Big( \frac 32 \int_{z_c}^z \mu_f(y) \; dy\Big)^{2/3} 
 \end{aligned}\end{equation}
with $\mu_f(y) = \eps^{-1/2} \sqrt{U-c+\alpha^2 \eps}$. This yields 
$$  e^{\pm \sqrt{2|Z|} Z/3} = e^{\pm\int_{z_c}^z\mu_f(y) \; dy } .$$
Hence, asymptotically as $z\to \infty$, the four Orr-Sommerfeld solutions behave as 
$$ \begin{aligned}
\phi_{s,\pm} \approx e^{\pm \alpha z} ,
\qquad \phi_{f,\pm} \approx e^{\pm\int_{z_c}^z\mu_f(y) \; dy } 
\end{aligned}$$
which coincide with the bounds in the case when critical layers are absent. 


\subsection{Slow Orr-Sommerfeld modes}\label{sec-construction-phi1}


In this section, we iteratively construct two exact slow-decaying and -growing solutions $\phi_{1,2}$ to the full Orr-Sommerfeld equations, starting from the Rayleigh solutions.
We recall that throughout this chapter, the parameter $\alpha$ is considered within the following range 
$$\nu^{1/2} \ll \alpha \ll \nu^{-1/4}.$$
Precisely, we obtain the following propositions whose proof will be given at the end of the section.

\begin{prop}[Small $\Re c$ or large $\alpha$]\label{prop-cons-phi1} Assume that $c, z_c, \delta,$ and $\alpha$ satisfy 
$$ (\delta + \min\{ |z_c|, \alpha^{-1} \}) (1+|\log \delta \log \Im c|) \ll1.$$
Then, for each Rayleigh solution $\phi_{Ray}$, there exists a corresponding Orr-Sommerfeld solution $\phi_s$ so that  
$$ \| \phi_s - \phi_{Ray}\|_{X^{2,\eta}} \le C (\delta + \min\{ |z_c|, \alpha^{-1} \} ) (1+|\log \delta \log \Im c|),$$
for some positive constants $C, \eta$ independent of $\alpha, \epsilon, c$.   
\end{prop}

\begin{remark}{\em  If we start our construction with the exact Rayleigh solutions $\phi_{Ray,\pm}$, which were constructed in Section \ref{sec-Rayleigh}, then Proposition \ref{prop-cons-phi1} yields existence of two exact solutions $\phi_{s,\pm}$ to the homogenous Orr-Sommerfeld equation. }
\end{remark}


\begin{prop}[Bounded $\Re c$ and bounded $\alpha$]\label{prop-cons-phi1c}  Assume that $|z_c|\gtrsim 1$, and
$$ \epsilon^{1/4}(1+|\log \epsilon \log \Im c|)\ll1.$$
Then, for each Rayleigh solution $\phi_{Ray}$, there exists an exact solution $\phi_s(z)$ to the Orr-Sommerfeld equations
so that $\phi_s$ is arbitrarily close to $\phi_{Ray}$ in $X^{2,\eta}$
$$ \| \phi_s - \phi_{Ray}\|_{X^{2,\eta}} \le C \epsilon^{1/4}(1+|\log \epsilon \log \Im c|)$$
for some positive constants $C, \eta$ independent of $\alpha, \epsilon, c$.   
\end{prop}


Next, we obtain the following lemma. 
\begin{lemma}\label{lem-analytic01} 
The slow modes $\phi_s$ constructed in Proposition \ref{prop-cons-phi1} and in Proposition \ref{prop-cons-phi1c} depend analytically in $c$, for $\I c>0$. 
\end{lemma}
\begin{proof} The proof is straightforward since the only ``singularities'' are of the forms: $\log(U-c)$, $1/(U-c)$, $1/(U-c)^2$, and $1/(U-c)^3$, which are analytic in $c$ when $\I c>0$.  
\end{proof}

\begin{remark}  \textup{It can be shown that the constructed slow modes $\phi_{s}$ can be extended $C^\gamma$-H\"older continuously on the axis $\{\I c =0\}$, for $0\le \gamma<1$. 
}\end{remark}



The proof of Proposition \ref{prop-cons-phi1} follows at once from the proof of the following proposition, providing approximate solutions to the Orr-Sommerfeld equations.

\begin{proposition}\label{prop-construction-phi1}
Let $N$ be arbitrarily large and assume the same assumptions as in Proposition \ref{prop-cons-phi1}. For each $f\in X_\eta$ and each bounded Rayleigh solution $\phi_{Ray}$ so that 
$
Ray_\alpha (\phi_{Ray}) = f,
$
there exists a bounded function $\phi_{N} $ that approximately solves 
the Orr-Sommerfeld equation in the sense that 
\begin{equation}\label{def-phiNa} 
\OS(\phi_{N})(z) = f + O_N(z).
\end{equation}
In addition, there hold the following uniform bounds
$$ \| \phi_N - \phi_{Ray} \|_{X^{2,\eta}} \le C(\delta + \min\{ |z_c|, \alpha^{-1}\} ) (1+|\log \delta \log \Im c|) $$
and $$ 
\| O_N\|_{X^{2,\eta}}\le \Big[ C(\delta + \min\{ |z_c|, \alpha^{-1}\} ) (1+|\log \delta \log \Im c|)\Big]^N  .
$$
\end{proposition}

The construction starts from the Rayleigh solution $\phi_{Ray}$ so that $
Ray_\alpha (\phi_{Ray}) = f.
$
By definition, we have
\begin{equation}\label{Orr-1stapp} \OS(\phi_{Ray}) = f - \epsilon \Delta_\alpha^2 (\phi_{Ray})\end{equation}
in which $\Delta_\alpha^2 \phi_{Ray}$ is singular near the critical layer $z=z_c$.  To smooth out the singularity, we introduce 
$$
B_{s} :=  \Airy^{-1}(A_{s}) , \qquad A_{s}:=  \epsilon \Delta_\alpha^2 (\phi_{Ray}) .
$$ 
Together with the identity $Orr =  Airy + U'' $, it follows at once that 
$$\begin{aligned}
 \OS(B_{s} ) &=  \epsilon \Delta_\alpha^2 (\phi_{Ray})  + U'' B_s
 \end{aligned}$$
This leads us to set \begin{equation}\label{def-phi1-s0}
\phi_{1} := \phi_{Ray} + B_{s} 
\end{equation}
which solves the Orr-Sommerfeld equations with a new remainder: $$
\OS( \phi_{1} ) =  f + O_1, \qquad O_1:= U'' B_s .
$$ 
We shall show that the remainder $O_1$ is sufficiently small in some function space. The standard iteration would yield the proposition. Indeed, inductively, let us assume that we have constructed $\phi_{N}$ so that $$
 \OS(\phi_{N})  = f +  O_N,
$$ 
for some sufficiently small remainder. We then improve the error term by constructing a new approximate solution $\phi_{1,N+1}$ so that it solves the Orr-Sommerfeld equations with a smaller remainder. To this end, we first solve the Rayleigh equation by introducing
$$
\psi_{N} := -  Ray_\alpha^{-1}\Bigl( O_N  \Bigr) 
$$ 
and introduce 
$$
\begin{aligned}B_{s,N} :=  \Airy^{-1} A_{s,N} , \qquad 
A_{s,N}:&= \chi \epsilon \Delta_\alpha^2 \psi_N.
\end{aligned}$$
Then, the new approximate solution is defined by \begin{equation}\label{def-phi1N}
\phi_{1,N+1} := \phi_{N} + \psi_N + B_{s,N} 
\end{equation}
solving the Orr-Sommerfeld equations with a new remainder 
 $$
\OS( \phi_{1,N+1} ) =  f +  O_{N+1}, \qquad O_{N+1}: = U'' B_{s,N} .
$$ 
 To ensure the convergence, let us introduce the iterating operator 
\begin{equation}\label{def-Iter}
Iter(g) :=  U'' \Airy^{-1}(A_{s}(g)) , \qquad A_{s}(g) : = \epsilon \Delta_\alpha^2 Ray^{-1}_\alpha (g) .
\end{equation}
  Then the relation between old and new remainders reads
\begin{equation}\label{newON}
O_{N+1} := Iter(O_N) .
\end{equation}
We then inductively iterate this procedure to get an accurate approximation to $\phi_1$. We shall prove the following key lemma which gives sufficient estimates on the $Iter$ operator and would therefore complete the proof of Proposition \ref{prop-construction-phi1}.  

\begin{lemma}\label{lem-keyIter} For $g \in X^{2,\eta}$, the $Iter(\cdot)$ operator defined as in \eqref{def-Iter} is a well-defined map from $X^{2,\eta}$ to $X^{2,\eta}$. Furthermore, there holds 
\begin{equation}\label{est-keyIter} \| Iter(g)\|_{X^{2,\eta}} \le C\Big[ \delta +  \min\{ |z_c|, \alpha^{-1} \} \Big] |\log \delta \log \Im c|  \|g\|_{X^{2,\eta}},\end{equation}
for some universal constant $C$. 
\end{lemma}

\begin{proof} Let $g \in X^{2,\eta}$. By a view of Proposition \ref{prop-exactRayS}, we have
$$
\| Ray_\alpha^{-1}(g) \|_{Y^{4,\alpha}}\le C (1+|\log (\Im c)|)  \|g\|_{X^{2,\eta}}.$$
Next, we consider $A_{s}(g) = \epsilon \Delta_\alpha^2 ( Ray_\alpha^{-1}( g)),$ 
which has a singularity of order $(z-z_c)^{-3}$ due to the $z\log z$ singularity in $Ray_\alpha^{-1}(\cdot)$. Applying Proposition \ref{prop-mAiry}, together with the exponential decay from $U''$ at infinity, 
we get 
$$\Big\| U''\Airy^{-1}( \epsilon \Delta_\alpha^2 (h) ) \Big\|_{X^{2,\eta}} \le C\|h\|_{Y^{4,\alpha}}  \Big[ \delta +  \min\{ |z_c|, \alpha^{-1} \} \Big] |\log \delta \log \Im c|.$$
Putting these estimates together, we obtain 
$$
\| Iter(g)\|_{X^{2,\eta}} \le  C\Big[ \delta +  \min\{ |z_c|, \alpha^{-1} \} \Big] |\log \delta \log \Im c| \|g\|_{X^{2,\eta}} $$ 
which gives the lemma. 
\end{proof}

\begin{proof}[Proof of Proposition \ref{prop-cons-phi1c}] We now consider the case when both $\Re c$ and $\alpha$ are bounded. Let $r$ be an arbitrary positive number so that $\delta \ll r$. We first solve the Orr-Sommerfeld equations on $\Omega_r: = \{ z\ge \Re z_c - r\}$. 
Recall from Proposition \ref{prop-mAiry2} that 
$$
\begin{aligned}
\Big\| \Airy^{-1}( \epsilon \partial_x^4 f ) \Big\|_{X^{2,\eta'} (\Omega_r)} 
&\le  C\|f\|_{Y^{4,\eta}}\Big[ \delta +  \min\{r, \alpha^{-1} \} \Big] (1+|\log \delta| )
\\
&\le  C\|f\|_{Y^{4,\eta}} r (1+|\log \delta|) .
 \end{aligned}$$
 Using this and following the proof of Proposition \ref{prop-cons-phi1}, we obtain an exact Orr-Sommerfeld solution defined on $\Omega_r$ so that 
\begin{equation}\label{bd-phisRs} \| \phi_s - \phi_{Ray}\|_{X^{2,\eta}(\Omega_r)} \le C r (1+|\log \delta \log \Im c|)\end{equation}
as long as $ r (1+|\log \delta \log \Im c|) \ll1$. 

It remains to show that $\phi_s$ can be continued down to $z =0$ and remains close to the Rayleigh solution $ \phi_{Ray}$. Set $\Omega_r^c: = [0,\Re z_c - r]$ and  $\phi = \phi_s - \phi_{Ray}$. We are led to solve the following inhomogenous Orr-Sommerfeld equations 
$$
\OS(\phi) = \epsilon \Delta_\alpha^2 \phi_{Ray},
$$
on $\Omega_r^c$, together with boundary conditions $\phi =  \phi_s - \phi_{Ray}$ and $\phi' =  \phi_s' - \phi_{Ray}'$ at $z= \Re z_c - r$ and $\phi = \phi'=0$ at $z=0$. The solution $\phi$ is defined by 
\begin{equation}\label{def-phirc} \phi (z) = \int_{\Omega_r^c} G_{\alpha,c} (x,z)   \epsilon \Delta_\alpha^2 \phi_{Ray} (x)\; dx + \sum_{j=s,f} A_{j,\pm} \phi_{j,\pm}(z)\end{equation}
in which $G_{\alpha,c}(x,z)$ denotes the Green function of $\OS(\cdot)$ on $\Omega_r^c$, $\phi_{j,\pm}$ are four independent solutions of the Orr-Sommerfeld equations on $\Omega_r^c$, and the constants $A_{j,\pm}$ are added to correct the boundary conditions. 

Since there are no critical layers in $\Omega^c_r$, solutions to the Orr-Sommerfeld equation can be constructed via the standard iteration from the two Rayleigh solutions $\phi_{s,\pm}\sim e^{\pm\mu_s z}$ and two Airy solutions $\phi_{f,\pm} \sim e^{\pm \mu_f (z-z_c)}$. We then construct the Green function for the Orr-Sommerfeld equations, without taking care of the boundary conditions. In particular, the $L^1$ norm of $G_{\alpha,c}(\cdot,z)$ is uniformly bounded. Thus, 
$$ \Big|  \int_{\Omega_r^c} G_{\alpha,c} (x,z)   \epsilon \Delta_\alpha^2 \phi_{Ray} (x)\; dx \Big| \lesssim \sup_{z\in \Omega_r^z} |  \epsilon \Delta_\alpha^2 \phi_{Ray} (z)| \lesssim \epsilon r^{-3},$$
upon recalling that the singularity of the Rayleigh solution $\phi_{Ray}$ is of order $(z-z_c) \log (z-z_c)$. Take $r = \epsilon^{1/4}$. The above integral is thus bounded by a constant times $\epsilon^{1/4}$. In addition, by \eqref{bd-phisRs}, the boundary value 
$$ |\partial_z^k\phi_{\vert_{z=\Re z_c - r}}| \lesssim  \epsilon^{1/4 - k/4} (1+|\log \delta \log \Im c|) . $$
This proves that 
$$
\begin{aligned}
 \sum_{j=s,f} A_{j,\pm} {\partial_z^k\phi_{j,\pm}}_{\vert_{z=0,\Re z_c - r}} &\lesssim \epsilon^{1/4-k/4} (1+|\log \delta \log \Im c|) 
 ,\end{aligned}$$ 
which yields $A_{s,\pm}, A_{f,+} \lesssim  \epsilon^{1/4} (1+|\log \delta \log \Im c|) $ and $A_{f,-}$ is exponentially small. In particular, 
$$\Big|\sum_{j=s,f} A_{j,\pm} \partial_z^k\phi_{j,\pm}(z)\Big| \lesssim \epsilon^{1/4-k/4} (1+|\log \delta \log \Im c|) .$$
By view of \eqref{def-phirc} and \eqref{bd-phisRs}, the proposition is proved. 
\end{proof}

\subsection{Fast Orr-Sommerfeld modes}

In this section, we shall construct exact solutions to the Orr-Sommerfeld equations that are close to the two independent Airy solutions:  
\begin{equation}\label{def-phi30}
\phi_{3,0}(z) : = \gamma_3 Ai(2,\delta^{-1}\eta(z)) ,\qquad \phi_{4,0}(z) : = \gamma_4 Ci(2,\delta^{-1}\eta(z)) ,
\end{equation} 
in which $\gamma_3 = 1/ Ai(2,\delta^{-1}\eta(0))$ and $\gamma_4 = 1/Ci (2,\delta^{-1} \eta(0))$ are normalizing constants. 
 Here, $Ai(2,\cdot)$ and $Ci(2,\cdot)$ are the second primitive of the Airy solutions $Ai(\cdot)$ and $Ci(\cdot)$, respectively, and $\eta(z)$ denotes the Langer's variable
\begin{equation}\label{Lg-transform}
\delta = \Bigl( { \eps \over i U'_c} \Bigr)^{1/3} , \qquad\quad \eta(z)
= \Big[ \frac 32 \int_{z_c}^z \Big( \frac{U-c}{U'_c}\Big)^{1/2} \; dz \Big]^{2/3}.
\end{equation}
We recall that as $Z = \eta(z)/\delta$ which tends to the infinity along the line $e^{i\pi/6}\mathbb{R}$, the Airy solution $Ai(2, e^{i \pi/6}Z)$ behaves as $e^{\mp \frac {\sqrt{2}}{3} |Z|^{3/2}}$, whereas $Ci(2, e^{i \pi/6}Z)$ is of order $e^{\pm \frac {\sqrt{2}}{3} |Z|^{3/2}}$. In particular, since $\eta(0) \sim -z_c$, $\delta^{-1}\eta(0) \to \infty$ with an angle approximately $\frac 76 \pi$. Hence, the normalizing constant $\gamma_{3,4}$ is approximately of order $\langle z_c/\delta \rangle ^{5/4}e^{\pm\frac {\sqrt{2}}{3} |z_c/\delta|^{3/2}}$, respectively. 
Let us also recall that the critical layer is centered at $z = z_c$ and has a typical
size of $|\delta|\sim \epsilon^{1/3}$. Inside the critical layer, the Airy functions play a crucial role.

\begin{prop}\label{prop-construction-phi3} 
We assume that $\delta (1+\delta \alpha^2) \ll1$. Then, there are two exact independent solutions $\phi_{j}(z)$, $j = 3,4$, solving the Orr-Sommerfeld equation
$$ 
\OS(\phi_{j}) = 0, \qquad j = 3,4,
$$
so that $\phi_{j}(z)$ is approximately close to $\phi_{j,0}(z)$ in the sense that 
\begin{equation}\label{est-phi3-Orr} 
 \phi_j  =  \phi_{j,0} (1+ \psi)
  \end{equation}
with 
$$ \| \psi\|_\eta \le C \delta (1+\delta \alpha^2) $$
 for some constants $C$ independent of $\alpha,\epsilon,$ and $c$.  
\end{prop}

From the construction, we also obtain the following lemma. 

\begin{lemma}\label{lem-analytic03} 
The fast Orr mode $\phi_{3,4}(z)$ constructed in Proposition \ref{prop-construction-phi3} 
depends analytically in $c$ with $\I c \not = 0$. 
\end{lemma}

\begin{proof} This is simply due to the fact that both Airy function 
and the Langer transformation \eqref{Lg-transform} are analytic in their arguments. 
\end{proof}



\begin{proof}[Proof of Proposition \ref{prop-construction-phi3}]
We start with $\phi_{3,0}(z)  =  \gamma_3 Ai(2,\delta^{-1}\eta(z))$. We note that $\phi_{3,0}$ and $\dz\phi_{3,0}$ are both bounded on $z\ge 0$, and so are $\eps \dz^4 \phi_{3,0}$ and $(U-c)\dz^2 \phi_{3,0}$. We shall show indeed that $\phi_{3,0}$ approximately solves the Orr-Sommerfeld equation. First, by the estimates on the Airy functions in Lemma \ref{lem-classicalAiry}, when the $z$-derivative hits  the Airy functions, it gives an extra $\delta^{-1}\langle Z \rangle^{1/2}$, and hence,
$\phi_{3,0}(z)$ satisfies 
$$ |\partial^k_z\phi_{3,0}(z)| \le C_0 \gamma_3\delta^{-k}\langle Z \rangle^{-5/4+k/2}   e^{- \sqrt{2|Z|} Z/3} ,\qquad k \ge 0.$$
We note that thanks to the normalizing constant $\gamma_3$, the above estimate in particular yields that 
$$ |\partial^k_z\phi_{3,0}(0)| \le C_0\delta^{-k} \langle z_c/\delta \rangle ^{k/2} ,\qquad k \ge 0, $$
which could be large in the limit $\epsilon, z_c \to 0$. When $z$ is away from zero, the exponent $e^{- \sqrt{2|Z|} Z/3}$ is sufficiently small, of order $e^{-1/|\delta|^{3/2}}$ in the limit $\delta \to 0$. This controls any polynomial growth in $1/\delta$.  
Next, direct calculations yield 
$$\begin{aligned}
 \Airy(\phi_{3,0}):= &~\eps \delta^{-1} \gamma_3\eta^{(4)} Ai(1,Z) + 4 \eps \delta^{-2}  \gamma_3\eta' \eta^{(3)} Ai(Z) + 3\eps \delta^{-2} \gamma_3 (\eta'')^2 Ai(Z)
 \\&+ \eps \delta^{-4} \gamma_3 (\eta')^4 Ai''(Z) + 6 \eps \delta^{-3} \gamma_3 \eta '' (\eta')^2 Ai'(Z) \\& - \gamma_3 (U-c) \Big[\eta'' \delta^{-1}Ai(1,Z) + \delta^{-2}(\eta')^2 Ai(Z)\Big]
 \\&+ \alpha^2 (U-c+\epsilon \alpha^2) \phi_{3,0},
\end{aligned}
$$ with $Z = \delta^{-1}\eta(z)$. Let us first look at the leading terms with a factor of $\eps \delta^{-4}$ and of $(U-c)\delta^{-2}$. Using the facts that $\eta' = 1/\dot z$, $\delta ^3= \eps/U_c'$, and $(U-c)\dot z^2 = U_c' \eta(z)$, we have 
$$\begin{aligned}  \eps \delta^{-4} (\eta')^4 Ai''(Z) &- \delta^{-2}(\eta')^2  (U-c)Ai(Z) \\
&=  \eps \delta^{-4} (\eta')^4 \Big[Ai''(Z) - \delta^{2}\eps^{-2}  (U-c)\dot z^2 Ai(Z)\Big]  \\&=  \eps \delta^{-4} (\eta')^4 \Big[Ai''(Z) - Z Ai(Z)\Big] = 0.
\end{aligned}$$
The next terms in $Airy(\phi_{3,0})$ are
$$\begin{aligned}
6 \eps \delta^{-3} \gamma_3 \eta '' (\eta')^2 Ai'(Z)  &-  \gamma_3(U-c)\eta'' \delta^{-1}Ai(1,Z) \\&= \gamma_3 \Big[6 \eta '' (\eta')^2 U_c' Ai'(Z)  - Z U_c' \eta'' (\eta'^2)Ai(1,Z)\Big]
\\&= \gamma_3 \eta '' (\eta')^2 U_c'  \Big[ 6Ai'(Z)  - Z Ai(1,Z) \Big] ,\end{aligned}$$ 
 which is bounded for $z\ge 0$. The rest is of order $\cO(\eps^{1/3})$ or smaller. That is, we obtain
 $$
 \begin{aligned} \Airy(\phi_{3,0}) = I_0(z): &=  \gamma_3 \eta '' (\eta')^2 U_c'  \Big[ 6Ai'(Z)  - Z Ai(1,Z) \Big] \\&\quad + \alpha^2 (U-c+\epsilon \alpha^2) \phi_{3,0}+ \cO(\eps^{1/3}).\end{aligned}$$
Here we note that the right-hand side $I_0(z)$ is very localized and depends primarily on the fast variable $Z$ as $Ai(\cdot)$ does. Precisely, we have 
\begin{equation}\label{I-bound} |\partial_z^kI_0(z)|\le C  \gamma_3\delta^{-k} (1+ \delta \alpha^2) \langle Z \rangle^{1/4+k/2} e^{- \sqrt{2|Z|} Z/3} \end{equation}
in which we have used $\epsilon \alpha^2 \lesssim \delta$ and $(U-c+\epsilon\alpha^2)\dot z^2  = \delta U'_c Z$. Again as $z\to 0$, we obtain the following bound, using the normalizing constant $\gamma_3$, 
$$| \partial_z^k I_0(0)| \le C_0 \delta^{-k}(1+ \delta \alpha^2) \langle z_c/\delta \rangle^{3/2+k/2}.$$

To obtain a better error estimate, we then introduce 
$$ \phi_{3,1}(z): = \phi_{3,0}(z)  + \Airy^{-1} (I_0) (z).$$ 
This yields 
$$ \OS(\phi_{3,1}) = I_1(z): =  U'' \Airy^{-1} (I_0) (z).$$
The existence of $I_1$ is guaranteed by the Proposition \ref{prop-exactAiry}. 
More precisely, we have
\begin{equation}\label{I1-bound}|\partial_z^kI_1(z)| \le C\delta^{1-k} \gamma_3 (1+ \delta \alpha^2) \langle Z \rangle^{-7/4+k/2}   e^{- \sqrt{2|Z|} Z/3}   + C_k\delta^{1-k},\end{equation}
for $k\ge 0$, in which $C_k$ vanishes for $k\ge 2$. Indeed, in the above estimate, the terms on the right are due to the convolution with the localized and non-localized part of the Green function of the $\Airy$ operator, respectively. Hence, when $z$-derivative hits the non-localized part $E(x,z)$, we have by definition, $\partial_z^k E(x,z) = 0$ for $k\ge 1$ and $x<z$, and $|\partial_zE(x,z) |\le  C\delta^{-1}\langle x \rangle^{1/3}\langle X \rangle^{-1} $ and $\partial_z^k E(x,z) = 0$ for $x>z$ and for $k\ge 2$. In particular, there is no linear growth in $Z$ in the last term on the right of \eqref{I1-bound}. In addition, we also note that as $z\to 0$, there holds the uniform estimate:
$$|\partial_z^kI_1(0)|\le C\delta^{1-k} (1+ \delta \alpha^2) \langle z_c/\delta \rangle^{-1/2+k/2} 
+ C \delta^{1-k}.$$ 
In particular, $I_1 = \cO(\delta)(1+ \delta \alpha^2)$ and $\partial_zI_1 = \cO(1+ \delta \alpha^2)$.

The construction of $\phi_3$ now follows from the standard iteration as done in the proof of Proposition \ref{prop-cons-phi1} and of Proposition \ref{prop-cons-phi1c}. A similar construction applies for $\phi_{4,0} = \gamma_4 Ci(2,\delta^{-1}\eta(z))$, since both $Ai(2,\cdot)$ and $Ci(2,\cdot)$ solve the same primitive Airy equation. 
\end{proof}

\section{Green function for Orr-Sommerfeld}

In this section, we prove Theorem \ref{theo-GreenOS-stable}, giving the pointwise bounds on the Green function of the Orr-Sommerfeld problem. Given the four independent solutions to the homogenous Orr-Sommerfeld equations, the Green function is constructed exactly as done in \cite{GrN1}, with a special attention near critical layers. Precisely, let $\phi_{s,\pm}, \phi_{f,\pm}$ be slow and fast modes of the Orr-Sommerfeld equations, respectively, constructed in the previous sections, and let $\phi_{s,\pm}^*, \phi_{f,\pm}^*$ be the corresponding adjoint solutions, defined through the following algebraic relations
\begin{equation}\label{OSadjoint}\begin{aligned}
\Phi^*_{j,+} \mathcal{B}\Phi_{k,-} &= \delta_{jk}, \qquad \Phi^*_{j,+} \mathcal{B} \Phi_{k,+}  = 0,
\\
\Phi^*_{j,-}  \mathcal{B} \Phi_{k,-} &=0, \qquad \Phi^*_{j,-}  \mathcal{B} \Phi_{k,+}  = \delta_{jk},
\end{aligned}
\end{equation}
for $j,k = s,f$, in which we have used the column vector notation 
$$
\Phi = [\phi, \phi', \phi'', \phi''']^t,
$$
respectively for $\phi=\phi_{j, \pm}$ with $j = s,f$. 
The same notation applies for the adjoint solutions $\Phi^*$, defined as a row vector. 
Here, the jump matrix 
\begin{equation}\label{def-cBBB} 
\mathcal{B}: = [\mathcal{G}_{\alpha,c}(x,z)]_{\vert_{z=x}}^{-1}
\end{equation}
is computed by $$
\mathcal{B}= 
 \begin{pmatrix} 
B(x) & \mathcal{O}(\epsilon) \\ \mathcal{O}(\epsilon) &0\end{pmatrix}, \qquad  B(x)= \begin{pmatrix} 
b'(x) & b(x;\epsilon) 
\\ -b(x;\epsilon)& 0
\end{pmatrix}
$$
with $b(x;\eps) = (U - c + 2 \epsilon \alpha^2)$. 

Following \cite{ZH,GrN1}, we obtain the following representation for the Green kernel $G_{\alpha,c}(x,z)$
\begin{equation}\label{rep-GrOS}G_{\alpha,c}(x,z) = \left \{ \begin{aligned} 
\sum_{j,k = s,f} d_{jk} \phi_{j,-}(z) \phi_{k,-}^* (x)  + \sum_{k = s,f} \phi_{k,-}(z) \phi_{k,+}^* (x) , \qquad & z>x\\
\sum_{j,k = s,f} d_{jk} \phi_{j,-}(z) \phi_{k,-}^* (x)  - \sum_{k = s,f} \phi_{k,+}(z) \phi_{k,-}^* (x) , \qquad & z<x\\
\end{aligned} \right. 
\end{equation}
on the half-line, in which the coefficient matrix $(d_{jk})$ is defined by 
$$M =(d_{jk})= \begin{pmatrix} \phi_{s,-} & \phi_{f,-} \\ \phi'_{s,-} & \phi'_{f,-}\end{pmatrix}^{-1} 
\begin{pmatrix} \phi_{s,+} & \phi_{f,+} \\ \phi'_{s,+} & \phi'_{f,+}\end{pmatrix} _{\vert_{z=0}}.
$$

It suffices to give bounds on the four independent Orr-Sommerfeld solutions and their adjoints. Recalling Theorem \ref{theo-OS-solutions} and the formula \eqref{def-ZZZZZ} of the Langer's variables, 
we have
$$
\phi_{s,\pm}(z) \approx e^{\pm \mu_sz}, \qquad \phi_{f,\pm} \approx \Big \langle \int_{z_c}^z \mu_f(y)\; dy \Big \rangle^{-5/6} e^{\pm \int_{z_c}^z \mu_f(y)\; dy}
$$ 
with $\mu_s  = \alpha$ and $\mu_f = \epsilon^{-1/2}\sqrt{U-c + \alpha^2 \eps}.$ In particular, $\Re \mu_f  \gg \mu_s$, since $\epsilon \alpha^2 \ll1$. 
Note that near the critical layers $z = z_c$, the fast modes are of the form  
$$\phi_{f,\pm} \approx \langle Z \rangle^{-5/4}   e^{ \pm\sqrt{2|Z|} Z/3}, \qquad Z \approx \delta^{-1} |z-z_c|.$$

It remains to study the adjoint solutions $\phi^*_{s,\pm}$ and $\phi^*_{f,\pm}$ which are defined through the algebraic relations \eqref{OSadjoint}. Let us first consider $z$ being away from the critical layers. Since $B(x)$ is of order one, it follows at once that asymptotically at $z = \infty$, the adjoint solutions decay and grow at the same exponential rate as those of $\phi_{j,\pm}$: that is, $\phi^*_j(z)\sim c_{j,\pm}e^{\pm\mu_j |z|}$, for some constants $c_{j,\pm}$, and for each $j = s,f$. To compute the coefficients, let $\Pi_1$ be the projection on the first two components of a vector in $\mathbb{C}^4$. From the relation $\Phi^*_{f_+}[\mathcal{G}(x,z)]_{\vert_{z=x}} ^{-1} \Phi_{f,-} = 1$, we first compute 
\begin{equation}\label{Pi-relation} 1 \approx \Pi_1 \Phi_{f,+}^* B(x)  \Pi_1 \Phi_{f,-} = b' \phi_{f,-} \phi_{f,+}^* - b (\phi_{f,-} \partial_z \phi_{f,+}^* - \phi_{f,+}^* \partial_z \phi_{f,-}).\end{equation}
 Similar relations hold for $\phi_{f,-}^*$. As the fast modes $\phi_{f,\pm}, \phi_{f,\pm}^*$ behave as the classical Airy functions, the $z$-derivative of these modes yields the factor of order $\delta^{-1}\langle Z \rangle^{1/2} \gg 1$. This shows that 
$$\phi_{f,\pm}^*(z) \approx \Big \langle \int_{z_c}^z \mu_f(y)\; dy \Big \rangle^{1/2}   e^{\pm \int_{z_c}^z \mu_f\; dy}$$
as $z\to \infty$. Similarly, $\phi_{s,\pm}^* (z)\approx \mu_s^{-1}e^{\pm \mu_s |z|}$. As for the behavior near the critical layers $z = z_c$, the approximation \eqref{Pi-relation} yields 
$$\phi^*_{f,\pm} \approx \delta \langle Z \rangle^{3/4}   e^{ \pm\sqrt{2|Z|} Z/3}, \qquad Z \approx \delta^{-1} |z-z_c|$$
near the critical layers $z = z_c$, upon recalling that $b(x) = U - c + 2 \epsilon \alpha^2$. 


We are now ready to put the above estimates into the representation formula \eqref{rep-GrOS} for the Green function. We consider the case when $x<z$ and $\alpha\gtrsim 1$. We have 
$$|\phi_{s,-}(z) \phi_{s,+}^* (x) | \le \mu_s^{-1} e^{-\alpha |x-z|} $$
and 
$$|\phi_{f,-}(z) \phi_{f,+}^* (x)| \le \delta \langle Z\rangle^{-5/4} \langle X \rangle^{3/4} e^{ -  \int_{x}^z \Re \mu_f(y)\; dy}$$
in which $Z = \delta^{-1}\eta(z)$ and $X = \delta^{-1}\eta(x)$. Similarly, we obtain 
$$ 
\begin{aligned}
|\phi_{s,-}(z) \phi_{s,-}^* (x) | &\le \mu_s^{-1} e^{-\alpha (|x| + |z|)} 
\\
|\phi_{s,-}(z) \phi_{f,-}^* (x) | &\le  \delta \langle X \rangle^{3/4} e^{-\alpha z}  e^{ -  \int_{z_c}^x \Re \mu_f(y)\; dy}
\\
|\phi_{f,-}(z) \phi_{s,-}^* (x) | &\le \mu_s^{-1} e^{-\alpha x} \langle Z \rangle^{-5/4} e^{ -  \int_{z_c}^z \Re \mu_f(y)\; dy} \\
|\phi_{f,-}(z) \phi_{f,-}^* (x) | &\le \delta \langle Z\rangle^{-5/4} \langle X \rangle^{3/4} e^{ -  \int_{z_c}^z \Re \mu_f(y)\; dy}e^{ -  \int_{z_c}^x \Re \mu_f(y)\; dy}
\end{aligned}$$
in which we note that $\Re \mu_f \gg \alpha$. Putting these estimates into \eqref{rep-GrOS} and simplifying terms, we obtain at once the claimed bounds on the Green function and hence Theorem \ref{theo-GreenOS-stable} for $\alpha \gtrsim 1$. 

Finally, we observe that in the case when $\alpha \ll1$ (and still, $\epsilon \ll1$), the slow modes are approximated by the Rayleigh solutions
$$ \phi_s \approx e^{-\alpha z}\Big( U-c + \cO(\alpha)\Big), \qquad \psi_s  = \phi_s \int_{z_c}^z \frac{1}{\phi_s^2} \; dy  .$$
See Section \ref{sec-exactRayleigh}. In particular, $\psi_s \approx e^{\alpha z} (1 + (z-z_c) \log(z-z_c))$. Hence, in this case, we have 
\begin{equation}\label{Gslow-smalla} |\phi_s \phi_s^*|  \le \frac{C_0}{|U-c|} \Big( |U-c| + \cO(\alpha)\Big) e^{-\alpha |x-z|} \end{equation}
for $\alpha \ll1$. Theorem \ref{theo-GreenOS-stable} follows.

\begin{proof}[Proof of Theorem \ref{theo-vort-stable}]
Finally, we derive bounds on $\Delta_\alpha G_{\alpha,c} (x,z)$ for the stable case. Decompose $\Delta_\alpha G_{\alpha,c} (x,z)$ 
as in \eqref{def-vortG12}, leaving the remainder  
$$  \mathcal{R}_G(x,z) := \int_0^\infty \mathcal{G}(y,z) U''(y) G_{\alpha,c} (x,y) \; dy $$
as in \eqref{def-RGGG}. 
We recall fromProposition \ref{prop-exactGrmAiry} that 
the approximate Green function $\mathcal{G}(x,z)$ satisfies 
$$
|\D_z^\ell \D_x^k \mathcal{G}(x,z) | 
\le 
C_{x,z} \delta^{-2-k-\ell}\langle Z \rangle^{(2k-1)/4}\langle X \rangle^{(2\ell-1)/4}
e^{ -  \int_{x}^z \Re \mu_f(y)\; dy}
$$
simply written in term of the integral of $\mu_f(z)$. The estimates follow the same lines as done in \cite[Section 5]{GrN1}, upon noting that the convolution of $\mathcal{G}(x,z)$ yields a pre-factor $\delta \langle Z \rangle^{-1/2}$; see the proof of Lemmas \ref{lem-ConvAiry-cl}  and \ref{lem-locConvAiry}. The theorem follows at once. 
\end{proof}


\section{Semigroup bounds for Navier-Stokes}


We are now ready to derive semigroup bounds for the linearized Navier-Stokes. We shall work with boundary layer norms, which we recall 
\begin{equation}\label{assmp-wbl-stable123} \| \omega_\alpha\|_{ \beta, \gamma, p} = \sup_{z\ge 0} \Big[ \Bigl( 1 +  \sum_{q=1}^p\delta^{-q} \phi_{P-1+q} (\delta^{-1} z)  \Bigr)^{-1} e^{\beta z} |\omega_\alpha(z)|\Big]\end{equation}
with $p\ge 0$, and the boundary layer thickness 
$$\delta = \gamma \nu^{1/8}$$
for some fixed positive constant $\gamma>0$.  

\subsection{Semigroup of modified vorticity}
In this section, we provide bounds on $\mathcal{S}_\alpha$, the semigroup of the modified vorticity equation, that is, the linearized vorticity equation \eqref{EE-vort} dropping the lower order term $vU''$:
\begin{equation}\label{def-Seqs} (\partial_t + i\alpha U) \omega = \sqrt \nu \Delta_\alpha \omega .\end{equation}
Precisely, we obtain the following proposition. 

\begin{proposition}\label{prop-tSa-stable} Define the semigroup $\mathcal{S}_\alpha$ as in \eqref{def-SRa}. Then, for any positive $\gamma_1$, there is a constant $C_0$ so that 
$$\| \mathcal{S}_\alpha  \omega_\alpha\|_{ \beta, \gamma, p} \le C_0 e^{\gamma_1\nu^{1/4} t }  e^{- \frac14 \alpha^2 \sqrt\nu t}  \| \omega_\alpha\|_{ \beta, \gamma, p}.$$
In addition, for $k\ge 0$, there holds
$$\begin{aligned}
\| \partial_z^k \mathcal{S}_\alpha\omega_\alpha \|_{ \beta, \gamma, p+k} \le C_k e^{\gamma_1 \nu^{1/4} t } e^{- \frac14 \alpha^2 \sqrt \nu t}  \sum_{a+b\le k}  t^a \| \alpha^a\partial_z^b\omega_\alpha \|_{ \beta, \gamma, p+b} .\end{aligned}$$
\end{proposition}

\begin{remark}
We note that the exponential growth $e^{\gamma_1 \nu^{1/4}t}$ in time can be replaced by $e^{\eta_0 \sqrt \nu t}$, if we modify the boundary layer thickness $\delta = \gamma \nu^{1/4}$ by $\delta(t) =\gamma \nu^{1/4} \sqrt{1+t}$ in the boundary layer norm $\| \cdot \|_{\beta,\gamma,p}$. In this case, there holds the following uniform bound
\begin{equation}\label{sharp-grSbd}\| \mathcal{S}_\alpha  \omega_\alpha\|_{ \beta, \gamma, p} \le C_0 e^{\eta_0\sqrt \nu t }  e^{- \frac14 \alpha^2 \sqrt\nu t}  \| \omega_\alpha\|_{ \beta, \gamma, p}\end{equation}
for some positive constants $C_0,\eta_0$; see Remark \ref{rem-expnu}. The polynomial growth in time for derivatives is sharp and can be seen from the usual transport operator $\partial_t + i\alpha U(z)$.  
\end{remark}

\subsubsection{Temporal Green function}
In this section, we study the temporal Green function $G_S (z,t;x)$ of 
$$( \partial_t + i\alpha U)\omega   - \sqrt \nu \Delta_\alpha  \omega=0,$$ defined by 
\begin{equation}
\label{def-tGreen}G_S(z,t;x): = \frac{1}{2\pi i}   \int_{\Gamma_{\alpha}} e^{\lambda t}   \mathcal{G} (x,z) \;\frac{d\lambda}{i\alpha} \end{equation} 
for an arbitrary contour $\Gamma_\alpha$ leaving on the left the resolvent set of $ ( \lambda + i\alpha U)\omega   - \sqrt \nu \Delta_\alpha  \omega$. Here, $\mathcal{G}(x,z)$ denotes the Green function of the modified Airy operator $ \epsilon \Delta_\alpha  \omega - U+c$, with $\epsilon = \sqrt\nu / i\alpha$ and $c = -\lambda/i\alpha$; see Proposition \ref{prop-exactGrmAiry}.  
To prove Proposition \ref{prop-tSa-stable}, we first obtain the following proposition, giving pointwise bounds on the temporal Green function. 

\begin{proposition}\label{prop-tGreen-Sa} There exists a positive universal constant $\theta_0$ so that the following holds
$$
|G_S(z,t;x)| \le C_0 e^{- \frac12 \alpha^2 \sqrt \nu t} e^{-\frac{|x-z|^2}{8\sqrt \nu t}} e^{-\frac18 \alpha |x-z|} I_{x,z}
$$
in which $I_{x,z}$ is uniformly bounded in $L^1$ norm with respect to $x$. 
\end{proposition}

\begin{proof}
The proof relies on the appropriate choice of the contour $\Gamma_\alpha$, depending on the location of $x$ and $z$. Since the Green function is symmetric in $x$ and $z$, it suffices to consider the case when $x<z$. 
Using the Cauchy's theory, we decompose the contour of integration as follows: 
\begin{equation}\label{def-GaSa1}\Gamma_\alpha = \Gamma_{\alpha,1} \cup \Gamma_{\alpha,2} \cup \Gamma_{\alpha,3} \end{equation}
in which $$
\begin{aligned}
\Gamma_{\alpha,1} &:= \Big\{ \lambda =\gamma - \alpha^2 \sqrt\nu- i \alpha k, \qquad \min_{y\in [x,z]} U(y) \le k \le \max_{y\in [x,z]} U(y) \Big\} 
\\
\Gamma_{\alpha,2} &:= \Big\{ \lambda = \gamma - k^2 \sqrt\nu  - \alpha^2 \sqrt\nu  - i \alpha \min_{[x,z]} U + 2 \sqrt\nu i ak, \qquad k\ge 0 \Big\}
\\
\Gamma_{\alpha,3} &:= \Big\{ \lambda = \gamma - k^2 \sqrt\nu  - \alpha^2 \sqrt\nu   - i \alpha \max_{[x,z]} U + 2 \sqrt\nu i ak, \qquad k\le 0 \Big\} 
\end{aligned}$$
for \begin{equation}\label{def-ga-stable}
\gamma: = a^2 \sqrt\nu + \frac 12\alpha^2 \sqrt\nu, \qquad a = \frac{|x-z|}{2\sqrt\nu t} .
\end{equation}
We shall estimate the integral \eqref{def-tGreen} over the above contours. As $\Gamma_{\alpha,2}$ and $\Gamma_{\alpha,3}$ are away from the critical layers, the integral \eqref{def-tGreen} over these contours is estimated in the same way as done in \cite[Section 6.3]{GrN1}, and hence we avoid to repeat the details. It remains to focus on the integral over $\Gamma_{\alpha,1}$. 

\subsubsection*{Away from critical layers: $\gamma \gtrsim 1$.}

In the case when $\gamma \ge \theta_0$, for an arbitrary small and fixed constant $\theta_0>0$, $\lambda$ is away from $- \alpha^2 \sqrt\nu -i\alpha \mathrm{Range}(U)$, and thus again the contour is away from the critical layers $z_c$. The bounds on the Green function are already obtained in \cite[Section 6.3]{GrN1}. 


\subsubsection*{Near critical layers: $\gamma \ll 1$.}
When $\gamma \ll1$, it suffices to bound the integration on $\Gamma_{\alpha,1}$, which is now arbitrarily close to $- \alpha^2 \sqrt\nu - i\alpha \mathrm{Range} (U)$, and thus we need to study the Green function near the critical layers. We recall that 
$\mathcal{G}(x,z)$ is constructed in Proposition \ref{prop-exactGrmAiry}. In particular, there holds the following bound for $x\le z$: 
$$ |\mathcal{G}(x,z)| \le C|\dot x|\delta_{cr} \epsilon^{-1} \langle X \rangle^{-1/4} \langle Z \rangle^{-1/4}  | e^{\frac23 X^{3/2}} e^{-\frac23 Z^{3/2}} | $$
in which $X = \delta_{cr}^{-1}\eta(x)$ and $Z = \delta_{cr}^{-1} \eta(z)$, with the Langer's variable $\eta(z)$ and with the critical layer thickness:
$$\delta_{cr} = (\eps/U'_c)^{1/3}.$$
Here, we recall that $$\begin{aligned}
 Z 
 &= \delta_{cr}^{-1}\eta(z) = \eps^{-1/3}\Big( \frac 32 \int_{z_c}^z \sqrt{U-c} \; dy\Big)^{2/3} 
 \\& = \Big( \frac 32 \int_{z_c}^z \nu^{-1/4}\sqrt{\lambda + \alpha^2 \sqrt\nu + i \alpha U(y)} \; dy\Big)^{2/3} 
 \end{aligned}$$
This in particular yields 
\begin{equation}\label{exp-xzxz}e^{\frac23 X^{3/2}} e^{-\frac23 Z^{3/2}}= e^{-\int_x^z \eps^{-1/2} \sqrt{U-c}} \; dy = e^{-\int_x^z \nu^{-1/4} \sqrt{\lambda + \alpha^2 \sqrt\nu + i\alpha U(y)} \; dy}.\end{equation}
For $a\ge 0$, we note that 
$$ \Re \sqrt{a+ib} \ge \sqrt a, \qquad \Re \sqrt{a+ib} \ge \frac{\sqrt2}{2}\sqrt{|b|}.$$
Hence, for $\lambda \in \Gamma_{\alpha,1}$ with $\lambda = \gamma - \alpha^2 \sqrt\nu - i\alpha \Re c$, we have 
$$\begin{aligned}
| e^{\frac23 X^{3/2}} e^{-\frac23 Z^{3/2}}|  
&=  e^{-\Re \int_x^z \nu^{-1/4} \sqrt{\lambda + \alpha^2 \sqrt\nu + i\alpha U(y)} \; dy} 
\\&\le e^{- \frac 34\nu^{-1/4} \sqrt{\gamma }|x-z|} e^{-\theta_0 \nu^{-1/4} \sqrt \alpha \int_x^z \sqrt{|U-\Re c|}\; dy}
\end{aligned}$$
for some positive $\theta_0$. In addition, since $\Re c $ belongs to the range of $U$ over $[x,z]$, the critical layer $z=z_c$ stays between $x$ and $z$. By view of \eqref{exp-xzxz} and the fact that $\sqrt{\lambda + \alpha^2 \sqrt\nu + i\alpha U(y)}$ takes the positive real part, we have 
$$| e^{\frac23 X^{3/2}} e^{-\frac23 Z^{3/2}}|   \le C e^{-\theta_0 |X|^{3/2} } e^{-\theta_0 |Z|^{3/2}}$$
for all $x,z$ (as long as the critical layer $\Re z_c$ stays in the interval $[x,z]$). This proves 
\begin{equation}\label{bd-Ga-Ga1} |\mathcal{G}(x,z)| \le C|\dot x|\delta_{cr} \epsilon^{-1} e^{- \frac 34\nu^{-1/4} \sqrt{\gamma }|x-z|}  e^{-\theta_0 | X|^{3/2} } e^{-\theta_0 |Z|^{3/2}}
\end{equation}
for all $\lambda \in \Gamma_{\alpha,1}$ and for $c = -\lambda/i\alpha$. We stress that the Langer's variables $X,Z$ not only depend on $x,z$ but also on $c$. Here, $\gamma$ is defined as in \eqref{def-ga-stable}.  
 
Recall that $\delta_{cr} = (\eps/U'_c)^{1/3}$ and $\lambda = \gamma - \alpha^2 \sqrt\nu - i\alpha \Re c$. We then compute 
$$\begin{aligned}
 \Big| \int_{\Gamma_{\alpha,1}} e^{\lambda t}   \mathcal{G}(x,z) \; \frac{d\lambda}{i\alpha} \Big|
 &\le  Ce^{\gamma t } e^{- \alpha^2 \sqrt\nu t}  e^{- \frac34\nu^{-1/4} \sqrt{\gamma }|x-z|}  
 \\&\quad \times  \epsilon^{-2/3}
\int_{U[x,z]} |U_c'|^{-1/3}  e^{-\theta_0 |X|^{3/2} } e^{-\theta_0 |Z|^{3/2}} \;  |\dot x|d\Re c
\\ &\le  Ce^{\gamma t } e^{- \alpha^2 \sqrt\nu t}  e^{- \frac34\nu^{-1/4} \sqrt{\gamma }|x-z|}  I_{x,z}   \end{aligned}
  $$
in which we have introduced 
\begin{equation}\label{def-Ixzloc} I_{x,z}: = \epsilon^{-2/3} 
\int_{U[x,z]} |U_c'|^{-1/3} e^{-\theta_0 |X|^{3/2} } e^{-\theta_0 |Z|^{3/2}} \;  |\dot x| d\Re c .\end{equation}
By definition of $\gamma = \frac{|x-z|^2}{4\sqrt \nu t^2} + \frac12 \alpha^2 \sqrt \nu$, it follows that  
$$\begin{aligned}
e^{\gamma t }e^{- \alpha^2 \sqrt \nu t} e^{- \frac12 \nu^{-1/4} \sqrt{\gamma } |x-z|} &\le  e^{- \frac12 \alpha^2 \sqrt\nu t} \end{aligned}$$ 
and 
$$e^{- \frac14 \nu^{-1/4} \sqrt{\gamma } |x-z|} \le e^{-\frac{|x-z|^2}{8\sqrt\nu t}} e^{-\frac18 \alpha|x-z|}.$$
This proves that 
$$\begin{aligned}
 \Big| \int_{\Gamma_{\alpha,1}} e^{\lambda t}   \mathcal{G}(x,z) \; \frac{d\lambda}{i\alpha} \Big|
 \le  Ce^{- \frac12 \alpha^2 \sqrt\nu t}   e^{-\frac{|x-z|^2}{8\sqrt\nu t}} e^{-\frac18 \alpha|x-z|}  I_{x,z}   \end{aligned}
  $$
It remains to prove that $I_{x,z}$ is uniformly bounded in $L^1$. Indeed, we compute 
$$ \int_0^\infty I_{x,z} \; dx = \epsilon^{-2/3} \int_0^\infty 
\int_{U[x,z]} |U_c'|^{-1/3}  e^{-\theta_0 |X|^{3/2} } e^{-\theta_0 |Z|^{3/2}}\; |\dot x| d\Re c  dx .$$
Let us make the Langer's change of variables $X = \delta_{cr}^{-1}\eta(x)$ and hence $dX = \delta_{cr}^{-1} \eta'(x)dx$.  Recalling that $\dot x \eta'(x) =1$ and $|\dot x|\le C \langle x \rangle^{1/3}$ (which is bounded by the exponential decay term), we thus obtain 
$$\begin{aligned}
 \int_0^\infty I_{x,z} \; dx &\le \epsilon^{-2/3}  \int_{U[0,z]} \int_0^\infty  |U_c'|^{-1/3} e^{-\theta_0 |X|^{3/2} } e^{-\theta_0 |Z|^{3/2}} \;  \delta_{cr} |\dot x|^2dX d\Re c 
 \\
 &\le C \epsilon^{-1/3}  \int_{U[0,z]}  |U_c'|^{-2/3} e^{-\theta_0 |Z|^{3/2}}  \; d\Re c .  
 \end{aligned}$$
Furthermore, we introduce another change of variable: $c = U(z_c)$ and then the Langer's variable: $Z = \delta_{cr}^{-1}\eta(z_c)$. The above then yields 
\begin{equation}\label{est-Ixzz}\begin{aligned}
 \int_0^\infty I_{x,z} \; dx 
 &\le C \epsilon^{-1/3}  \int_0^z  |U_c'|^{1/3} e^{-\theta_0 |Z|^{3/2}}   \; d \Re z_c 
 \\
 &\le C \epsilon^{-1/3}  \int_0^\infty  |U_c'|^{1/3} e^{-\theta_0 |Z|^{3/2}}   \; \delta_{cr} |\dot z| d Z
 \\
 &\le C \int_0^\infty e^{-\theta_0 |Z|^{3/2}} \langle z\rangle^{1/3} \; d Z
\\&\le C .
 \end{aligned}\end{equation}

To summarize, we have obtained $$\begin{aligned}
 \Big|& \int_{\Gamma_\alpha } e^{\lambda t}   \mathcal{G}(x,z) \; \frac{d\lambda}{i\alpha} \Big|
 \\&\le  C e^{- \frac12 \alpha^2 \sqrt\nu t}   e^{-\frac{|x-z|^2}{8\sqrt\nu t}} e^{-\frac14 \alpha|x-z|} \Big[ I_{x,z} + (\sqrt \nu t)^{-1/2} e^{-\theta_0\frac{|x-z|^2}{\sqrt \nu t}}\Big].   \end{aligned}
  $$
Clearly, $ (\sqrt \nu t)^{-1/2} e^{-\theta_0\frac{|x-z|^2}{\sqrt \nu t}}$ is uniformly bounded in $L^1$ (in $x$ or $z$). This finishes the proof of Proposition \ref{prop-tGreen-Sa}.    
\end{proof}

\subsubsection{Bounds on $\mathcal{S}_\alpha$}

In this section, we prove the first part of Proposition \ref{prop-tSa-stable}, giving bounds on $\mathcal{S}_\alpha$. Recall that 
\begin{equation}\label{conv-SaGS}
\mathcal{S}_\alpha  \omega_\alpha (z) = \int_0^\infty G_S(z,t;x) \omega_\alpha (x) \; dx
\end{equation}
in which $G_S(z,t;x) $ satisfies the bounds obtained in Proposition \ref{prop-tGreen-Sa}. We now study the convolution with the boundary layer data $\omega_\alpha (z)$, satisfying \eqref{assmp-wbl-stable}. The bound on $\mathcal{S}_\alpha$ stated in Proposition \ref{prop-tSa-stable} follows at once from the following lemma and bounds on $G_S(z,t;x)$ from  Proposition \ref{prop-tGreen-Sa}. 

\begin{lemma}\label{lem-Heatconv} Let $H(z,t;x): =  e^{-\frac{|x-z|^2}{M\sqrt \nu t}} I_{x,z} $, for some positive $M$ and some $I_{x,z}$, which is uniformly bounded in $L^1_x$. For any $\beta,\gamma_1>0$ and $ p\ge 0$, there is a positive constant $C_0$ so that 
$$\Big \| \int_0^\infty H(\cdot,t;x) \omega_\alpha(x)\; dx \Big\|_{ \beta, \gamma, p} \le C_0 e^{\gamma_1\nu^{1/4} t } \| \omega_\alpha\|_{ \beta, \gamma, p}.$$
\end{lemma}
\begin{proof}
In the case when $|x-z|\ge M\beta  \sqrt \nu t$, it is clear that 
$$e^{-\frac{|x-z|^2}{M\sqrt \nu t}} e^{-\beta  |x|} \le e^{-\beta  |z|} e^{-|x-z| \Big( \frac{|x-z|}{M\sqrt \nu t} - \beta \Big)} \le e^{-\beta  |z|}.$$
Whereas, for $|x-z| \le M \beta \sqrt \nu t$, we note that 
$$ e^{- M \beta^2 \sqrt \nu t} e^{-\beta  |x|} \le e^{-\beta |x-z|} e^{-\beta  |x|} \le e^{-\beta  |z|}.$$
That is, the exponential decay $e^{-\beta z}$ is recovered at an expense of a slowly growing term in time: $e^{M \beta^2\sqrt\nu t}$. Precisely, this proves 
\begin{equation}\label{exp-beta13}e^{-\frac{|x-z|^2}{M\sqrt \nu t}} e^{-\beta x} \le e^{M \beta^2\sqrt \nu t} e^{-\beta |z|}, \qquad \forall x,z\in \RR.\end{equation}

It remains to study the integral 
\begin{equation}\label{conv-heatbl1-stable}
\begin{aligned}
\int_0^\infty e^{-\frac{|x-z|^2}{M\sqrt \nu t}} 
 I_{x,z}  \Bigl( 1 + \sum_{q=1}^p \delta^{-q} \phi_{P-1+q} (\delta^{-1} x)  \Bigr) \; dx.
\end{aligned}\end{equation}
First, without the boundary layer behavior, the integral is bounded thanks to the integrability assumption on $I_{x,z}$: 
\begin{equation}\label{int-IL11}
\begin{aligned}
 \int_0^\infty I_{x,z} \; dx \le  C_0 
 .\end{aligned}\end{equation}
Next, we treat the boundary layer term. Using \eqref{int-IL11} and the fact that $\phi_{P-1+q}(\delta^{-1}x)$ is decreasing in $x$, we have 
$$
\begin{aligned}
 \int_{z/2}^\infty &   e^{-\frac{|x-z|^2}{M\sqrt \nu t}} I_{x,z}
  \delta^{-q} \phi_{P-1+q} (\delta^{-1} x) \; dx 
  \\&\le C_0   \delta^{-q} \phi_{P-1+q} (\delta^{-1} z)  \int_{z/2}^\infty I_{x,z} \; dx
  \\&\le C_0   \delta^{-q} \phi_{P-1+q} (\delta^{-1} z) .  \end{aligned}$$
Whereas on $x\in (0,\frac z2)$, we have $|x-z|\ge \frac z2 $ and $\phi_{P-1+q} \le 1$. Hence, again using \eqref{int-IL11}, we have   
\begin{equation}\label{large-time-del}
\begin{aligned}
 \int_0^{z/2} &e^{-\frac{|x-z|^2}{M\sqrt \nu t}}  I_{x,z}
  \delta^{-q} \phi_{P-1+q} (\delta^{-1} x) \; dx 
\le C_0 e^{-\frac{|z|^2}{8M\sqrt \nu t}}  \delta^{-q} .
  \end{aligned}\end{equation}
which is  bounded by $C_0  \delta^{-q}  \phi_{P-1+q} (\delta^{-1} z)$, provided that $|z| \gtrsim \nu^{3/8} t$. In case when $|z|\ll \nu^{3/8} t$, we use the allowable exponential growth in time $e^{\gamma_1 \nu^{1/4} t}$, yielding 
$$ e^{ - \gamma_1 \nu^{1/4} t} \le e^{-\gamma_1 z/\nu^{1/8}}$$
which is again bounded by the boundary layer weight function $C_0 \phi_{P-1+q} (\delta^{-1} z)$, upon recalling that the boundary layer thickness is of order $\delta = \gamma \nu^{1/8}$.  
\end{proof}

\begin{remark}\label{rem-expnu}
We observe that if the boundary layer thickness $\delta = \gamma \nu^{1/4}$ is replaced by $\delta(t) = \gamma\nu^{1/4}\sqrt{1+t}$, the right-hand side of \eqref{large-time-del} is  bounded by $C_0  \delta^{-q}  \phi_{P-1+q} (\delta^{-1} z)$, for all $z\ge 0$. That is, we do not lose a growth in time of order $e^{\gamma_1\nu^{1/4}t}$ in Lemma \ref{lem-Heatconv}: precisely, there holds
$$\Big \| \int_0^\infty H(\cdot,t;x) \omega_\alpha(x)\; dx \Big\|_{ \beta, \gamma, p} \le C_0 e^{M\beta^2 \sqrt \nu t } \| \omega_\alpha\|_{ \beta, \gamma, p}$$ 
with the boundary layer norm $\| \cdot \|_{ \beta, \gamma, p}$ having the time-dependent boundary layer thickness $\delta(t) = \gamma \nu^{1/4}\sqrt{1+t}$. 
\end{remark}

\subsubsection{Derivative bounds on $\mathcal{S}_\alpha$}\label{sec-dervSa-stable}
Having obtained the bounds on $\mathcal{S}_\alpha$, we now derive the derivative bounds. First, note that the derivative $\partial_z\mathcal{S}_\alpha (t)[\omega_\alpha]$ solves 
$$ (\partial_t + i\alpha U -\sqrt \nu\Delta_\alpha) \partial_z \mathcal{S}_\alpha (t) [\omega_\alpha]  = - i\alpha U'(z)  \mathcal{S}_\alpha (t) [\omega_\alpha] $$
with initial data $ \partial_z \omega_\alpha$. 
The Duhamel's formula yields 
$$
\begin{aligned}
 \partial_z \mathcal{S}_\alpha (t) [\omega_\alpha]
 &=  \mathcal{S}_\alpha(t) [ \partial_z \omega_\alpha] - i\alpha \int_0^t \mathcal{S}_\alpha (t-s) [ U' \mathcal{S}_\alpha (s) [\omega_\alpha] ]  \; ds.
 \end{aligned}$$
Now, applying the bounds on $\mathcal{S}_\alpha$ obtained from Proposition \ref{prop-tSa-stable}, for $\tau >0$, we obtain at once 
$$\begin{aligned}
\| \mathcal{S}_\alpha(t)[\partial_z \omega_\alpha] \|_{ \beta, \gamma, p+1} \le C_0 e^{\gamma_1\nu^{1/4} t }  e^{- \frac14 \alpha^2 \sqrt\nu t}\| \partial_z \omega_\alpha \|_{ \beta, \gamma, p+1} \end{aligned}.$$
We apply again Proposition \ref{prop-tSa-stable}, for some positive $\tau_1 <\tau$, on $\mathcal{S}_\alpha (t-s)$. Thanks to the embedding estimate: $\| \omega \|_{ \beta, \gamma, p+1} \le \|\omega \|_{ \beta, \gamma, p}$,  we get 
$$ 
\begin{aligned}
\alpha \Big\| \int_0^t &\mathcal{S}_\alpha (t-s) [ U' \mathcal{S}_\alpha (s) [\omega_\alpha] ]  \; ds \Big \|_{ \beta, \gamma, p+1}
\\ &\le \int_0^t  C_0 e^{\gamma_1\nu^{1/4}  (t-s) } e^{- \frac14 \alpha^2 \sqrt \nu (t-s)}  \| \alpha U'  \mathcal{S}_\alpha (s) [\omega_\alpha] \|_{ \beta, \gamma, p+1} \; ds
\\ &\le \int_0^t  C_0 e^{\gamma_1\nu^{1/4}  (t-s) } e^{- \frac14 \alpha^2 \sqrt \nu (t-s)} C_\tau e^{\gamma_1\nu^{1/4}  s } e^{- \frac14 \alpha^2 \sqrt \nu s}  \|\alpha  \omega_\alpha \|_{ \beta, \gamma, p}   \; ds
\\ &\le  C_0 t e^{\gamma_1\nu^{1/4}  t } e^{- \frac14 \alpha^2 \sqrt \nu t} \| \alpha \omega_\alpha \|_{ \beta, \gamma, p} .
  \end{aligned} $$
This proves 
$$
\begin{aligned}
\| &\partial_z \mathcal{S}_\alpha(t)[\omega_\alpha] \|_{ \beta, \gamma, p+1} 
\le C_0 e^{\gamma_1\nu^{1/4}  t } e^{- \frac14 \alpha^2 \sqrt \nu t} \Big[  \| \partial_z \omega_\alpha \|_{ \beta, \gamma, p+1} + t\| \alpha \omega_\alpha \|_{ \beta, \gamma, p}  \Big].
\end{aligned}
$$
By induction, the higher order derivative estimates and thus Proposition \ref{prop-tSa-stable} follow.

\subsection{Semigroup remainders for vorticity}
In this section, we study the remainder of semigroup for vorticity defined as in \eqref{def-SRa}. For convenience, we write 
\begin{equation}\label{def-Rsemigroup}
\mathcal{R}_{\alpha} \omega_\alpha (z)= \int_0^\infty G_R(z,t;x) \omega_\alpha (x)\; dx 
\end{equation}
in which the residual Green function $G_R(z,t;x)$ is defined by 
\begin{equation}\label{def-GRztx}G_R(z,t;x) : =  \frac{1}{2\pi i}  \int_{\Gamma_{\alpha}}  e^{\lambda t}  \mathcal{R}_G(x,z) \; \frac{d \lambda}{i\alpha}  \end{equation}
with $ \mathcal{R}_G(x,z): = \Delta_\alpha G_{\alpha,c}(x,z) -  \mathcal{G}(x,z) $. 
The contour of integration $\Gamma_\alpha$ is chosen so that it remains on the right of the complex strip $ -(\alpha^2+ k^2) \sqrt \nu + i \alpha \mathrm{Range}(U)$, $k\ge0$. The choice of $\Gamma_\alpha$ depends not only on $\alpha$, but also on the relative position of $x,z,$ and $t$.

\subsubsection{Mid and high spatial frequency: $\alpha \gtrsim 1$}
Throughout this section, we assume that the spatial frequency $\alpha$ is bounded away from zero. We obtain the following proposition. 

\begin{proposition}[Mid and high frequency]\label{prop-tRa-largea} Define the semigroup $\mathcal{R}_\alpha$ as in \eqref{def-Rsemigroup}. Assume that $\alpha \gtrsim 1$. Then, for $0<\beta\le \frac \alpha2$, there are positive constants $C_0,\theta_0$ so that 
$$\| \mathcal{R}_\alpha  \omega_\alpha\|_{ \beta, \gamma, 1} \le C_0 e^{\gamma_1 \nu^{1/4} t} e^{- \frac12 \alpha^2 \sqrt\nu t}  \| \omega_\alpha\|_{ \beta, \gamma, 1}$$
and more generally for all $p\ge 0$,  
$$\| \mathcal{R}_\alpha  \omega_\alpha\|_{ \beta, \gamma, p} \le C_0e^{\gamma_1 \nu^{1/4} t} e^{- \frac12 \alpha^2 \sqrt\nu t}   \| \omega_\alpha\|_{ \beta, \gamma, p} \Big( \alpha^{-1} + \chi_{\{\alpha \delta \ll1\}}\delta^{1-p}\Big).$$
\end{proposition}

We first prove the following simple lemma. 
\begin{lemma}\label{lem-Lapconv} Let $\eta_0>0$, $0< \beta\le\frac{\eta_0}{2}$, and let $P(z,t;x): =  \eta_\nu e^{-\eta_\nu |x-z|} $ for some $\eta_\nu \ge \eta_0$. Then, there is a positive constant $C_0$, independent of $\delta, \eta_\nu$, so that 
$$\Big \| \int_0^\infty P(\cdot,t;x) \omega_\alpha(x)\; dx \Big\|_{ \beta, \gamma, 1} \le C_0\eta_\nu \| \omega_\alpha\|_{ \beta, \gamma, 1},$$
and more generally for all $p\ge 0$,  
$$\Big \| \int_0^\infty P(\cdot,t;x) \omega_\alpha(x)\; dx \Big\|_{ \beta, \gamma, p} \le C_0\| \omega_\alpha\|_{ \beta, \gamma, p} \Big( 1 + \chi_{\{\eta_\nu \delta \ll1\}}\delta^{1-p} \eta_\nu\Big).$$
 \end{lemma}
\begin{remark}
We note that there is no time growth $e^{\sqrt \nu t}$ in the above semigroup bound, since the spatially localized behavior $e^{-\beta z}$ is obtained through the Green kernel of the Laplacian $\alpha^{-1} e^{-\alpha |x-z|}$ (instead of the heat kernel $(\sqrt\nu t)^{-1/2} e^{-\frac{|x-z|^2}{4\sqrt \nu t}}$ as the case for $\mathcal{S}_\alpha$). 
\end{remark}

\begin{proof}[Proof of Lemma \ref{lem-Lapconv}] Without loss of generality, we assume $  \| \omega_\alpha \|_{ \beta, \gamma, p} =1$. By definition, we study the convolution 
\begin{equation}\label{est-P111}
\begin{aligned}
&\int_0^\infty  P(z,t;x) |\omega_\alpha(x)|\; d  x 
\\&\le  \int_0^\infty   \eta_\nu e^{-\eta_\nu |x-z|} e^{- \beta x} \Bigl( 1 +  \sum_{q=1}^p\delta^{-q} \phi_{P-1+q} (\delta^{-1} x)  \Bigr) \; d  x.
\end{aligned}\end{equation}
Using  $\frac{\eta_0} 2\ge \beta$ and the triangle inequality $|x|\ge |z| - |x-z|$, we obtain 
$$e^{-\frac12 \eta_0 |x-z|} e^{- \beta   x} \le e^{-\beta|z|} $$
yielding the exponential decay in the boundary layer norm. We shall now estimate each term in \eqref{est-P111}. Since $ \eta_\nu e^{-\eta_\nu |x-z|}$ is bounded in $L^1_x$, the term without the boundary layer weight is bounded. As for the boundary layer term, we obtain 
$$
\begin{aligned}
\int_0^\infty & \eta_\nu e^{-\frac12\eta_\nu |x-z|}  \delta^{-q} \phi_{P-1+q} (\delta^{-1} x)   \; d  x
\\&\le C_0 \eta_\nu \int_0^\infty   \delta^{-q} \phi_{P-1+q} (\delta^{-1} x)  \; d  x 
\\&\le C_0 \eta_\nu \delta^{1-q}. 
\end{aligned}$$ 
This proves the lemma when $\eta_\nu \delta \ll1$. 

In the case when $\eta_\nu \delta \gtrsim 1$, the boundary layer behavior coming from the Green function has smaller thickness, which we shall now use. Hence, using $\phi_{P-1+q} \le 1$ and its decreasing property, we compute 
\begin{equation}\label{fast-mus}
\begin{aligned}
 \int_0^\infty  \eta_\nu & e^{-\frac12 \eta_\nu |x-z|} \delta^{-q} \phi_{P-1+q} (\delta^{-1} x) \; d  x
\\
&\le C_0 \delta^{-q} \Big[ e^{-\frac18 \eta_\nu |z|}  \int_0^{z/2} \eta_\nu e^{-\frac14 \eta_\nu |x-z|}\; d  x 
\\&\quad +   \phi_{P-1+q} (\delta^{-1} z) \int_{z/2}^\infty \eta_\nu e^{-\frac12\eta_\nu |x-z|} \; d  x\Big]
\\
&\le C_0\delta^{-q} \Big[ e^{-\frac18 \eta_\nu |z|} +   \phi_{P-1+q} (\delta^{-1} z)\Big]
\\
&\le C_0 \delta^{-q}\phi_{P-1+q} (\delta^{-1} z)
\end{aligned}\end{equation}
in which the last inequality used the fact that $\eta_\nu \delta \gtrsim 1$. 
\end{proof}

Next, we prove the following pointwise bounds:

\begin{lemma}\label{lem-tempGRztx} Define the temporal Green function $G_R(z,t;x)$ as in \eqref{def-GRztx}. Assume that $\alpha \gtrsim 1$. There holds 
$$ |G_R(z,t;x)| \le C
 e^{-\frac12 \alpha^2 \sqrt \nu t} e^{-\theta_0\alpha |x-z|}  .
$$ 
\end{lemma}
\begin{proof}
By construction of $G_R(z,t;x)$, we have 
\begin{equation}\label{def-reGR} G_R(z,t;x)= 
 \frac{1}{2\pi i}  \int_0^\infty  \int_{\Gamma_{\alpha}}  e^{\lambda t} \mathcal{G}(y,z)   U''(y) G_{\alpha,c} (x,y)  \; \frac{d\lambda dy}{i\alpha} \end{equation}
in which $c = -\lambda / i\alpha$. We note that when $\lambda$ is away from $-\alpha^2 \sqrt \nu - i\alpha \mathrm{Range(U)}$, the pointwise estimates on $G_R(z,t;x)$ are done exactly as in the case away from critical layers (\cite[Section 6.4]{GrN1}). In order to obtain sharp bounds on the semigroup, it is obliged to move the contour of integration as close to the imaginary axis as possible. 
As done earlier for $\mathcal{S}_\alpha$, it suffices to focus on the integral over the following contour: 
\begin{equation}\label{redef-Ga111}\Gamma_{\alpha,1} = \Big\{ \lambda = \gamma  - \alpha^2 \sqrt\nu - i\alpha k: \qquad  k \in \mathrm{Range}(U)  \Big\} \end{equation}
in which $\gamma$ 
is defined by 
$$ \gamma:  = \gamma_1 \nu^{1/4} + a^2 \sqrt\nu + \frac 12 \alpha^2 \sqrt\nu ,\qquad a: = \frac{|y-z|}{2\sqrt \nu t}$$
for positive $\gamma_1>\gamma_0$. Again, the case when $\gamma\gtrsim 1$ is treated exactly as in the previous unstable case (\cite[Section 6.4]{GrN1}). In what follows, we assume that $\gamma \ll1$.

In view of the spectral assumption, Assumption \ref{assump-Evans}, $c=-\lambda/i\alpha$ satisfies 
$$ |D(\alpha,c) |\gtrsim 1$$
for $\alpha \gtrsim 1$. In addition, $\Im c\ll1$ and $\Im c \ge \frac12\alpha \sqrt\nu \gtrsim |\epsilon|$, since $\alpha \gtrsim 1$. This proves that the condition    
$$ \epsilon^{1/8} |\log \Im c | \ll1$$
used in Theorem \ref{theo-GreenOS-stable} holds. 

By putting these together, on $\Gamma_{\alpha,1}$, the bounds on the Green function from Theorem \ref{theo-GreenOS-stable}
now read 
\begin{equation}\label{decomp-Gsf}G_{\alpha,c}(x,y) = G_{\alpha,c,s}(x,y) + G_{\alpha,c,f}(x,y)\end{equation}
in which 
$$
\begin{aligned}
 |G_{\alpha,c,s}(x,y)|  
 &\le 
  C_0\alpha^{-1}e^{-\theta_0\alpha |x-y|}  
 \\
 |G_{\alpha,c,f}(x,y)|  
 &\le 
  C_0\delta_{cr} e^{-\frac34  \nu^{-1/4} \sqrt \gamma |x-y|} e^{-\theta_0 \sqrt{\max\{|X|, |Y|\} } |X-Y|} 
\end{aligned}
$$
in which $X,Y$ denote the usual Langer's variable: $X = \delta_{cr}^{-1}\eta(x)$, recalling the critical layer thickness $\delta_{cr} = (\epsilon/U_c')^{1/3}$.

We are now ready to estimate 
$$
\begin{aligned}
I_{\alpha,s}: &= \frac{1}{2\pi i}  \int_0^\infty  \int_{\Gamma_{\alpha,1}}  e^{\lambda t} \mathcal{G}(y,z) U''(y) G_{\alpha,c,s} (x,y)  \; \frac{d\lambda dy}{i\alpha} 
\\
I_{\alpha,f}: &= \frac{1}{2\pi i}  \int_0^\infty  \int_{\Gamma_{\alpha,1}}  e^{\lambda t} \mathcal{G}(y,z) U''(y) G_{\alpha,c,f} (x,y)  \; \frac{d\lambda dy}{i\alpha} 
\end{aligned}
$$
with $\Gamma_{\alpha,1}$ defined as in \eqref{redef-Ga111}. In the convolution against the slow behavior $G_{\alpha,c,s}(x,y)$, to leading order, the $\lambda$-dependence is precisely due to the Green function $\mathcal{G}(y,z)$ and hence we can further move the contour of integration so that 
$$
\begin{aligned}
\Gamma'_{\alpha,1} &:= \Big\{ \lambda =\gamma - \alpha^2 \sqrt\nu- i \alpha k, \qquad k\in U[y,z] \Big\} 
\end{aligned}$$
and $\Gamma_{\alpha,2}', \Gamma_{\alpha,3}'$ being defined similarly as done in \eqref{def-GaSa1}. Hence, for slow behavior of the Green function, it suffices to estimate
$$I'_{\alpha,s}: = \frac{1}{2\pi i}  \int_0^\infty  \int_{\Gamma'_{\alpha,1}}  e^{\lambda t} \mathcal{G}(y,z) U''(y) G_{\alpha,c,s} (x,y)  \; \frac{d\lambda dy}{i\alpha} .$$
As in \eqref{bd-Ga-Ga1}, on $\Gamma'_{\alpha,1}$, we have 
$$
 |\mathcal{G}(y,z)| \le C|\dot y| \epsilon^{-2/3}|U'_c|^{-1/3} e^{- \frac 34\nu^{-1/4} \sqrt{\gamma }|y-z|}  e^{-\theta_0 |Y|^{3/2}} e^{-\theta_0 |Z|^{3/2}}.$$
Whereas on $\Gamma_{\alpha,1}$, there holds 
$$
 |\mathcal{G}(y,z)| \le C|\dot y| \epsilon^{-2/3}|U'_c|^{-1/3} e^{- \frac 34\nu^{-1/4} \sqrt{\gamma }|y-z|}  \langle Z \rangle^{-1/2}e^{-\theta_0 \sqrt{\max\{ |Y|, |Z|\}} |Y-Z|}
$$
upon noting that on $\Gamma_{\alpha,1}$, $\Re c$ can now take values in the entire range of $U$ (other than that over $[y,z]$ as in the case of $\Gamma_{\alpha,1}'$). Here, we note that 
$$e^{-\frac14 \nu^{-1/4} \sqrt \gamma |y-z|} \le e^{-\frac{|y-z|^2}{8 \sqrt \nu t}} e^{-\frac18 \alpha |y-z|} $$
and $$ e^{\gamma t} e^{-\frac12 \alpha^2 \sqrt \nu t} e^{-\frac12 \nu^{-1/4} \sqrt \gamma |y-z|} \le e^{\gamma_1 \nu^{1/4} t}. $$

We now estimate $I'_{\alpha,s}$. Combining these with the above Green function bounds and recalling that $|U''(y)| \le C_0 e^{-\eta_0 y}$, we obtain at once 
$$\begin{aligned}
I_{\alpha,s}' &\le C \alpha^{-1}   
e^{-\frac12 \alpha^2 \sqrt \nu t} e^{-\theta_0\alpha |x-z|} e^{\gamma_1 \nu^{1/4} t}
\\&\quad \times \int_0^\infty  \int_{U[y,z]} \epsilon^{-2/3}|U'_c|^{-1/3}   e^{-\theta_0 |Y|^{3/2}} e^{-\theta_0 |Z|^{3/2}} \; dk dy.
\end{aligned}$$
The integral above is already estimated in \eqref{def-Ixzloc}-\eqref{est-Ixzz}, yielding 
$$\begin{aligned}
I_{\alpha,s}' &\le C \alpha^{-1}   
e^{-\frac12 \alpha^2 \sqrt \nu t} e^{-\theta_0\alpha |x-z|} e^{\gamma_1 \nu^{1/4} t}.\end{aligned}$$
Next, turning to the fast behavior of the Green function, we estimate 
$$\begin{aligned}
I_{\alpha,f} &\le   
e^{-\frac12 \alpha^2 \sqrt \nu t} e^{-\theta_0\alpha |x-z|} e^{\gamma_1 \nu^{1/4} t} \int_0^\infty  \int_{\mathrm{Range (U)}} \epsilon^{-1/3}|U'_c|^{-2/3}  
 \langle Z \rangle^{-1/2} 
 \\&\quad \times  e^{-\theta_0 \sqrt{\max\{ |Y|, |Z|\}} |Y-Z|}  e^{-\theta_0 \sqrt{\max\{|X|, |Y|\} } |X-Y|} e^{- \eta_0 y}\; dk dy .
\end{aligned}$$
Changing variable $c = U(z_c)$ and then the Langer's variable $Z = \delta^{-1}_{cr} \eta(z)$, we obtain 
$$\begin{aligned}
&
\int_0^\infty  \int_{\mathrm{Range (U)}} \epsilon^{-1/3}|U'_c|^{-2/3} 
  \langle Z \rangle^{-1/2}e^{-\theta_0 \sqrt{\max\{ |Y|, |Z|\}} |Y-Z|}   \; e^{-\eta_0 y}\; dk dy
\\&=\int_0^\infty  \int_0^\infty \epsilon^{-1/3}|U'_c|^{1/3} 
 \langle Z \rangle^{-1/2} e^{-\theta_0 \sqrt{\max\{ |Y|, |Z|\}} |Y-Z|}   \; e^{-\eta_0 y} d\Re z_c dy
\\&\le C 
\int_0^\infty  \int_0^\infty \epsilon^{-1/3}|U'_c|^{1/3} 
  \langle Z \rangle^{-1/2} e^{-\theta_0 \sqrt{\max\{ |Y|, |Z|\}} |Y-Z|}   \; e^{-\eta_0 y} \delta_{cr} |\dot z| dZ dy
\\&\le C 
\int_0^\infty   \; e^{-\eta_0 y} dy \le C.
\end{aligned}$$
This proves 
$$\begin{aligned}
I_{\alpha,f} &\le   
C e^{-\frac12 \alpha^2 \sqrt \nu t} e^{-\theta_0\alpha |x-z|} e^{\gamma_1 \nu^{1/4} t} .
\end{aligned}$$
This yields the lemma. 
\end{proof}

\begin{proof}[Proof of Proposition \ref{prop-tRa-largea}] The proposition now follows directly from the bounds on the Green function obtained from Lemma \ref{lem-tempGRztx}, and Lemma \ref{lem-Lapconv}, applying for the Green kernel $e^{-\theta_0\alpha |x-z|} $, with $\alpha \gtrsim 1$. 
\end{proof}

\subsubsection{Small spatial frequency: $\alpha \ll1$}
In this section, we study the case when $\alpha \ll1$. In this case, there are unstable eigenvalues $\lambda_\nu$ as spelled out in Section \ref{sec-spectral}, and as a consequence, the contour of integration has to be chosen more carefully to avoid the unstable eigenvalues.

\begin{proposition}[Small frequency] \label{prop-tRa-smalla} Define the semigroup $\mathcal{R}_\alpha$ as in \eqref{def-Rsemigroup}. Assume that $\alpha \ll 1$. Then, for any positive $\beta$, there is a constant $C_0$ so that 
$$\| \mathcal{R}_\alpha  \omega_\alpha\|_{L^\infty_{\eta_0}} \le C_0 \alpha^2\nu^{-1/4}e^{\gamma_1 \nu^{1/4}t}  \| \omega_\alpha\|_{ \beta, \gamma, 1}$$
for any fixed number $\gamma_1 > \gamma_0$, in which $\eta_0$ is as in the bound $|U'(z)|\le C_0 e^{-\eta_0 z}$. By definition, there holds the embedding: $L^\infty_{\eta_0} \subset \mathcal{B}^{\eta_0, \gamma,p}$, for $p\ge 0$. 
\end{proposition}

%
%

We shall prove the following lemma. 

\begin{lemma} Define the temporal Green function $G_R(z,t;x)$ as in \eqref{def-GRztx}. Assume that $\alpha \ll1$. There holds 
$$ |G_R(z,t;x)| \le C\alpha^2\nu^{-1/4}
 e^{\gamma_1 \nu^{1/4}t} e^{-\eta_0 z} \Big[ 1+ \alpha e^{-\theta_0\frac{|x-z|}{\nu^{1/8}} } \Big]$$ 
for any fixed number $\gamma_1 > \gamma_0$. 
\end{lemma}
\begin{proof} The proof follows the similar lines to that of Lemma \ref{lem-tempGRztx}. Precisely, for each $x,y,z$, to avoid the unstable eigenvalues, we now take the contours of integration near the critical layers to be the following:  
\begin{equation}\label{redef-Ga1-sa}
\begin{aligned}
\Gamma'_{\alpha,1} &:= \Big\{ \lambda =\gamma - i \alpha k, \qquad k\in U[y,z] \Big\} 
\\\Gamma_{\alpha,1} &: = \Big\{\lambda =  \gamma  - i\alpha k: \qquad  k \in \mathrm{Range}(U)  \Big\} 
\end{aligned}\end{equation}
used for the slow and fast behavior of the Green function, respectively, in which $\gamma$ 
is defined by 
$$ \gamma:  = \gamma_1 \nu^{1/4} + a^2 \sqrt\nu ,\qquad a: = \frac{|y-z|}{2\sqrt \nu t}$$
for arbitrary $\gamma_1>\gamma_0$. We observe that $\nu^{-1/4}\Re \lambda = \gamma_1 + a^2 \nu^{1/4} > \gamma_0$.  Hence, the addition of $\gamma_1 \nu^{1/4}$ allows the contour of integration to avoid the unstable eigenvalues (see also Assumption \ref{assump-Evans}), yielding 
$$ |D(\alpha,c) |\gtrsim 1$$
for $c = -\lambda / i\alpha$ and for $\lambda \in \Gamma_{\alpha,1}$. We note that in this case there holds $\Im c \ge \gamma_1 \nu^{1/4}/\alpha \gg |\epsilon|$ and hence the condition $ \epsilon^{1/8} |\log \Im c | \ll1$, used in Theorem \ref{theo-GreenOS-stable}, remains valid. Thus, as in \eqref{decomp-Gsf}, by using Theorem \ref{theo-GreenOS-stable} for small $\alpha$, the Green function satisfies 
$$
\begin{aligned}
 |G_{\alpha,c,s}(x,y)|  
 &\le 
  C_0 \Big( 1 + \alpha |U(x)-c|^{-1} \Big) e^{-\theta_0\alpha |x-y|}  
 \\
 |G_{\alpha,c,f}(x,y)|  
 &\le 
  C_0\delta_{cr} e^{-\frac34  \nu^{-1/4} \sqrt \gamma |x-y|} e^{-\theta_0 \sqrt{\max\{|X|, |Y|\} } |X-Y|} 
\end{aligned}
$$
for the slow and fast behavior of the Green function, respectively. In this case, there hold
$$ e^{\gamma t} e^{-\frac12 \nu^{-1/4} \sqrt \gamma |y-z|} \le e^{\gamma_1 \nu^{1/4}t}$$
and $$e^{-\frac14 \nu^{-1/4} \sqrt \gamma |y-z|} \le e^{-\frac{|y-z|^2}{8 \sqrt \nu t}} e^{-\theta_0\frac{|y-z|}{\nu^{1/8}} } ,$$
upon noting the gain of the localized new term $e^{-\theta_0\frac{|y-z|}{\nu^{1/8}} } $,  thanks to the addition of $\gamma_1\nu^{1/4}$ into the definition of $\gamma$. Similarly, we also gain the term $e^{-\theta_0\frac{|x-y|}{\nu^{1/8}} } $ in the fast behavior of the Green function $G_{\alpha,f}(x,y)$.  

We now estimate 
$$
\begin{aligned}
I'_{\alpha,s}: &= \frac{1}{2\pi i}  \int_0^\infty  \int_{\Gamma'_{\alpha,1}}  e^{\lambda t} \mathcal{G}(y,z) U''(y) G_{\alpha,c,s} (x,y)  \; \frac{d\lambda dy}{i\alpha} 
\end{aligned}
$$
Combining these with the above Green function bounds and recalling that $|U''(y)| \le C_0 e^{-\eta_0 y}$, we obtain at once 
$$\begin{aligned}
I_{\alpha,s}' 
&\le C_0 e^{\gamma_1 \nu^{1/4}t}  
  \int_0^\infty  \int_{U[y,z]} \epsilon^{-2/3}|U'_c|^{-1/3} 
e^{-\frac{|y-z|^2}{8 \sqrt \nu t}} e^{-\theta_0\frac{|y-z|}{\nu^{1/8}} }  
  \\&\quad \times   e^{-\theta_0 |Y|^{3/2}} e^{-\theta_0 |Z|^{3/2}}   (1+ \alpha |U(x)-c|^{-1}) e^{-\theta_0\alpha |x-y|}  e^{-\eta_0 y}\; d\Re c dy .
\end{aligned}$$
Using $c = U(z_c)$, we obtain 
$$\begin{aligned}
I_{\alpha,s}' 
&\le C_0 e^{\gamma_1 \nu^{1/4}t}  
  \int_0^\infty  \int_y^z \epsilon^{-2/3}|U'_c|^{2/3} 
e^{-\frac{|y-z|^2}{8 \sqrt \nu t}} e^{-\theta_0\frac{|y-z|}{\nu^{1/8}} }  
  \\&\quad \times   e^{-\theta_0 |Y|^{3/2}} e^{-\theta_0 |Z|^{3/2}}    (1+ \alpha |U(x)-c|^{-1}) e^{-\theta_0\alpha |x-y|}  e^{-\eta_0 y}\; d\Re z_c dy .
\end{aligned}$$
Clearly, $e^{-\theta_0\frac{|y-z|}{\nu^{1/8}} }   e^{-\eta_0 y} \le e^{-\eta_0 z}$. Now using the fact that $\alpha \Im c \gtrsim \nu^{1/4}$, we bound 
\begin{equation}\label{loss-Gnu} \alpha |U(x)-c|^{-1} \lesssim \alpha^2 \nu^{-1/4}.\end{equation}
In addition, the integral above is already estimated in \eqref{def-Ixzloc}-\eqref{est-Ixzz}, yielding 
$$\begin{aligned}
I_{\alpha,s}' &\le C_0 (1+\alpha^2 \nu^{-1/4})e^{\gamma_1 \nu^{1/4}t}   e^{-\eta_0 z}  .\end{aligned}$$

The estimate on $I_{\alpha,f}$, the integral involving the the fast behavior of the Green function, follows identically the previous case for $\alpha \gtrsim 1$. This completes the proof of the lemma. \end{proof}

\begin{remark}\label{rem-lossGnu}
The loss of $\nu^{-1/4}$ is precisely due to the singularity of $(U-c)^{-1}$ in the (rough) bound \eqref{loss-Gnu}, which appears precisely in the case when $\alpha \ll1$.  
\end{remark}

\begin{proof}[Proof of Proposition \ref{prop-tRa-smalla}] By definition and using the fact that $G_R(z,t;x)$ is bounded by $C_0\alpha^2\nu^{-1/4}e^{\gamma_1 \nu^{1/4}t} e^{-\eta_0 z} $, we study the convolution 
$$\begin{aligned}
\Big | \int_0^\infty G_R(z,t;x) \omega_\alpha(x)\; dx \Big |
&\le C_0\alpha^2\nu^{-1/4} e^{\gamma_1 \nu^{1/4}t}  e^{-\eta_0 z} \| \omega_\alpha \|_{L^1}
\\
&\le C_0 \alpha^2\nu^{-1/4} e^{\gamma_1 \nu^{1/4}t}  e^{-\eta_0 z} \| \omega_\alpha\|_{ \beta, \gamma, 1}
\end{aligned}$$
which proves the proposition. 
\end{proof}

\subsection{Proof of the semigroup bounds}

The main theorem, Theorem \ref{theo-eLt-stable}, now follows at once from the decomposition 
$$ e^{L_\alpha t} = \mathcal{S}_\alpha + \mathcal{R}_\alpha $$
upon using estimates from Proposition \ref{prop-tSa-stable} on $\mathcal{S}_\alpha$ and from Propositions \ref{prop-tRa-largea} and \ref{prop-tRa-smalla} on $\mathcal{R}_\alpha$. 

~\\
{\em Acknowledgement.} The authors would like to thank Yan Guo for the fruitful discussions and ICERM, Brown University for the hospitality in Spring 2017 during which part of this work was completed. TN's research was supported in part by the NSF under grant DMS-1405728.


\end{document}